\documentclass[12pt,oneside,english]{amsart}
\usepackage[T1]{fontenc}
\usepackage[latin9]{inputenc}
\usepackage{geometry}
\usepackage{amstext}
\usepackage{amsthm}
\usepackage{amssymb}
\usepackage{stmaryrd}
\usepackage{graphicx}
\usepackage{mathtools}
\usepackage{wasysym}
\usepackage{mathrsfs}
\usepackage{bbm}
\usepackage{esint}
\usepackage{color}
\usepackage{subfigure}
\usepackage{url}
\usepackage[percent]{overpic}
\usepackage{braket}
\usepackage{bm}
\usepackage{verbatim}
\usepackage{tikz}
\usetikzlibrary{intersections,calc,arrows.meta}
\input xy
\xyoption{all}

\newcommand{\hackcenter}[1]{
 \xy (0,0)*{#1}; \endxy}

\makeatletter
\numberwithin{equation}{section}
\numberwithin{figure}{section}
\theoremstyle{plain}
\newtheorem{thm}{\protect\theoremname}
\newtheorem*{thm*}{\protect\theoremname}
\newtheorem{prop}[thm]{\protect\propositionname}
\newtheorem*{prop*}{\protect\propositionname}
\newtheorem{lem}[thm]{\protect\lemmaname}
\newtheorem*{lem*}{\protect\lemmaname}
\newtheorem{defn}[thm]{\protect\definitionname}
\newtheorem{cor}[thm]{\protect\corollaryname}
\newtheorem{rmk}[thm]{\protect\remarkname}

\numberwithin{thm}{section}

\newcommand{\term}[1]{{\bf #1}} 

\newcommand{\ii}{\mathbf{i}}
\newcommand{\re}{\mathfrak{Re}}

\newcommand{\unitvec}[1]{\hat{e}_{#1}}
\newcommand{\eps}{\varepsilon}

\newcommand{\ContF}{\mathcal{C}}
\newcommand{\Mob}{\mathfrak{M}}

\newcommand{\const}{\mathrm{const.}}

\newcommand{\oo}{\mathfrak{o}}

\newcommand{\id}{\mathrm{id}}
\newcommand{\idof}[1]{\id_{#1}}
\newcommand{\isom}{\cong}
\newcommand{\dmn}{\mathrm{dim}}
\newcommand{\spn}{\mathrm{span}}
\newcommand{\End}{\mathrm{End}}
\newcommand{\Hom}{\mathrm{Hom}}
\newcommand{\Kern}{\mathrm{Ker}}

\newcommand{\ldual}{\langle}
\newcommand{\rdual}{\rangle}
\newcommand{\tens}{\otimes}

\newcommand{\C}{\mathbb{C}} 
\newcommand{\R}{\mathbb{R}} 
\newcommand{\Z}{\mathbb{Z}} 
\newcommand{\Znn}{\N} 
\newcommand{\Zpos}{\Z_{> 0}} 
\newcommand{\N}{\mathbb{N}} 
\newcommand{\Q}{\mathbb{Q}} 
\newcommand{\bC}{\C} 
\newcommand{\bR}{\R} 
\newcommand{\bZ}{\Z} 
\newcommand{\bN}{\N} 
\newcommand{\bQ}{\Q} 

\newcommand{\dee}{{\partial}}

\newcommand{\UEA}{\mathcal{U}}
\newcommand{\VermaSymbol}{M}
\newcommand{\VermaCH}[2]{\VermaSymbol(#1,#2)}
\newcommand{\VermaFR}[1]{\VermaSymbol_{#1}}

\newcommand{\PBWfilsym}{\mathscr{F}}
\newcommand{\PBWfil}[1]{\PBWfilsym^{#1}}

\renewcommand{\set}[1]{\left\{ #1 \right\}}

\newcommand{\hot}{\mathrm{h.o.t.}}

\newcommand{\voaV}{V}
\newcommand{\voaM}{W}
\newcommand{\modM}{\voaM}
\newcommand{\Mvec}{w}

\newcommand{\insym}{0}
\newcommand{\midsym}{1}
\newcommand{\outsym}{\infty}
\newcommand{\Mvecin}{{\Mvec}_{\insym}}
\newcommand{\Mvecmid}{{\Mvec}_{\midsym}}
\newcommand{\Mvecout}{{\Mvec}_{\outsym}}
\newcommand{\labin}{\lambda_{\insym}}
\newcommand{\labmid}{\lambda_{\midsym}}
\newcommand{\labout}{\lambda_{\outsym}}
\newcommand{\Vvec}{v}
\newcommand{\FRSymbol}{Q}
\newcommand{\voaFR}[1]{\FRSymbol_{#1}}
\newcommand{\hwvsym}{\bar{w}}
\newcommand{\hwvin}{\hwvsym_{\insym}}
\newcommand{\hwvmid}{\hwvsym_{\midsym}}
\newcommand{\hwvout}{\hwvsym_{\outsym}}
\newcommand{\hin}{h_{\insym}}
\newcommand{\hmid}{h_{\midsym}}
\newcommand{\hout}{h_{\outsym}}

\newcommand{\hwvFR}[1]{{\hwvsym}_{#1}}
\newcommand{\hwFR}[1]{h(#1)}
\newcommand{\voaY}{Y}
\newcommand{\modY}[1]{Y_{#1}}
\newcommand{\intertw}{\mathcal{Y}}
\newcommand{\intertwnorm}[3]{\intertw_{#2 \, #1}^{\;#3}}

\newcommand{\spintertw}{\mathcal{I}}
\newcommand{\inittermsym}{\mathrm{Init}}
\newcommand{\initterm}[1]{\inittermsym [#1]}
\newcommand{\genvar}{\mathfrak{z}}
\newcommand{\genvarB}{\mathfrak{w}}
\newcommand{\genvarC}{\mathfrak{u}}
\newcommand{\ipvar}{\boldsymbol{\zeta}}
\newcommand{\ipvarB}{\boldsymbol{\xi}}

\newcommand{\ipvarO}{\ipvar_0}
\newcommand{\ipvarI}{\ipvar_1}
\newcommand{\ipvarII}{\ipvar_2}
\newcommand{\apvar}{\boldsymbol{x}}
\newcommand{\apvarB}{\boldsymbol{y}}
\newcommand{\voaMin}{\voaM_{\insym}}
\newcommand{\voaMout}{\voaM_{\outsym}}
\newcommand{\voaMmid}{\voaM_{\midsym}}
\newcommand{\modMin}{\voaMin}
\newcommand{\modMout}{\voaMout}
\newcommand{\modMmid}{\voaMmid}
\newcommand{\voavac}{\mathsf{1}}
\newcommand{\voaconfvec}{\omega}

\newcommand{\Res}{\mathrm{Res}}
\newcommand{\voaFusion}[3]{\binom{#3}{#1 \; #2}}
\newcommand{\voaFus}[3]{\voaFusion{#1}{#2}{#3}}
\newcommand{\vir}{\mathfrak{vir}}
\newcommand{\hwv}[1]{\hwvsym_{#1}}
\newcommand{\cb}{\mathbb{M}}
\newcommand{\cbaff}{\widehat{\cb}}
\newcommand{\modMmidaff}{\widehat{\voaM}_{\midsym}}
\newcommand{\virgenrep}{\mathcal{W}}
\newcommand{\singvec}{\chi}
\newcommand{\singvecop}{S}

\newcommand{\sltwo}{\mathfrak{sl}_2}
\newcommand{\Uqsltwo}{\UEA_q (\sltwo)}
\newcommand{\Hcp}{{\Delta}}
\newcommand{\qnum}[1]{\llbracket #1 \rrbracket}
\newcommand{\qfact}[1]{\llbracket #1 \rrbracket !}
\newcommand{\QGgenrep}{\mathsf{V}}
\newcommand{\QGrep}[1]{\mathsf{M}_{#1}}
\newcommand{\QGhwvsym}{u}
\newcommand{\QGhwv}[1]{\QGhwvsym_0^{(#1)}}
\newcommand{\QGvec}[2]{\QGhwvsym_{#2}^{(#1)}}
\newcommand{\selRule}[2]{\mathrm{Sel}(#1,#2)}
\newcommand{\QGembed}[3]{{\iota}_{
        \;{#1}}^{{#2} , {#3}}}
\newcommand{\QGprcan}[3]{{\pi}_{
        {#1} , {#2}}^{\;{#3}}}
\newcommand{\QGpremb}[3]{{\hat{\pi}}_{
        {#1} , {#2}}^{\;{#3}}}
\newcommand{\QGlongprcan}[3]{{\pi}_{
        \{{#1} , {#2}\}}^{\;\;{#3}}}
\newcommand{\QGlongpremb}[3]{{\hat\pi}_{
        \{{#1} , {#2}\}}^{\;\;{#3}}}
\newcommand{\compQGpr}[1]{{P}^{(#1)}}
\newcommand{\sixj}[6]{\Big\{ \begin{array}{ccc}
                             {#2} & {#3} & {#6} \\
                             {#4} & {#1} & {#5}
                             \end{array}\Big\}}
\newcommand{\chamber}{\mathfrak{X}}

\newcommand{\cbseq}{{\varsigma}}
\newcommand{\cbseqB}{{\mu}}
\newcommand{\cbfullseq}{\underline{\cbseq}}
\newcommand{\cbfullseqB}{\underline{\cbseqB}}
\newcommand{\lambdafullseq}{\underline{{\lambda}}}
\newcommand{\sF}{\mathcal{F}}
\newcommand{\HWspace}[1]{\mathcal{H}_{#1}}
\newcommand{\betasym}{B}
\newcommand{\deltasym}{{\Delta}}
\newcommand{\deltasymB}{\widehat{\deltasym}}
\newcommand{\betacoef}[3]{\betasym_{
        {#1} , {#2}}^{\;{#3}}}
\newcommand{\deltaexp}[3]{\deltasym_{
        {#1} , {#2}}^{\;{#3}}}

\newcommand{\ud}{\mathrm{d}}
\newcommand{\der}[1]{\frac{\ud}{\ud {#1}}}
\newcommand{\pder}[1]{\frac{\partial}{\partial {#1}}}
\newcommand{\WittL}[2]{\mathscr{L}^{(#2)}_{#1}}
\newcommand{\varWittL}[2]{\widetilde{\mathscr{L}}_{#1}^{(#2)}}
\newcommand{\varvarWittL}[2]{\widetilde{\mathscr{L}}'{}_{#1}^{(#2)}}
\newcommand{\newvarWittL}[2]{\widehat{\mathscr{L}}_{#1}^{(#2)}}
\newcommand{\ldvarWittL}[2]{\underline{\widetilde{\mathscr{L}}}{\,}_{#1}^{(#2)}}	
\newcommand{\newldvarWittL}[2]{\widehat{\underline{\mathscr{L}}}_{{#1}}^{({#2})}}

\newcommand{\dopL}[1]{\mathscr{L}_{#1}}
\newcommand{\BSAoper}[1]{\mathscr{D}^{(#1)}}
\newcommand{\varBSAoper}[1]{\widetilde{\mathscr{D}}^{(#1)}}
\newcommand{\varvarBSAoper}[1]{\widetilde{\mathscr{D}}'{}^{(#1)}}
\newcommand{\newvarBSAoper}[1]{\widehat{\mathscr{D}}^{(#1)}}
\newcommand{\newldvarBSAoper}[1]{\widehat{\underline{\mathscr{D}}}^{(#1)}{}}
\newcommand{\ldvarBSAoper}[1]{\underline{\widetilde{\mathscr{D}}}{\,}^{(#1)}}
\newcommand{\extraterm}{\mathscr{R}}
\newcommand{\newextraterm}{\widehat{\mathscr{R}}}
\newcommand{\hwmesym}{C}
\newcommand{\newhwme}{\widehat{\hwmesym}}

\usepackage{babel}

  \providecommand{\corollaryname}{Corollary}
  \providecommand{\definitionname}{Definition}
  \providecommand{\lemmaname}{Lemma}
  \providecommand{\propositionname}{Proposition}
  \providecommand{\remarkname}{Remark}
\providecommand{\theoremname}{Theorem}

\usepackage{babel}

  \providecommand{\corollaryname}{Corollary}
  \providecommand{\definitionname}{Definition}
  \providecommand{\lemmaname}{Lemma}
  \providecommand{\propositionname}{Proposition}
  \providecommand{\remarkname}{Remark}
\providecommand{\theoremname}{Theorem}

\usepackage{babel}
\providecommand{\corollaryname}{Corollary}
  \providecommand{\definitionname}{Definition}
  \providecommand{\lemmaname}{Lemma}
  \providecommand{\propositionname}{Proposition}
  \providecommand{\remarkname}{Remark}
\providecommand{\theoremname}{Theorem}

\makeatother

\usepackage{babel}
\providecommand{\corollaryname}{Corollary}
\providecommand{\definitionname}{Definition}
\providecommand{\lemmaname}{Lemma}
\providecommand{\propositionname}{Proposition}
\providecommand{\remarkname}{Remark}
\providecommand{\theoremname}{Theorem}

\begin{document}

\title[Generic universal Virasoro VOA]%
{The quantum group dual \\ of the first-row subcategory \\ for the generic
Virasoro VOA}

\author{Shinji Koshida, Kalle Kyt\"ol\"a}

\address{Department of Physics, Faculty of Science and Engineering, Chuo University, Kasuga, Bunkyo, Tokyo 112-8551, Japan}

\email{koshida@phys.chuo-u.ac.jp}

\address{Department of Mathematics and Systems Analysis, Aalto University,
P.O. Box 11100, FI-00076 Aalto, Finland}

\email{kalle.kytola@aalto.fi}

\begin{abstract}
In several examples it has been observed that a module category of a vertex operator algebra (VOA) is equivalent to a category of representations of some quantum group. The present article is concerned with developing such a duality in the case of the Virasoro VOA at generic central charge; arguably the most rudimentary of all VOAs, yet structurally complicated. We do not address the category of all modules of the generic Virasoro VOA, but we consider the infinitely many modules from the first row of the Kac table. Building on an explicit quantum group method of Coulomb gas integrals, we give a new proof of the fusion rules, we prove the analyticity of compositions of intertwining operators, and we show that the conformal blocks are fully determined by the quantum group method. Crucially, we prove the associativity of the intertwining operators among the first-row modules, and find that the associativity is governed by the $6j$-symbols of the quantum group. Our results constitute a concrete duality between a VOA and a quantum group,
and they will serve as the key tools to establish
the equivalence of the first-row subcategory of modules of the generic Virasoro VOA and the category of (type-1) finite-dimensional representations of
$\Uqsltwo$.
\end{abstract}

\maketitle


\tableofcontents

\newpage

\section{Introduction}%
\label{sec: intro}
\subsubsection*{Conformal field theories, vertex operator algebras, and quantum groups}

Two-dimensional conformal field theories (CFT)
are an outstanding example of extremely fruitful
interaction of physics and mathematics
\cite{DMS-CFT, Gawedzki-lectures_on_CFT,
Huang-CFT_and_VOA, Nahm-bridge_over_troubled_waters}.
Their physical applications include string theory~\cite{GreenSchwarzWitten_superstring}
and critical phenomena in planar statistical physics
\cite{Mussardo-statistical_field_theory},
and they are among the best understood examples
of quantum field theories. In mathematics, ideas
from CFT have been instrumental to the 
Monster simple group~\cite{FLM-VOAs_and_the_Monster}, tensor categories~\cite{BakalovKirillovJr2001}, subfactors~\cite{Kawahigashi2015},
moduli spaces~\cite{FrenkelBen-Zvi2004}, the geometric Langlands program~\cite{Frenkel2007},
and conformally invariant random geometry~\cite{BauerBernard-conformal_transformations},
among others.

As with quantum field theories in general,
the mathematical axiomatization and construction
of CFTs are vast challenges, but CFTs possess 
remarkable structure that has enabled a highly
successful algebraic axiomatization based on
vertex operator algebras (VOA)
\cite{Kac-vertex_algebras, Lepowsky_Li-VOA,
Huang-CFT_and_VOA}.
A VOA 
is essentially
the chiral symmetry algebra of a CFT as envisioned in the seminal work of
Belavin, Polyakov, and 
Zamolodchikov~\cite{BPZ-infinite_conformal_symmetry_in_2d_QFT,
BPZ-infinite_conformal_symmetry_of_critical_fluctiations}.
The symmetry algebra always contains the Virasoro
algebra responsible for the conformal symmetry
itself.
Virasoro vertex operator algebras
are thus fundamental
in that they incorporate only and exactly the minimal
amount of symmetry that any CFT possesses.
At special choices of the central charge~$c$,
a key parameter of CFTs, there may actually 
exist two different Virasoro VOAs: 
the universal Virasoro VOA
\cite{Lepowsky_Li-VOA}
(maximally large) and its irreducible quotient, the
minimal Virasoro VOA
\cite{Wang-rationality_of_Virasoro_VOA}.
At generic values of~$c$, however, the universal Virasoro VOA itself is irreducible,
and we call it a generic Virasoro VOA.
The topic of this article is such
generic Virasoro VOAs.

Generic Virasoro VOAs have been studied very little
in comparison with many other VOAs. The main reason
is that they fail most structural properties that
have enabled significant progress. In particular,
they are far from rational VOAs~\cite{FrenkelZhu1992, Zhu96},
which possess a semisimple category of modules
with finitely many simple objects. Generic Virasoro
VOAs admit first of all infinitely many simple
modules,  and many more indecomposable but not
irreducible ones. We are lacking even the
description of the general indecomposable modules,
let alone the category of such modules equipped with
the desired structures of a tensor product and 
braiding.
In a notable recent 
progress~\cite{CJORY-tensor_categories_arising_from_Virasoro_algebra},
the category of $C_1$-cofinite modules of the
generic Virasoro VOA was studied,
and the tensor product constructed in it.

An intriguing aspect of conformal field theories,
and of the corresponding VOAs, is a hidden
quantum group symmetry. 
In a number of prominent
examples, a representation category of a
suitable quantum group has been found to
agree with a module category of a VOA~---
often together with the tensor products
and braiding in the categories.
The case of (VOAs based on) Wess--Zumino--Witten (WZW) CFTs
at various levels have been
treated in~\cite{Drinfeld1989,
KL-tensor_structures_affine_Lie,
McRae-nonneg_integer_level_affine_Lie_algebra_tensor_cat}, 
and the corresponding quantum group was a $q$-deformation
of the finite dimensional Lie algebra of the corresponding WZW theory.
Another well studied example is the triplet $W$-algebra of logarithmic CFT~\cite{FeiginGainutdinovSemikhatovTipunin2006a,FeiginGainutdinovSemikhatovTipunin2006b,NagatomoTsuchiya2011,KondoSaito2011, TsuchiyaWood2013,
GannonNegron-quantum_sltwo_and_log_VOAs, CreutzigLentnerRupert},
whose representation category is equivalent to that of the {\it restricted} quantum group of $\mathfrak{sl}_{2}$.
Though not as much as a categorical equivalence, a certain structure related to a quantum group has been also observed in the context of Liouville CFT~\cite{TeschnerVartanov2014}.

\subsubsection*{First row subcategory of modules
for the generic Virasoro VOA}

The category of all modules of the generic 
Virasoro VOA being hopelessly complicated,
we focus here on a subcategory we call
the first row subcategory. It is the semisimple
category whose infinitely
many simple modules are the irreducible
Virasoro highest weight modules ``in the
first row of the Kac table'', i.e.,
with highest weights $h=h_{1,s}$, $s \in \Zpos$,
when $h_{r,s}$, $r,s \in \Zpos$, denote the
usual Kac labeled highest 
weights~\cite{Kac-contravariant_form}.
These correspond to a certain infinite 
set of (chiral) primary fields in a CFT, which
has been found to be relevant in particular
to questions in conformally invariant 
random geometry~--- the two simplest of these 
primary fields after the identity, with Kac labeled
conformal weights~$h_{1,2}$ and~$h_{1,3}$,
correspond to SLE-type curves' starting 
points~\cite{BB-CFTs_of_SLEs, Dubedat-commutation}
and boundary visit points~\cite{BB-zig_zag, JJK-SLE_boundary_visits,
Dubedat-fusion},
respectively. For such SLE applications,
the generic Virasoro VOA corresponds to
generic values of the key parameter~$\kappa>0$ of
SLEs~\cite{Werner-random_planar_curves_and_SLE,
RS-basic_properties,
Lawler-conformally_invariant_processes_in_the_plane},
and is completely natural.

By contrast to the general Kac labeled highest weights~$h_{r,s}$,
$r,s \in \Zpos$, at the
first row highest weights~$h_{1,s}$, $s\in \Zpos$
one has truly explicit expressions of singular vectors
in the Virasoro Verma module, by the 
Benoit--St.~Aubin (BSA)
formula~\cite{BSA-degenerate_CFTs_and_explicit_expressions}.
Correspondingly the BPZ partial differential
equations~\cite{BPZ-infinite_conformal_symmetry_in_2d_QFT}
for the correlation functions of these primary fields
are explicit. This is a feature that facilitates
the analysis of the first row subcategory, but
resorting to the explicit partial differential
equations does not in principle seem essential.

Our analysis of this first row subcategory 
of the generic Virasoro VOA is
based on a quantum group method 
of~\cite{KP-conformally_covariant_bdry_correlations},
which is a concrete and practical version of
the hidden quantum group symmetry
\cite{MR-comment_on_quantum_group_symmetry_in_CFT,
PS-common_structures_between_finite_systems_and_CFTs,
RRR-contour_picture_of_quantum_groups_in_CFT,
FW-topological_representation_of_Uqsl2,
SV-quantum_groups_and_homology_of_local_systems,
Varchenko-multidimensional_hypergeometric_functions_and_representation_theory_of_Lie_algebras_and_quantum_groups,
GRS-quantum_groups_in_2d_physics}
developed with applications 
\cite{KP-pure_partition_functions_of_multiple_SLEs,
JJK-SLE_boundary_visits} to
random geometry in mind.
The corresponding quantum group is~$\Uqsltwo$,
a $q$-deformation of~$\sltwo$, at a deformation
parameter~$q$ which is not a root-of-unity.
Correspondingly the category of (type-1) finite-dimensional
representations of this quantum group is semisimple,
with infinitely many irreducible representations,
and it is equipped with tensor products and braiding \cite{Lusztig1993,BakalovKirillovJr2001}.

We consider it noteworthy that our VOA 
without extended symmetries of Lie group type 
(present, e.g., in WZW models) has a quantum group 
counterpart in this way, and that our
VOA, which is irrational with very complicated
representation theory, corresponds to a quantum
group with extremely well-behaved category
of finite-dimensional representations
(albeit with infinitely many irreducibles).
The case more often seen before has been rational
VOAs with good module categories, and
complicated root-of-unity quantum groups whose representation
categories are ``semisimplified'' for certain purposes.

In general terms, our main results are that 
the first row subcategory of modules of the generic
Virasoro VOA is stable under fusion,
and detailed calculations of the fusions
with the quantum group method.

\subsubsection*{Methods}

After reviewing the characterization and
construction of intertwining operators among 
the first row modules, we show that the 
intertwining operators and their arbitrary
compositions are the correlation functions
obtained with the quantum group method.
A priori, the compositions of intertwining
operators of a VOA are formal series,
but this shows that they are actual analytic
functions given by explicit integral formulas.
In particular one obtains convergence of the
series, and straightforward methods of
analytic continuation.

Showing associativity of tensor products of
modules of VOAs is generally a very difficult
task, and one of the main obstacles to 
constructing the appropriate 
tensor category of modules of a VOA
\cite{HuangLepowsky1992, HuangLepowsky1994, Huang--LepowskyI,
Huang--LepowskyII, Huang--LepowskyIII, Huang1995, Huang2005, HuangLepowskyZhangI, HuangLepowskyZhangII, HuangLepowskyZhangIII, HuangLepowskyZhangIV, HuangLepowskyZhangV, HuangLepowskyZhangVI, HuangLepowskyZhangVII, HuangLepowskyZhangVIII},
see also~\cite[Section 2]{Huang--Kirillov--Lepowsky}
for a review. 
The difficulties lie partly in the fact that
the formal series are not even supposed to
correspond to single-valued functions,
so nontrivial branch choices are inevitable,
and yet the convergence and analyticity of the 
formal series is far from obvious.
The explicit analytic expressions from
the quantum group method enable
our proof of associativity.
It is also the explicit expressions that
show the equivalence of the tensor categories
of the finite-dimensional representations
of~$\Uqsltwo$ and of the first-row modules of 
the generic Virasoro VOA.
The operator product expansion (OPE) coefficients
of the corresponding primary fields,
in particular, are explicit, and involve
the quantum $6j$-symbols.

The braiding in the tensor category of VOA 
modules makes the multivaluedness of the functions
even more manifest. We will postpone the
construction of the braiding in the
first row subcategory to a
subsequent article, but the key to it is
similarly the explicit analytical 
expressions that are amenable to analytic 
continuation.

\subsubsection*{Novelty and advantages of the approach}

Our results provide a very satisfactory 
VOA to quantum group duality for the fundamentally
important generic Virasoro VOA, especially when combined with
the follow-up work establishing the equivalence of the
tensor categories of the first row modules of the VOA and of
(type-1) finite-dimensional representations of the
quantum group~$\Uqsltwo$.
The formulation is practical, and 
in particular allows us to perform VOA
calculations with very straightforward 
linear algebra in finite-dimensional
representations of~$\Uqsltwo$.

Conversely the method sheds light onto the quantum
group method
of~\cite{KP-conformally_covariant_bdry_correlations}. Notably, VOA techniques can be used to systematize 
the calculation of general series expansion
coefficients of the correlation functions 
obtained from the method.
Moreover, the result gives a rather satisfactory
characterization of the space of solutions to
BSA PDEs obtained from the quantum group method:
it says that the obtained solutions are exactly
the linear span of the conformal blocks,
which can be described combinatorially, and
in this sense all solutions relevant to CFTs
are included.
By contrast, direct analytical description of
the solution space is complicated already in
the particular case involving only
second order BSA PDEs~\cite{FK-solution_space_for_a_system_of_null_state_PDEs_all}.

Underlying the method and the results is the
key observation that the intertwining operators
in the first row subcategory for the generic
Virasoro VOA are described by explicit analytic
functions, not just formal series.

\subsubsection*{Related works}

The recent 
article~\cite{CJORY-tensor_categories_arising_from_Virasoro_algebra}
also treats the question about the tensor
products in a subcategory of modules for the
generic universal Virasoro VOA. The category
of all $C_1$-cofinite modules considered there
is more general than the first row category
considered in the present article. 
Our approach is thus less general, but it is fully explicit, and 
additionally shows the intimate relationship with the quantum group.

Very recently, in~\cite{GannonNegron-quantum_sltwo_and_log_VOAs},
a ribbon tensor equivalence was established between a module category
of the Virasoro VOA at a central charge lying in a specific series and
a module category of the quantum $\mathrm{SL}_{2}$ at a root of unity.
That article employs results about tensor categories directly,
and specifically the fact that the tensor categories in question have
a distinguished generator.
The first row subcategory of the generic Virasoro VOA is also generated
(as a tensor category) by a single module. With the methods
of~\cite{GannonNegron-quantum_sltwo_and_log_VOAs} one could therefore
expect to obtain a complementary viewpoint to the relationship of the
$\Uqsltwo$ and our first row subcategory, which would be
directly category theoretical but not as explicit about the correlation
functions as our approach.

The setting of the article \cite{TeschnerVartanov2014} involves the same VOA and the same quantum group as the present work,
and the authors also observed a duality between modules of the two.
The modules of both are, however, different from what we consider here.
In other words, the results of us and \cite{TeschnerVartanov2014} thus pertain to different CFTs despite the fact that the VOA and the quantum group are the same.

The quantum group method 
of~\cite{KP-conformally_covariant_bdry_correlations}
relies on tensor products of finite-dimensional
representations of the quantum group~$\Uqsltwo$
at generic values of the deformation parameter~$q$
when the representation theory is semisimple.
Questions about it often reduce to the
commutant of the quantum group on the tensor 
product representation, via a general
$q$-Schur-Weyl duality. 
This approach to calculations with the quantum
group method has been developed in particular
by Flores and Peltola, in a series of articles
\cite{FP-standard_modules_radicals_and_the_valenced_TL,
FP-generators_projectors_and_the_JW,
FP-quantum_Schur_Weyl,
FP-monodromy_invariant_correlations}.
The commutant is a generalization of 
Temperley-Lieb algebas, and Flores and Peltola
have developed specific representation
theoretic tools that are suitable 
for explicit calculations in the quantum group method.
Some of the calculations in the present article,
especially related to the $6j$ symbols,
are closely parallel to such a $q$-Schur-Weyl
duality approach. The results we need are,
however, sufficiently concrete and tractable directly,
so we do not need to introduce the commutant
algebra and its presentation by generators and 
relations.

Ideas of reconstructing intertwining operators for the Virasoro VOA
from integral formulas
appeared already in the seminal article~\cite{Felder-BRST_approach},
and more recently~\cite{KKP-conformal_blocks}
contains a conjecture of a special case of the precise
relationship between the quantum group method and the generic Virasoro
VOA that we establish here.

As a future perspective, it would be desirable to develop the method to a more systematic one and
facilitate generalizations in particular to non-semisimple, root of unity cases.
For that purpose, we view the theories of twisted (co)homologies~\cite{AomotoKita, TsuchiyaWood_IMRN_2014} and Nichols algebras~\cite{Lentner2021} as particularly promising.

\subsection*{Acknowledgments}

SK is supported by the Grant-in-Aid for JSPS Fellows (No.\ 19J01279).

KK wishes to thank Eveliina Peltola and Steven Flores for discussions and for sharing their relevant work with us, Steven Flores and David Radnell for contributions to our study group, and Yi-Zhi Huang for accepting the invitation to give a minicourse in Helsinki and for sharing valuable insights.

KK also wishes to thank the organizers of the conference ``The Mathematical Foundations of Conformal Field Theory and Related Topics'' for hospitality. Numerous discussions at the conference, particularly with Ingo Runkel, Shashank Kanade, Jinwei Yang, and Robert McRae, were very helpful for this work.

\section{Background on the quantum group method}%
\label{sec: quantum group method}
In this section we review the method 
of~\cite{KP-conformally_covariant_bdry_correlations},
by which one construct functions of relevance to
conformal field theories 
from vectors in tensor product representations of a quantum group. 
In Section~\ref{sec: explicit conformal block vectors} we select 
specific vectors in such representations, which will correspond to
the conformal blocks that are crucial to all of our main results:
the construction (Section~\ref{sec: VOA}) of the intertwining operators among the first-row
modules of the generic Virasoro VOA, their
compositions (Section~\ref{sec: composition}), and
associativity (Section~\ref{sec: associativity}).

The textbook~\cite{Kassel-quantum_groups} uses conventions similar to ours
about the quantum group~$\Uqsltwo$.
Our specific choices and notations are identical
to~\cite{KP-conformally_covariant_bdry_correlations}.

\subsection{The quantum group}

The quantum group~$\Uqsltwo$ is the Hopf algebra defined as follows.
Fix a non-zero complex number~$q \in \bC \setminus \set{0}$.

\subsubsection*{Definition of the quantum group}

As an algebra, 
$\Uqsltwo$ is generated by elements
\begin{align*} 
E, F, K, K^{-1} , 
\end{align*}
subject to relations
\begin{align*} 
K K^{-1} = \; & 1 = K^{-1} K , &
K E = \; & q^2 \, E K , \\
E F - F E = \; & \frac{1}{q-q^{-1}} \big( K - K^{-1} \big) &
K F = \; & q^{-2} \, F K . 
\end{align*}
The Hopf algebra structure on~$\Uqsltwo$ is uniquely determined by 
the coproduct, an algebra homomorphism
$\Hcp \colon \Uqsltwo \to \Uqsltwo \tens \Uqsltwo$,
whose values on the generators are
\begin{align*}
\Hcp(K) = K \tens K , \qquad
\Hcp(E) = E \tens K + 1 \tens E , \qquad
\Hcp(F) = F \tens 1 + K^{-1} \tens F .
\end{align*}

\subsubsection*{Representations of the quantum group}

We consider the representations of~$\Uqsltwo$ which continuously $q$-deform the 
finite-dimensional representations of~$\sltwo$. These are specified by a 
highest weight~$\lambda \in \Znn$. 
The $(\lambda+1)$-dimensional representation~$\QGrep{\lambda}$
of~$\Uqsltwo$ has a basis~$(\QGvec{\lambda}{j})_{j=0}^\lambda$ in which the 
generator~$K$ acts diagonally
\begin{align*}
K . \QGvec{\lambda}{j} = \; & q^{\lambda - 2 j} \, \QGvec{\lambda}{j} ,\qquad \text{for $j=0,\dots,\lambda$,}
\end{align*}
and the generators~$E$ and~$F$ act as raising and lowering operators
\begin{align*}
E . \QGvec{\lambda}{j} = \; & \begin{cases}
		0 & \text{ if $j = 0$} \\
		\qnum{j} \, \qnum{\lambda + 1 - j} \, \QGvec{\lambda}{j-1} 
			  & \text{ if $0 < j \leq \lambda$,}
		\end{cases} \\
F . \QGvec{\lambda}{j} = \; & \begin{cases}
		\QGvec{\lambda}{j+1} 
		\phantom{\qnum{j} \, \qnum{\lambda + 1 - j} \,}
			  & \text{ if $0 \leq j < \lambda$} \\
		0 & \text{ if $j = \lambda$,}
		\end{cases} 
\end{align*}
where we used the $q$-integers defined by
\begin{align*}
\qnum{n} = \frac{q^{n}-q^{-n}}{q-q^{-1}} .
\end{align*}
The representations~$\QGrep{\lambda}$, $\lambda \in \Znn$, are irreducible
if~$q$ is not a root of unity, as we will assume throughout the present 
article.

\subsubsection*{Tensor product representations}

Using the coproduct $\Hcp \colon \Uqsltwo \to \Uqsltwo \tens \Uqsltwo$,
we can equip the tensor product~$\QGgenrep' \tens \QGgenrep''$ of any two 
representations~$\QGgenrep', \QGgenrep''$ of~$\Uqsltwo$ with the structure of a 
representation.
Coassociativity of~$\Hcp$ ensures that we can unambiguoulsy define triple 
tensor products such as~$\QGgenrep' \tens \QGgenrep'' \tens \QGgenrep'''$, as 
well as further iterated 
tensor products. One should note, however, that
due to the lack of cocommutativity of~$\Hcp$, we can not canonically identify
$\QGgenrep' \tens \QGgenrep''$ with $\QGgenrep'' \tens \QGgenrep'$. In the 
category of modules that we 
will consider, such identifications can be done by braiding, but the choice of 
braiding direction must be specified.

The rest of the section assumes
that~$q$ is not a root of unity,
\begin{align*}
q^n \neq 1 
\qquad \text{ for all } n \in \bZ \setminus \set{0} .
\end{align*}
Then the tensor products of the 
representations~$\QGrep{\lambda}$, $\lambda \in \Znn$,
are completely reducible, and the 
Clebsch-Gordan decomposition is the same as for~$\sltwo$
\begin{align}
\label{eq:tensor_prod_Uq_mods}
\QGrep{\lambda} \tens \QGrep{\mu}
\; \isom \; \bigoplus_{\ell=0}^{\min(\mu , \lambda)} 
		\QGrep{\mu + \lambda - 2\ell} ,
\end{align}
see, e.g., 
\cite[Lemma~2.4]{KP-conformally_covariant_bdry_correlations}.
Let us therefore define the selection rule set associated 
to~$\mu , \lambda \in \Znn$ as the set of those~$\sigma$ such that
a copy of~$\QGrep{\sigma}$ is 
contained in~$\QGrep{\lambda} \tens \QGrep{\mu}$, i.e.,
\begin{align*}
\selRule{\mu}{\lambda}
:= \set{\sigma \in \Znn \; \Big| \; 
    \sigma + \mu + \lambda \equiv 0 \;(\mathrm{mod}\;2) , \;
    |\mu - \lambda| \leq \sigma \leq \mu + \lambda } .
\end{align*}
Note the symmetries
\begin{align*}
\selRule{\mu}{\lambda} = \selRule{\lambda}{\mu} \qquad \text{ and }
\end{align*}
\begin{align*}
\sigma \in \selRule{\mu}{\lambda}
\quad \Longleftrightarrow \quad
\mu \in \selRule{\sigma}{\lambda}
\quad \Longleftrightarrow \quad
\lambda \in \selRule{\mu}{\sigma} .
\end{align*}
The most convenient formulation of the Clebsch-Gordan rule for our purposes 
is in terms of the following embedding.
\begin{lem}
\label{lem: embedding of QG reps in tensor product of two}
Let~$\mu , \lambda \in \Znn$.
Then we have
\begin{align*}
\dmn \Big( \Hom_{\Uqsltwo} 
    (\QGrep{\sigma} , \; \QGrep{\lambda} \tens \QGrep{\mu}) \Big) 
= \begin{cases}
  1 & \text{ if } \sigma \in \selRule{\mu}{\lambda} \\
  0 & \text{ otherwise} .
  \end{cases}
\end{align*}
In the case $\sigma \in \selRule{\mu}{\lambda}$,
any $\Uqsltwo$-module map
$\QGrep{\sigma} \to \QGrep{\lambda} \tens \QGrep{\mu}$ is proportional 
to the embedding
\begin{align}\label{eq: Clebsch Gordan embedding of irreducible}
\QGembed{\sigma}{\lambda}{\mu} \;
  \colon \; \QGrep{\sigma} \, \hookrightarrow \; 
           \QGrep{\lambda} \tens \QGrep{\mu}
\end{align}
which is uniquely determined by
\begin{align*}
\QGembed{\sigma}{\lambda}{\mu} \big( \QGvec{\sigma}{0} \big)
= \; & \sum_{i,j = 0}^{\frac{\mu+\lambda-\sigma}{2}} 
    \delta_{i+j,\frac{\mu+\lambda-\sigma}{2}} \, (-1)^{j} \,
    \frac{\qfact{\mu-j}\qfact{\lambda-i}}%
        {\qfact{\mu}\qfact{\lambda}\qfact{i}\qfact{j}}
    \frac{q^{j(\mu + 1 - j)}}%
        {(q-q^{-1})^{\frac{\mu+\lambda-\sigma}{2}}}
    \; (\QGvec{\lambda}{i} \tens \QGvec{\mu}{j}) .
\end{align*}
\end{lem}
\begin{proof}
This follows directly from, e.g.,
\cite[Lemma~2.4]{KP-conformally_covariant_bdry_correlations}.
\end{proof}
Since for $\sigma \in \selRule{\mu}{\lambda}$ the multiplicity
of the irreducible representation~$\QGrep{\sigma}$ in the 
tensor product~$\QGrep{\lambda} \tens \QGrep{\mu}$ is one,
there exist a unique $\Uqsltwo$-module map
\begin{align}\label{eq: Clebsch Gordan projection with embedding}
\QGpremb{\lambda}{\mu}{\sigma} \colon
    \QGrep{\lambda} \tens \QGrep{\mu} \to \QGrep{\sigma}
\qquad \text{ such that } \qquad
\QGpremb{\lambda}{\mu}{\sigma} \circ \QGembed{\sigma}{\lambda}{\mu}
    = \idof{\QGrep{\sigma}} .
\end{align}

The projection
\begin{align*}
\QGprcan{\lambda}{\mu}{\sigma} \colon
    \QGrep{\lambda} \tens \QGrep{\mu} \to \QGrep{\lambda} \tens \QGrep{\mu} ,
\end{align*}
from $\QGrep{\lambda} \tens \QGrep{\mu}$ to its unique
subrepresentation isomorphic to~$\QGrep{\sigma}$
then agrees with the composition of
$\QGpremb{\lambda}{\mu}{\sigma}$ with the
embedding $\QGembed{\sigma}{\lambda}{\mu}$,
\begin{align}\label{eq: Clebsch Gordan projection within tensor product}
\QGprcan{\lambda}{\mu}{\sigma}
    = \QGembed{\sigma}{\lambda}{\mu} \circ \QGpremb{\lambda}{\mu}{\sigma} .
\end{align}

\subsection{The correspondence with functions}
\label{sub: quantum group method main statement}

Throughout this section, we parametrize~$q$ by $\kappa$ via
\begin{align*}
q = q(\kappa) = \; e^{\ii 4 \pi / \kappa} ,
\end{align*}
and assume that~$\kappa \in (0,\infty) \setminus \bQ$.
Then $q$~has unit modulus, but is not a root of unity.

For~$\lambda \in \Znn$, we write
\begin{align}\label{eq: first row conformal weight ad hoc}
\hwFR{\lambda} 
:= \frac{\lambda \, \big( 2 (\lambda+2)-\kappa \big)}{2\kappa}
\end{align}
for the conformal weight of a module in the 
``first row of the Kac table'' (see Section~\ref{sec: VOA}).

For $N \in \bN$, let us denote by
\begin{align*}
\chamber_N := \set{ (x_1, \ldots x_N) \in \bR^N \; \big| \;
		x_1 < \cdots < x_N}
\end{align*}
the chamber of $N$ ordered real variables.
We also fix parameters~$\lambda_1 , \ldots, \lambda_N \in \Znn$,
and sometimes refer to all of them collectively as
\[ \lambdafullseq = (\lambda_1 , \ldots, \lambda_N) . \]

We are interested in functions~$F \colon \chamber_N \to \bC$ that
satisfy certain properties motivated by conformal field theory:
$N$~linear partial differential equations,
asymptotics as two variables approach each other,
as well as translation invariance, homogeneity, and
sometimes covariance under more general M\"obius transformations.

Specifically, for each~$j \in \set{1,\ldots,N}$,
we define a Benoit--Saint-Aubin partial differential operator
\begin{align}\label{eq: BSA differential operator}
\BSAoper{j} := \sum_{k=1}^{\lambda_j+1} 
    \sum_{\substack{p_1 , \ldots , p_k \geq 1 \\ 
                    p_1 + \cdots + p_k = \lambda_j + 1}}
        \frac{(-4/\kappa)^ {1+\lambda_j-k} \, \lambda_j !^2}%
        {\prod_{u=1}^{k-1} (\sum_{i=1}^{u} p_i) (\sum_{i=u+1}^{k} p_i)}
      \; \WittL{-p_1}{j} \cdots \WittL{-p_k}{j} ,
\end{align}
of order~$\lambda_j+1$, where
\begin{align}\label{eq: Witt differential operator}
\WittL{n}{j} := 
    - \sum_{\substack{1 \leq i \leq N \\ i \neq j}}
        \bigg( (x_i-x_j)^{1+n} \pder{x_i} 
				+ (1+n) \hwFR{\lambda_i} \, (x_i-x_j)^n \bigg)
    , \qquad \text{for } n\in\Z .
\end{align}
The parameters~$\lambda_1 , \ldots, \lambda_N \in \Znn$ 
as well as the number of variables~$N$
are implicit in this notation.
Note that with a fixed $j$,
the operators $\WittL{n}{j}$ 
satisfy the Witt algebra commutation relations
\begin{align*}
	[\WittL{m}{j},\WittL{n}{j}]=(m-n)\WittL{m+n}{j},
\qquad \text{ for } m,n\in \bZ .
\end{align*}

The correspondence involves the representation of the quantum 
group~$\Uqsltwo$ constructed as the tensor product of the irreducible 
representations $\QGrep{\lambda_1} , \ldots , \QGrep{\lambda_N}$. The 
tensorands are ordered from left to right in the reverse order of the index, 
and we use the shorthand notation
\begin{align}\label{eq: big tensor product}
\bigotimes_{i=1}^N \QGrep{\lambda_i}
\; = \; \QGrep{\lambda_N} \otimes \cdots \otimes \QGrep{\lambda_1}
\end{align}
for this ordering convention. Similarly we denote, e.g.,
\begin{align*}
\bigotimes_{i<j} \QGrep{\lambda_i}
\, = \; & \QGrep{\lambda_{j-1}} \otimes \cdots \otimes \QGrep{\lambda_1} , &
\bigotimes_{i>j} \QGrep{\lambda_i} 
\, = \; & \QGrep{\lambda_{N}} \otimes \cdots \otimes \QGrep{\lambda_{j+1}} .
\end{align*}
For the projection $\QGprcan{\lambda_{j+1}}{\lambda_{j}}{\tau}$ 
of~\eqref{eq: Clebsch Gordan projection within tensor product} applied in the 
two consecutive tensorands with indices~$j, j+1$, we use the notation
\begin{align*}
\QGlongprcan{j}{j+1}{\tau} \, \colon \; &
\bigotimes_{i=1}^N \QGrep{\lambda_i} \to \bigotimes_{i=1}^N \QGrep{\lambda_i}
\\
\QGlongprcan{j}{j+1}{\tau} \, = \; &
\Big( \big( \bigotimes_{i>j+1} \id_{\QGrep{\lambda_i}} \big) \tens 
	\QGprcan{\lambda_{j+1}}{\lambda_j}{\;\;\tau} 
    \tens \big( \bigotimes_{i<j} \id_{\QGrep{\lambda_i}} \big) \Big) ,
\end{align*}
and when the projection $\QGpremb{\lambda_{j+1}}{\lambda_{j}}{\sigma}$ 
is applied instead (thus reducing the number of tensorands by one), we use the 
notation
\begin{align*}
\QGlongpremb{j}{j+1}{\tau} \, \colon \; &
\bigotimes_{j=1}^N \QGrep{\lambda_j} \to 
\Big( \big( \bigotimes_{i>j+1} \QGrep{\lambda_i} \big) \tens \QGrep{\tau} \tens
\big( \bigotimes_{i<j} \QGrep{\lambda_i} \big) \Big) \\
\QGlongpremb{j}{j+1}{\tau} \, = \; &
\Big( \big( \bigotimes_{i>j+1} \id_{\QGrep{\lambda_i}} \big) \tens 
	\QGpremb{\lambda_{j+1}}{\lambda_j}{\;\;\tau} 
    \tens \big( \bigotimes_{i<j} \id_{\QGrep{\lambda_i}} \big) \Big) .
\end{align*}
Within the tensor product~\eqref{eq: big tensor product},
we are primarily concerned with the subspace
\begin{align}
\HWspace{\lambdafullseq} :=
    \set{ v \in \bigotimes_{j=1}^N \QGrep{\lambda_j}
        \; \bigg| \; E.v =0 } 
\end{align}
consisting of highest weight vectors.

\begin{thm}[\cite{KP-conformally_covariant_bdry_correlations}]
\label{thm:properties_of_correlations}
There is a family of linear mappings
$\sF \colon \HWspace{\lambdafullseq} \to \ContF^\infty(\chamber_N)$
indexed by $\lambdafullseq \in \bigsqcup_{N \in \bN} \bN^N$,
normalized so that
for $N=1$ and any $\lambda_1 \in \bN$
we have $\sF[\QGvec{\lambda_1}{0}](x_1) \equiv 1$,
and with the following properties:
\begin{itemize}
\item[(PDE)] For any $u \in \HWspace{\lambdafullseq}$
the function $F = \sF[u] \colon \chamber_N \to \bC$ satisfies
\begin{align}
\nonumber
\BSAoper{j} F (x_1 , \ldots , x_N) \, = \, 0 &
\\ \nonumber
\text{ for all $j=1,\ldots,N$ and $(x_1 , \ldots , x_N) \in \chamber_N$} & \, .
\end{align}
\item[(COV)] For any $u \in \HWspace{\lambdafullseq}$
the function $F = \sF[u] \colon \chamber_N \to \bC$ is translation invariant,
\begin{align}
\label{eq:trans_inv}
F (x_1 + t , \ldots , x_N + t) \, = \,
F (x_1 , \ldots , x_N) &
\\ \nonumber
\text{ for all $(x_1 , \ldots , x_N) \in \chamber_N$ and $t \in \bR$}
\, & .\notag
\end{align}
If, moreover, 
$u$ is a Cartan eigenvector,
$u \in \HWspace{\lambdafullseq} \cap \Kern(K - q^{\sigma})$, 
then the function
$F = \sF[u] \colon \chamber_N \to \bC$ is translation invariant and homogeneous,
\begin{align}
\label{eq:trans_inv_homog}
F (s x_1 + t , \ldots , s x_N + t) \, = \,
s^{\hwFR{\sigma} - \sum_{i=1}^N \hwFR{\lambda_i}} \; F (x_1 , \ldots , x_N) &
\\ \nonumber
\text{ for all $(x_1 , \ldots , x_N) \in \chamber_N$ and $t \in \bR$, $s>0$}
\, & . \notag
\end{align}
Finally, if $u$ lies in a trivial 
subrepresentation, $u \in \HWspace{\lambdafullseq} \cap \Kern(K-1)$, then the 
function
$F = \sF[u] \colon \chamber_N \to \bC$ is fully M\"obius-covariant in the sense 
that for any $(x_1 , \ldots , x_N) \in \chamber_N$ and
any $\Mob(z) = \frac{az+b}{cz+d}$ such that $\Mob(x_1) < \cdots < \Mob(x_N)$,
we have
\begin{align}
\label{eq: Mobius covariance}
F ( \Mob(x_1) , \ldots , \Mob(x_N) ) \, = \,
\prod_{i=1}^N \Mob'(x_i)^{-\hwFR{\lambda_i}} \; F (x_1 , \ldots , x_N) .
\end{align}
\item[(ASY)] If $u \in \HWspace{\lambdafullseq}$ lies in the subrepresentation 
corresponding to the irreducible~$\QGrep{\tau}$ in the tensor product of the 
$j$:th and $j+1$:th factors, i.e., if
$ u = \QGlongprcan{j}{j+1}{\tau} (u) $,
then the function $F = \sF[u] \colon \chamber_N \to \bC$ has the expansion
\begin{align*}
& F(x_1, \ldots, x_j , x_{j+1} , \ldots , x_N) \\
= \; & \betasym \, (x_{j+1}-x_j)^{\deltasym} \, \Big( 
    \hat{F} \big( x_1, \ldots, 
        \xi , 
							\ldots , x_N \big)
        + 
          \oo (1) \Big) 
\qquad \text{ as $x_{j} , x_{j+1} \to \xi$},
\end{align*}
where $\hat{F} = \sF[\hat{u}]$ with
\begin{align*}
\hat{u} = \; & \QGlongpremb{j}{j+1}{\tau} (u) \;
\in \; \Big( \big( \bigotimes_{i>j+1} \QGrep{\lambda_i} \big)
    \tens \QGrep{\tau} \tens \big( \bigotimes_{i<j} \QGrep{\lambda_i} \big)
    \Big) ,
\end{align*}
and
$
\betasym = \betacoef{\lambda_{j+1}}{\lambda_j}{\;\;\tau}
		= \frac{1}{((\lambda_j+\lambda_{j+1}-\tau)/2)!}
		  \prod_{p=1}^{(\lambda_j+\lambda_{j+1}-\tau)/2}
		    \frac{\Gamma( 1+\frac{4}{\kappa} p ) \,
				  \Gamma( 1-\frac{4}{\kappa} (1+\lambda_j-p) ) \,
				  \Gamma( 1-\frac{4}{\kappa} (1+\lambda_{j+1}-p) ) }%
		      {\Gamma( 1+\frac{4}{\kappa} ) \,
				  \Gamma( 2-\frac{2}{\kappa} 
					  (4-2p + \lambda_j + \lambda_{j+1} + \tau) )}
$
and
$\deltasym = \deltaexp{\lambda_{j+1}}{\lambda_j}{\;\;\tau}
		= \hwFR{\tau} - \hwFR{\lambda_j} - \hwFR{\lambda_{j+1}}$.
\end{itemize}
\end{thm}

\subsection{Series expansions of the functions}

The method of Theorem~\ref{thm:properties_of_correlations}
in fact yields not only smooth functions,
but analytic functions which have Frobenius series
expansions on the codimension one
boundaries of the chamber~$\chamber_N$.
These Frobenius series will be important in 
Sections~\ref{sec: composition} and~\ref{sec: associativity}.
We start with a general definitions about the
assumptions we use on parametrized power series, and
then state the series expansion results.
The proofs are left to Appendix~\ref{app: series expansions}.

\subsubsection*{Controlled parametrized power series}

We will need to expand the functions of 
Section~\ref{sub: quantum group method main statement}
as power series recursively one variable at a time. Therefore,
we will treat one of the variables as the variable of the
power series, and the other variables as parameters.
In order to be able to
perform the natural operations on the power series, we need
the following type of control of the power series coefficients
locally uniformly over the the parameters.

\begin{defn}
\label{def: controlled power series}
Let $\Omega \subset \bR^m$ be an open set and $c_k \colon \Omega \to \bC$
smooth functions for~${k \in \Znn}$. For $R>0$, we 
say that $(c_k)_{k \in \Znn}$
are \term{locally uniformly $R$-controlled power series coefficients}
if for every compact $K\subset \Omega$ and every 
multi-index~$\alpha \in \bN^m$ we have
\begin{align*}
\limsup_{k \to \infty} 
  \big( \sup_{y \in K} |\partial^\alpha c_k(y)| \big)^{1/k} \leq \frac{1}{R} .
\end{align*}
\end{defn}

As simpler terminology, in the above situation
we may just say that the power series
\begin{align*}
\sum_{k=0}^\infty c_k(y) \, z^k
\end{align*}
parametrized by~$y \in \Omega$ is locally uniformly $R$-controlled.
Note that by the Cauchy-Hadamard formula for the radius of
convergence, this implies in particular that for any~$y \in \Omega$
the radius of convergence of the power series itself and
its coefficientwise derivatives with respect to the parameters~$y$
have radius of convergence at least~$R$.

\subsubsection*{Analyticity and Frobenius series statements}


The analyticity statement in a single variable
for the functions from Theorem~\ref{thm:properties_of_correlations}
is the following.
\begin{lem}
\label{lem: analyticity of QG functions}
Let $F = \sF[u] \colon \chamber_N \to \bC$
be the function associated to 
any~$u\in\HWspace{\lambdafullseq}$,
and let $(x_1, \ldots, x_N) \in \chamber_N$,
and let $j \in \set{1,\ldots,N}$.
Then we have a power series expansion
\begin{align*}
F(x_1, \ldots, x_{j-1}, z_j, x_{j+1}, \ldots, x_N)
\; = \; \sum_{k=0}^\infty
    c_k (x_1, \ldots, x_{j-1}, x_{j+1}, \ldots, x_N) \; \big( z_j - x_j \big)^k
\end{align*}
in the~$j$:th variable.
For fixed $x_j \in \bR$ and $R>0$,
viewing the other variables $(x_i)_{i \neq j}$ as parameters,
on the subset $\Omega \subset \bR^{N-1}$
defined by the conditions $x_1 < \cdots < x_N$
and $\min_{i \neq j} |x_i - x_j| > R$, the power series
is locally uniformly $R$-controlled.
\end{lem}
The proof is elementary, but it
is instructive as a preparation for the
consideration of the Frobenius series, so we give it in
Appendix~\ref{app: series expansions}.

The Frobenius series statement that we will use
is the following. Variants of this formulation
with obvious modifications to the statement and
proof could be done as well.
\begin{lem}
\label{lem: Frobenius series of QG functions}
Let $j \in \set{2,\ldots,N}$. 
Suppose that~$\tau \in \selRule{\lambda_{j-1}}{\lambda_j}$
and that $u\in\HWspace{\lambdafullseq}$
is such that~$u = \QGlongprcan{j-1}{j}{\tau} (u)$.
The function
$F = \sF[u] \colon \chamber_N \to \bC$
associated to~$u$ has a 
Frobenius series expansion in the variable~$z=x_j-x_{j-1}$
\begin{align*}
F \big( x_1, \ldots 
    , x_{j-1} , (x_{j-1}+z) , x_{j+1} , \ldots, x_N \big)
\; = \; z^{\deltasym} \,
    \sum_{k=0}^\infty 
      c_k (x_1, \ldots, x_{j-1}, x_{j+1}, \ldots, x_N) \; z^{k}
\end{align*}
where
the indicial exponent is
$\deltasym = \hwFR{\tau} - \hwFR{\lambda_{j}} - \hwFR{\lambda_{j-1}}$. 
For fixed $R>0$,
viewing the other variables $(x_i)_{i \neq j}$ as parameters,
on the subset $\Omega \subset \bR^{N-1}$
defined by the conditions $x_1 < \cdots < x_N$
and $\min_{i \neq j, j-1} |x_i - x_{j-1}| > R$, the power series
part of this Frobenius series is locally uniformly $R$-controlled,
and for $0<z<R$ the Frobenius series converges to the function~$F$
on the left hand side.
\end{lem}
\begin{rmk}\label{rmk: leading coefficient of Frobenius series}
The leading coefficient~$c_0$ of the Frobenius series
in Lemma~\ref{lem: Frobenius series of QG functions}
is related to the value of the function
associated to~$\hat{u} = \QGlongpremb{j-1}{j}{\tau} (u)$
as follows
\begin{align*}
c_0 (x_1, \ldots , x_{j-1}, x_{j+1}, \ldots, x_N)
\; = \; \betasym \times 
        \sF [\hat{u}]
            (x_1, \ldots , x_{j-1}, x_{j+1}, \ldots, x_N) ,
\end{align*}
where $\betasym = \betacoef{\lambda_{j}}{\lambda_{j-1}}{\;\;\tau}$
is as in Theorem~\ref{thm:properties_of_correlations}.
This can be seen from the (ASY) part of
Theorem~\ref{thm:properties_of_correlations} (or more
directly from the proofs).
\end{rmk}

The proof of Lemma~\ref{lem: Frobenius series of QG functions}
is based on an elaboration of ideas
from~\cite{KP-conformally_covariant_bdry_correlations},
in particular those leading to part~(ASY)
of Theorem~\ref{thm:properties_of_correlations}.
Since this result will be crucially relied on in
the present article, we outline the proof in
Appendix~\ref{app: series expansions}.

\section{Construction of conformal block vectors for the quantum group}%
\label{sec: explicit conformal block vectors}
In this section we construct specific vectors in
tensor product representations of the quantum group~$\Uqsltwo$,
which will correspond to our basis of ``conformal blocks''.
More precisely, in the subsequent sections,
compositions of intertwining operators in the first row
category of modules for the generic Virasoro VOA
will be obtained from these vectors via
the quantum group method of Section~\ref{sec: quantum group method}.

\subsection{The quantum $6j$-symbols}

The embeddings $\QGembed{\sigma}{\lambda}{\mu} \;
  \colon \; \QGrep{\sigma} \, \hookrightarrow \; 
           \QGrep{\lambda} \tens \QGrep{\mu}$
in Lemma~\ref{lem: embedding of QG reps in tensor product of two}
can be combined in different ways, and the 
relationships between the choices are given by the 
quantum $6j$-symbols.
\begin{lem}\label{lem: bases for 6j symbols}
Let~$\lambda_1 , \lambda_2 , \lambda_3 \in \Znn$.
Then for any~$\sigma \in \Znn$, the space
\begin{align*}
\Hom_{\Uqsltwo} 
    \big( \QGrep{\sigma} , \; 
    \QGrep{\lambda_3} \tens \QGrep{\lambda_2} \tens \QGrep{\lambda_1} \big) 
\end{align*}
has one basis consisting of 
\begin{align*}
\big( \QGembed{\nu}{\lambda_3}{\lambda_2} \tens \id_{\QGrep{\lambda_1}} \big)
\circ \QGembed{\sigma}{\nu}{\lambda_1} 
, \qquad 
\nu \in \selRule{\sigma}{\lambda_1} \cap \selRule{\lambda_3}{\lambda_2} ,
\end{align*}
and another basis consisting of 
\begin{align*}
\big( \id_{\QGrep{\lambda_3}} \tens \QGembed{\kappa}{\lambda_2}{\lambda_1} \big)
\circ \QGembed{\sigma}{\lambda_3}{\kappa}
, \qquad 
\kappa \in \selRule{\sigma}{\lambda_3} \cap \selRule{\lambda_2}{\lambda_1} .
\end{align*}
\end{lem}
\begin{proof}
Clearly the given maps are $\Uqsltwo$-module maps
$\QGrep{\sigma} \to 
\QGrep{\lambda_3} \tens \QGrep{\lambda_2} \tens \QGrep{\lambda_1}$.
It follows straightforwardly from 
coassociativity and Clebsch-Gordan decompositions
that the two collections both span and are linearly independent.
\end{proof}

The expansions of the elements of the second basis
of Lemma~\ref{lem: bases for 6j symbols}
with respect to the first basis are denoted as in the following:
\begin{align}\label{eq: 6j decomposition}
\big( \id_{\QGrep{\lambda_3}} \tens \QGembed{\kappa}{\lambda_2}{\lambda_1} \big)
        \circ \QGembed{\sigma}{\lambda_3}{\kappa}
= \; & \sum_{\nu} \sixj{\sigma}{\lambda_3}{\lambda_2}{\lambda_1}{\kappa}{\nu}
\big( \QGembed{\nu}{\lambda_3}{\lambda_2} \tens \id_{\QGrep{\lambda_1}} \big)
\circ \QGembed{\sigma}{\nu}{\lambda_1} .
\end{align}
\begin{align*}
    \hackcenter{
    \begin{tikzpicture}
        \draw (0,2)--(1,1)--(1,0);
        \draw (1,2)--($(0,2)!(1,2)!(1,1)$);
        \draw (2,2)--($(0,2)!(2,2)!(1,1)$);
        \draw (0,2)node[above]{$\sigma$};
        \draw (1,0)node[below]{$\lambda_{1}$};
        \draw (1,2)node[above]{$\lambda_{3}$};
        \draw (2,2)node[above]{$\lambda_{2}$};
        \draw ($($(0,2)!(1,2)!(1,1)$)!0.5!($(0,2)!(2,2)!(1,1)$)$)node[below left]{$\kappa$};
    \end{tikzpicture}
    } \; = \;\;\; \sum_{\nu}
    \sixj{\sigma}{\lambda_3}{\lambda_2}{\lambda_1}{\kappa}{\nu}
    \hackcenter{
     \begin{tikzpicture}
        \draw (0,2)--(1,1)--(1,0);
         \draw (2,2)--($(0,2)!(2,2)!(1,1)$);
        \draw (1,2)--($(2,2)!(1,2)!(1,1)$);
        \draw (0,2)node[above]{$\sigma$};
        \draw (1,0)node[below]{$\lambda_{1}$};
        \draw (1,2)node[above]{$\lambda_{3}$};
        \draw (2,2)node[above]{$\lambda_{2}$};
        \draw ($($(2,2)!(1,2)!(1,1)$)!0.5!(1,1)$)node[below right]{$\nu$};
    \end{tikzpicture}}
\end{align*}
The coefficients $\sixj{\sigma}{\lambda_3}{\lambda_2}{\lambda_1}{\kappa}{\nu}$
in these expansions are called the quantum $6j$ symbols.

\subsection{The construction of the conformal block vectors}

Fix
\begin{align*}
\lambdafullseq = ( \lambda_{0}, \lambda_{1},  \ldots, \lambda_{N}, \lambda_{\infty} ) \in \Znn^{N+2}
\end{align*}
throughout. 
Note that compared to Section~\ref{sec: quantum group method}
we now have two additional labels, $\lambda_{0}$ and~$\lambda_{\infty}$.
They will later be seen to have the interpretations of labels of 
primary fields at the origin and at infinity.

\begin{defn}\label{def: admissible sequence}
A sequence
\begin{align*}
\cbfullseq = (\cbseq_0 , \cbseq_1 , \ldots, \cbseq_{N-1}, \cbseq_{N}) \in 
\Znn^{N+1}
\end{align*}
is said to be $\lambdafullseq$-admissible, if we have
\begin{align*} 
\cbseq_{0} = \lambda_{0} , \qquad \cbseq_{N} = \lambda_{\infty} ,
& \qquad \text{ and } \\
\lambda_{j} \in \selRule{\cbseq_j}{\cbseq_{j-1}} 
& \qquad \text{ for all $j = 1 , \ldots , N$}  
\end{align*}
\end{defn}
The following picture should serve as a 
visual guide to what is going on at the level of
representations of~$\Uqsltwo$ (and in fact at the level of
modules of the generic Virasoro VOA in later sections);
the $\lambdafullseq$-admissibility of the sequence
$\cbfullseq$ exactly ensures that the selection rules
are satisfied at every vertex of this picture:

\begin{align*} 
\begin{tikzpicture}
    \draw (0,0)--(6,0);
    \draw (1,1)--($(0,0)!(1,1)!(6,0)$);
    \draw (2,1)--($(0,0)!(2,1)!(6,0)$);
    \draw (4,1)--($(0,0)!(4,1)!(6,0)$);
    \draw (5,1)--($(0,0)!(5,1)!(6,0)$);
    \draw (3,1/2)node {$\cdots$};
    \draw (0,0)node[left]{$\cbseq_{N}$\;};
    \draw (-.3,-.5)node[left]{$\rotatebox{90}{=}$};
     \draw (0,-1)node[left]{$\lambda_{\infty}$};
    \draw (6,0)node[right]{\;$\cbseq_{0}$\; .};
    \draw (6.2,-.5)node[right]{$\rotatebox{90}{=}$};
    \draw (6,-1)node[right]{\;$\lambda_{0}$};
    \draw (1,1)node[above]{$\lambda_{N}$};
    \draw (2,1)node[above]{$\lambda_{N-1}$};
    \draw (4,1)node[above]{$\lambda_{2}$};
    \draw (5,1)node[above]{$\lambda_{1}$};
    \draw (1.5,0)node[below]{$\cbseq_{N-1}$};
    \draw (4.5,0)node[below]{$\cbseq_{1}$};
\end{tikzpicture}
\end{align*}

We seek to associate a conformal block to such an
admissible sequence, and for that purpose we will
first associate to it a suitable vector
\begin{align*}
u_{\cbfullseq} \; \in \; 
\Big( \bigotimes_{j=1}^N \QGrep{\lambda_j} \Big) \tens \QGrep{\lambda_{0}} .
\end{align*}
Our convention is that tensor products are formed in the order with the indices 
increasing from right to left, i.e., the space above is
\[ 
\QGrep{\lambda_N} \tens \bigg( \QGrep{\lambda_{N-1}} \tens \Big( \cdots 
     \tens \big( \QGrep{\lambda_2} \tens
         (\QGrep{\lambda_1} \tens \QGrep{\lambda_0} )
     \big) \cdots \Big) \bigg) .
\]
Above we have placed the parentheses to illustrate the 
idea according to which the vector~$u_{\cbfullseq}$ is chosen.
The construction of the vector is done with 
the composition
\begin{align}
\label{eq: sequence of embeddings}
\xymatrix{%
\;\QGrep{\cbseq_{N}} \;
  \ar@{^{(}->}[r]
& \;\QGrep{\lambda_N} \tens \QGrep{\cbseq_{N-1}} \;
  \ar@{^{(}->}[r]
& \; \QGrep{\lambda_N} \tens \QGrep{\lambda_{N-1}} \tens \QGrep{\cbseq_{N-2}} \;
  \ar@{^{(}->}[r]
& \; \cdots \;\;
  \ar@{^{(}->}[r]
& \; \big(\bigotimes_{j=1}^N \QGrep{\lambda_j} \big)
    \tens \QGrep{\cbseq_{0}} ,
}
\end{align}
of embeddings
\begin{align*}
\xymatrix{%
\Big( \big( \bigotimes_{i>j} \idof{\QGrep{\lambda_i}} \big) \tens 
    \QGembed{\;\cbseq_{j}}{\lambda_{j}}{\cbseq_{j-1}} \Big)
\; \colon \; &
\Big( \big( \bigotimes_{i>j} \QGrep{\lambda_i} \big) \tens 
    \QGrep{\cbseq_{j}} \Big) \;
\ar@{^{(}->}[r] 
& \Big( \big( \bigotimes_{i>j-1} \QGrep{\lambda_i} \big) \tens 
    \QGrep{\cbseq_{j-1}} \Big)
} .
\end{align*}
Namely, we take $u_{\cbfullseq}$ to be the image of the highest weight 
vector~$\QGhwv{\lambda_{\infty}} \in \QGrep{\lambda_{\infty}}$
under this composition of  embeddings,
\begin{align}
\label{eq: construction of conformal block vector}
u_{\cbfullseq} \; := \; 
\; & \bigg( \Big( \big( \bigotimes_{j=2}^N \idof{\QGrep{\lambda_j}} \big) 
                \tens \QGembed{\cbseq_1}{\lambda_1}{\cbseq_{0}} \Big) \circ
     \cdots \circ
     \Big( \idof{\QGrep{\lambda_N}} 
        \tens \QGembed{\cbseq_{N-1}}{\lambda_{N-1}}{\cbseq_{N-2}} \Big) \circ
     \QGembed{\lambda_{\infty}}{\lambda_{N}}{\cbseq_{N-1}} 
   \bigg) (\QGhwv{\lambda_{\infty}}) .
\end{align}

\begin{lem}
The vector $u_{\cbfullseq}$
in~\eqref{eq: construction of conformal block vector} satisfies
\begin{align*}
E.u_{\cbfullseq} = 0 
\qquad \text{ and } \qquad
K.u_{\cbfullseq} = q^{\lambda_{\infty}} \, u_{\cbfullseq} .
\end{align*}
\end{lem}
\begin{proof}
These properties are satisfied by the 
vector~$\QGhwv{\lambda_{\infty}} \in \QGrep{\lambda_{\infty}}$, and the mapping applied on 
this vector in~\eqref{eq: construction of conformal block vector} is a 
$\Uqsltwo$-module map.
\end{proof}

More formally, the reason for the choice of~$u_{\cbfullseq}$ is the 
following projection conditions.
We use partial compositions of the following sequence 
\begin{align*}
\xymatrix{%
\QGrep{\lambda_N} \tens \cdots \tens 
    \QGrep{\lambda_{1}} \tens \QGrep{\lambda_{0}} \;
  \ar@{->>}[r]
& \;\; \cdots \;
  \ar@{->>}[r]
& \; \QGrep{\lambda_N} \tens \QGrep{\lambda_{N-1}} \tens \QGrep{\cbseq_{N-2}} \;
  \ar@{->>}[r]
& \; \QGrep{\lambda_N} \tens \QGrep{\cbseq_{N-1}} \;
  \ar@{->>}[r]
& \; \QGrep{\lambda_{\infty}} 
}
\end{align*}
of projections
\begin{align*}
\xymatrix{%
\Big( \big( \bigotimes_{i>j} \idof{\QGrep{\lambda_i}} \big) \tens 
    \QGpremb{\lambda_{j}}{\cbseq_{j-1}}{\;\cbseq_{j}} \Big)
\; \colon \; &
\Big( \big( \bigotimes_{i>j-1} \QGrep{\lambda_i} \big) \tens 
    \QGrep{\cbseq_{j-1}} \Big) \;
\ar@{->>}[r] 
& \Big( \big( \bigotimes_{i>j} \QGrep{\lambda_i} \big) \tens 
    \QGrep{\cbseq_{j}} \Big)
} 
\end{align*}

\begin{prop}
For each $j=1,\dots, N$, denote by
$\compQGpr{j} (u_{\cbfullseq}) \in 
\big( \bigotimes_{i \geq j} \QGrep{\lambda_i} \big) \tens \QGrep{\cbseq_{j-1}}$ 
the image of $u_{\cbfullseq}$ under the composition
of the first~$j-1$ projections above.
The following conditions hold for the vector $u_{\cbfullseq}$
in~\eqref{eq: construction of conformal block vector}:
\begin{align*}
\compQGpr{j} (u_{\cbfullseq})
= \; & \Big( \big( \bigotimes_{i>j} \idof{\QGrep{\lambda_i}} \big) 
                \tens \QGprcan{\lambda_j}{\cbseq_{j-1}}{\cbseq_j} \Big)
			  \big( \compQGpr{j} (u_{\cbfullseq}) \big)
\end{align*}
for each $j=1,\ldots,N$.
Moreover, $u_{\cbfullseq}$ is up to a multiplicative constant 
the unique vector in 
$\big( \bigotimes_{j=1}^N \QGrep{\lambda_j} \big) \tens \QGrep{\lambda_{0}}$
for which the above conditions hold.
\end{prop}
\begin{proof}
The conditions for vector $u_{\cbfullseq}$ follow directly from its 
construction, using the relationships
$\QGprcan{\lambda}{\mu}{\sigma} \circ \QGembed{\sigma}{\lambda}{\mu} 
    = \QGembed{\sigma}{\lambda}{\mu}$
and
$\QGpremb{\lambda}{\mu}{\sigma} \circ \QGembed{\sigma}{\lambda}{\mu}
    = \id_{\QGrep{\sigma}}$
between the projections and embeddings.

Uniqueness (up to multiplicative constants) can be shown
by an induction over~$N$, using the multiplicity-free
branching rule~(\ref{eq:tensor_prod_Uq_mods}).
\end{proof}

\section{Generic Virasoro VOA}%
\label{sec: VOA}
This section introduces the main algebraic structure of the
present work, the generic Virasoro vertex operator algebra.
We also define its modules and intertwining operators between modules.

From the point of view of physics,
vertex operator algebras serve as the chiral algebras of conformal
field theories, and the case of the Virasoro VOA is appropriate for the
case with conformal symmetry alone. The modules of a VOA correspond to the
(conformal families of) fields in the CFT. Intertwining operators
are the building blocks of the correlation functions of these fields.

This section is organized as follows.
In Section~\ref{sub: conventions about formal series}
we introduce notation and fix conventions about formal
series. 
In Section~\ref{sub: Virasoro algebra and the VOA}
we introduce the Virasoro algebra and its highest
weight representations, as well as the VOAs based on them.
Section~\ref{sub: modules and intertwining operators generally}
contains the general definition of modules and intertwining
operators of VOAs, and 
Section~\ref{sub: first row modules and intertwining operators}
concentrates on the specific case of the first row
modules of the generic Virasoro VOA. The specific
result about the fusion rules, in particular, is given in
Section~\ref{sub: first row modules and intertwining operators}.
Much of the topic of this section can be found in textbooks.
To the extent possible, in our presentation we follow
\cite{Lepowsky_Li-VOA} in
Sections~\ref{sub: conventions about formal series}--%
\ref{sub: Virasoro algebra and the VOA},
and \cite{Xu1998,Li99} in
Sections~\ref{sub: conventions about formal series}--%
\ref{sub: modules and intertwining operators generally}.
The more specific fusion rule statement of
Section~\ref{sub: first row modules and intertwining operators}
has been obtained through a different method in~\cite{FrenkelZhu2012}.

\subsection{Some notational conventions}
\label{sub: conventions about formal series}

Let us first fix some notational conventions.

\subsubsection*{General conventions}

When a statement depends on a real number~$m$ 
(integer, natural number, \ldots), we
use the quantifier ``for~$m \gg 0$'' (resp. ``for~$m \ll 0$'') to 
mean that the statement holds for all sufficiently large~$m$ (resp. 
sufficiently small~$m$), i.e., that 
there exists some~$m_0$ such that the statement holds for all $m > m_0$
(resp. for all~$m<m_0$).

\subsubsection*{Formal series}

We will have to consider various types of formal
series: polynomials, Laurent polynomials,
formal power series, formal Laurent series, and 
formal series of yet more general types.
The formal series are formal sum expressions with terms which
are a coefficient times a power of a formal variable.
The coefficients are always taken to lie in some complex
vector space, and consequently also the spaces of 
formal series are naturally vector spaces with addition and scalar
multiplication defined coefficientwise.

Let $V$ be a vector space, and let $\genvar$ be a formal variable.

The \term{space of formal power series} with coefficients in~$V$ is
\begin{align}
V[[\genvar]] = \set{ \sum_{n \in \Znn} v_n \, \genvar^n 
    \; \bigg| \; v_n \in V \text{ for all } n \in \Znn } ,
\end{align}
and the \term{space of polynomials} is the
subspace~$V[\genvar] \subset V[[\genvar]]$
consisting of those formal power
series $\sum_{n \in \Znn} v_n \, \genvar^n$
which only have finitely many non-zero coefficients, i.e.,
$v_n = 0$ for all $n \gg 0$.

Similarly the \term{space of formal Laurent series} with 
coefficients in~$V$ is
\begin{align}
V[[\genvar^{\pm 1}]] = \set{ \sum_{m \in \bZ} v_m \, \genvar^m 
    \; \bigg| \; v_m \in V \text{ for all } m \in \bZ } ,
\end{align}
and \term{space of Laurent polynomials} is the 
subspace~$V[\genvar^{\pm 1}] \subset V[[\genvar^{\pm 1}]]$
consisting of those formal Laurent 
series $\sum_{m \in \bZ} v_m \, \genvar^m$ which only have finitely 
many non-zero coefficients, i.e.,
$v_m = 0$ for $|m| \gg 0$.
The \term{residue} of a formal Laurent series is defined as 
\begin{align*}
\Res_{\genvar} \Big( \sum_{m \in \bZ} v_m \, \genvar^m \Big) = v_{-1} .
\end{align*}

The space of \term{general formal series} with coefficients in~$V$
is \begin{align}
V\{\genvar\} = \set{ \sum_{\alpha \in \bC} v_\alpha \, \genvar^\alpha
    \; \bigg| \; v_\alpha \in V \text{ for all } \alpha \in \bC } .
\end{align}

Elements of any of the above are typically denoted by 
e.g.~$f(\genvar) = \sum_{i} v_i \, \genvar^i$, to explicitly
indicate the formal variable~$\genvar$,
and to emphasize the analogue with functions.

Series with several formal variables are defined by
considering series in one variable with 
coefficients in a vector space of formal series of other variables,
and natural identifications are made without comment:
we set, e.g.,
$V[[\genvar,\genvarB]] = \big( V[[\genvar]] \big) [[\genvarB]] 
= \big( V[[\genvarB]] \big)[[\genvar]]$.

\subsubsection*{Binomial expansion convention}

We follow the commonly used \term{binomial expansion convention}
according to which the power of a binomial in formal variables
is always expanded in
non-negative integer powers of the second variable:
for example if $\genvar, \genvarB$ are two formal variables 
and $\beta \in \bC$, we interpret
\begin{align}\label{eq: binomial expansion}
(\genvar + \genvarB)^\beta
= \; & \sum_{n=0}^\infty {\beta \choose n} \, \genvar^{\beta-n} \, \genvarB^n 
    \; \in \; \bC\{\genvar\}[[\genvarB]] , \\
\nonumber
\text{ where } \qquad
{\beta \choose n}
:= \; & \frac{1}{n!} \, \prod_{j=0}^{n-1} (\beta-j) .
\end{align}
The convention does require some caution:
for instance for $n \in \Zpos$
the two series
\begin{align*}
(\genvar-\genvarB)^{-n} \neq (-\genvarB+\genvar)^{-n} 
\end{align*}
are not equal, but are different Laurent series expansions of the same 
rational function function.
The series on the left is in~$\bC[[\genvar^{\pm 1},\genvarB]]$ and
is convergent in the region~$|\genvar|>|\genvarB|$, while
the series on the right is in~$\bC[[\genvar,\genvarB^{\pm 1}]]$ and
is convergent in the region~$|\genvarB|>|\genvar|$. 
In the case of non-integer~$\beta$,
note that non-integer powers are placed on only one of the formal variables, 
leading to different branch choice issues when specializing the formal 
variables to actual complex values.

\subsubsection*{The formal delta function}

The formal delta-function in the formal variable~$\genvar$ is the 
formal Laurent series
\begin{align}\label{eq: formal delta function}
\delta(\genvar) = \sum_{m \in \bZ} \genvar^m \, \in \, 
\bC[[\genvar^{\pm 1}]] .
\end{align}
If $\genvar, \genvarB, \genvarC$ are formal variables, with the binomial expansion convention we interpret
\begin{align*}
\delta\Big( \frac{\genvar-\genvarB}{\genvarC} \Big)
    = \sum_{m \in \bZ} \sum_{n = 0}^\infty 
        {m \choose n} (-1)^n  \,\genvarC^{-m} \, \genvar^{m-n} \, \genvarB^n
    \; \in \, \bC[[\genvarC^{\pm 1},\genvar^{\pm 1},\genvarB]] .
\end{align*}

\subsubsection*{Multiplication of formal series}

In the case when~$V = A$ is an associative algebra, e.g.,
$V = \bC$ or $V=\End(W)$ for some vector space~$W$, multiplication
of formal series of particular types may be meaningful:
e.g. the product $a(\genvar) \, b(\genvar)$ of 
two formal power series $a(\genvar), b(\genvar) \in A[[\genvar]]$ is 
well-defined
in~$A[[\genvar]]$
(there are finitely many contributions to the coefficient of any~$\genvar^n$),
and the product $f(\genvar) \, g(\genvar)$ of a general series $f(\genvar) \in 
A\{\genvar\}$ with a Laurent polynomial~$g(\genvar) \in A[\genvar^{\pm 1}]$
is well-defined in~$A\{\genvar\}$
(there are finitely many contributions to the 
coefficient of any~$\genvar^\alpha$).

Likewise, suitable formal series with complex coefficients can 
be multiplied with suitable series with coefficients in complex 
vector spaces: e.g., the multiplication of a Laurent polynomial
$r(\genvar) \in \bC[\genvar^{\pm 1}]$ with a formal power series
$h(\genvar) \in V[[\genvar]]$ is well-defined in~$V[[\genvar^{\pm 1}]]$
(there are finitely many contributions to the 
coefficient of any~$\genvar^m$).

Where well-defined, we use any such products without explicit comment
in what follows.

\subsection{Generic Virasoro vertex operator algebra}
\label{sub: Virasoro algebra and the VOA}

\subsubsection*{Virasoro algebra}

The \term{Virasoro algebra} is the complex Lie algebra 
\begin{align*}
\vir = \; \bigoplus_{n\in{\bZ}} \bC L_{n} \; \oplus \; \bC C 
\end{align*}
with the Lie bracket determined by
\begin{align*}
[L_{m},L_{n}] 
= \; & (m-n) \, L_{m+n} + \frac{m^{3}-m}{12} \, \delta_{m+n,0} \, C 
\qquad \text{ for $m,n\in{\bZ}$,} \\  
[C,\vir] = \; & 0 .
\end{align*}
To describe the relevant representations of~$\vir$, we introduce
the Lie subalgebras 
\begin{align*}
\vir_{> 0} := \; & \bigoplus_{n > 0} \bC L_{n} , &
\vir_{< 0} := \; & \bigoplus_{n < 0} \bC L_{n} , \\
\vir_{0} := \; & \bC L_0 \oplus \bC C , &
\vir_{\geq 0} := \; & \vir_{0} \oplus \vir_{>0} .
\end{align*}
The universal enveloping algebra of~$\vir$ is denoted by~$\UEA(\vir)$.

\subsubsection*{Verma module}
For $c,h \in \bC$,
the \term{Verma module}~$\VermaCH{c}{h}$ of \term{central charge}~$c$ and 
\term{conformal weight}~$h$ is defined as
the quotient of the universal enveloping algebra by the left 
ideal generated by the elements~$C-c1$, $L_0-h1$, and~$L_n$ for $n>0$, i.e.
\begin{align}\label{eq: Verma module}
\VermaCH{c}{h} := \UEA(\vir) \, \Big/ \,
    \Big( \UEA(\vir) (C-c1) + \UEA(\vir) (L_0-h1)
    + \UEA(\vir) \, \vir_{>0} \Big) .
\end{align}
By construction, 
the vector $\hwv{c,h} := [1] \in \VermaCH{c}{h}$ satisfies
\begin{align*}
C \hwv{c,h} = \; & c \, \hwv{c,h} , &
L_0 \hwv{c,h} = \; & h \, \hwv{c,h} , &
L_n \, \hwv{c,h} = \; & 0  \quad \text{ for $n>0$},
\end{align*}
and it is a cyclic vector, i.e., it generates the whole representation,
$\UEA(\vir) \hwv{c,h} = \VermaCH{c}{h}$.

Owing to the Poincar\'{e}--Birkhoff--Witt (PBW) theorem, 
as a vector space the Verma module is 
isomorphic to~$\UEA(\vir_{<0})$; a PBW basis for $\VermaCH{c}{h}$
consists of vectors
\begin{align}\label{eq: basis for Verma module}
L_{-n_{k}} \cdots L_{-n_{1}} \, \hwv{c,h} ,
\end{align}
where $k \in \Znn$ and $0 < n_1 \leq n_2 \leq \cdots \leq n_k$.

The central element~$C$ acts as the scalar~$c$ on~$\VermaCH{c}{h}$.
The element $L_0$ is diagonalizable and has eigenvalues $h+d$, $d \in \Znn$,
and we use this to define a \term{grading}
of the Verma module 
\begin{align*}
\VermaCH{c}{h} = 
    \bigoplus_{d = 0}^\infty 
        \Kern \Big( L_0 - (h+d) \, \idof{\VermaCH{c}{h}} \Big) .
\end{align*}
The homogeneous subspaces in this grading are finite-dimensional,
since the basis elements~\eqref{eq: basis for Verma module}
are eigenvectors of~$L_0$, with eigenvalues $h+d$, where
$d = n_1 + \cdots + n_k$.

The Verma module has a filtration associated with the PBW basis,
which we will use extensively.
For each $p\in\Z_{\ge 0}$, we define the subspace
\begin{align}\label{eq: PBW filtration}
\PBWfil{p}\VermaCH{c}{h}
	:= \spn \set{L_{-n_{k}} \cdots L_{-n_{1}} \, \hwv{c,h}
	    \; \Big| \; 0 < n_1 
        \leq \cdots \leq n_k,\ k\le p}
    \; \subset \; \VermaCH{c}{h}
\end{align}
spanned by basis vectors~\eqref{eq: basis for Verma module} with
``PBW word-length'' at most~$p$.
The \term{PBW filtration} is the increasing sequence of subspaces
\begin{equation*}
    \C \hwv{c,h} = \PBWfil{0} \VermaCH{c}{h} 
    \, \subset \, \cdots
    \, \subset \, \PBWfil{p}\VermaCH{c}{h}
    \, \subset \, \PBWfil{p+1}\VermaCH{c}{h}
    \, \subset \, \cdots 
    \, \subset \, \VermaCH{c}{h} ,
\end{equation*}
which clearly has the property
that~$\bigcup_{p \in \Znn} \PBWfil{p}\VermaCH{c}{h} = \VermaCH{c}{h}$.
Note furthermore that each subspace $\PBWfil{p}\VermaCH{c}{h}$ is
itself a 
representation of the Lie subalgebra~$\vir_{\geq 0} \subset \vir$.

\subsubsection*{Highest weight modules}
If $\singvec$ is a non-zero vector in
a representation~$\virgenrep$ of~$\vir$, which
for some~$\eta \in \bC$ and $c \in \bC$
satisfies
\begin{align*}
L_0 \singvec = \eta \singvec  ,
\qquad
C \singvec = c \singvec ,
\qquad \text{ and } \qquad
L_n \singvec = 0 \; \text{ for all $n>0$,}
\end{align*}
then we call~$\singvec \in \virgenrep$ a \term{singular vector}.

If a singular vector~$\singvec \in \virgenrep$ generates the whole
representation,
\begin{align*}
\UEA(\vir) \singvec = \virgenrep , 
\end{align*}
then it is called a \term{highest weight vector}, the
representation~$\virgenrep$ is called a
\term{highest weight representation}, and
the $L_0$-eigenvalue~$\eta$ and the $C$-eigenvalue~$c$ of~$\singvec$
are called its
\term{conformal weight} and \term{central charge},
respectively.
As an example, the Verma module~$\VermaCH{c}{h}$ is a highest weight 
representation and $\hwv{c,h}$ its highest weight vector.
Note that a highest weight vector in a given representation
is necessarily unique up to 
non-zero scalar multiples.

By construction the Verma module~$\VermaCH{c}{h}$
has the universal property that for any highest weight 
representation~$\virgenrep$ with the same central charge~$c$ and highest 
weight~$h$, 
there exists a surjective $\UEA(\vir)$-module 
map
\begin{align*}
\xymatrix{
\VermaCH{c}{h} \ar@{->>}[r] & \virgenrep } .
\end{align*}
As a consequence, any highest 
weight representation is isomorphic to a quotient of a Verma module
by a proper subrepresentation (possibly zero). In particular, 
a highest weight representation~$\virgenrep$ 
with highest weight~$h$ also admits
a $\Znn$-grading by eigenvalues
of~$L_0-h \, \idof{\virgenrep}$, and the homogeneous subspaces
$\Kern \big( L_0 - (h+d) \, \idof{\virgenrep} \big)$
are finite-dimensional.
The PBW filtration of a Verma module is also inherited to its quotient:
letting~$\PBWfil{p}\virgenrep$ denote the image of 
$\PBWfil{p}\VermaCH{c}{h}$ under the above surjection, we obtain
an increasing filtration
\begin{align*}
	\C\singvec = \PBWfil{0} \virgenrep
	\, \subset \, \PBWfil{1} \virgenrep
	\, \subset \, \cdots
	\, \subset \, \PBWfil{p} \virgenrep
	\, \subset \, \PBWfil{p+1} \virgenrep
	\, \subset \, \cdots
	\, \subset \, \virgenrep .
\end{align*}
Since the surjection is a $\UEA (\vir)$-module map, each $\PBWfil{p}\virgenrep$ 
is a representation of $\vir_{\geq 0}$.

The unique irreducible highest weight representation with given $c,h$
is the quotient of the Verma module~$\VermaCH{c}{h}$ by its maximal
proper subrepresentation.
A basic fact is that 
all subrepresentations of Verma modules
are generated by singular vectors,
and that at most two singular vectors are needed
to generate a given 
subrepresentation~\cite{FF-representations},
see also~\cite[Chapter 6]{IK-representation_theory_of_the_Virasoro_algebra}.

The easiest example of a non-trivial submodule of a Verma module
appears when~${h=0}$: the vector $L_{-1}\hwv{c,0}$ is a singular 
vector in the Verma module~$\VermaCH{c}{0}$ 
and thus generates a proper subrepresentation
$\UEA(\vir)L_{-1}\hwv{c,0} \subset \VermaCH{c}{0}$.
This easy case will be relevant for the universal Virasoro vertex 
operator algebra, but we first give other examples that 
are the building blocks of the module category that we study.

\subsubsection*{The Kac table and its first row}

The highest weights~$h$ for which the Verma module~$\VermaCH{c}{h}$
is not irreducible form what is called the Kac table:
these highest weights~$h_{r,s}$ are indexed by two 
positive integers $r,s$.

Let us parametrize central charges~$c$ by~$\kappa$ 
via\footnote{%
This parametrization of the central
charges~$c \leq 1$ by $\kappa>0$ arises naturally
in the context Schramm-Loewner Evolutions (SLE).
It is a $2$-to-$1$ parametrization (except at~$c=1$ and~$\kappa=4$):
the values $\kappa$ and 
$\kappa'=\frac{16}{\kappa}$ give rise to the same~$c$.
In these two cases with the same central charge~$c$ we will
introduce different categories of modules for the same vertex 
operator algebra.
What we call ``the first row of the Kac table'' for $\kappa'$ 
becomes instead ``the first column of the Kac table'' for $\kappa$~---
notice that~$h_{r,s}(\kappa') = h_{s,r}(\kappa)$
in the formula below. Thus with this
$2$-to-$1$ parametrization also the latter category becomes 
covered even if we explicitly only treat the former.}
\begin{align}\label{eq: central charge parametrization}
c = c(\kappa) 
:= 1-\frac{3(\kappa-4)^{2}}{2\kappa}
= 13-6 \Big( \frac{\kappa}{4} + \frac{4}{\kappa} \Big) 
\end{align}
Then an explicit formula for the Kac table highest weights is
\begin{equation*}
h_{r,s}(\kappa):= 
    \frac{16 (s^2-1) + 8 \kappa (1-rs) + \kappa^2 (r^2-1)}{16 \kappa} ,
\qquad \text{ for $r, s \in \Zpos$.}
\end{equation*}
We will be interested in highest weight modules with conformal
weights in the 
first row of the Kac table, i.e., with $r=1$. 
These conformal weights are exactly the ones 
in~\eqref{eq: first row conformal weight ad hoc};
we have
\begin{align}\label{eq: first row highest weights}
\hwFR{\lambda} = h_{1,1+\lambda}(\kappa)=\frac{\lambda(2(\lambda+2)-\kappa)}{2\kappa},
\qquad \text{ for } \lambda\in \Znn .
\end{align}

\begin{lem}[\cite{FF-representations}]
\label{lem: first row generic case}
Assume $\kappa\not\in\mathbb{Q}$ 
and $\lambda \in \Znn$, 
and let~$c$ and~$\hwFR{\lambda}$ be as 
in~\eqref{eq: central charge parametrization}
and~\eqref{eq: first row highest weights}.
Then the maximal proper subrepresentation of the 
Verma module $\VermaCH{c}{\hwFR{\lambda}}$ is 
isomorphic to the Verma module $\VermaCH{c}{\,\hwFR{\lambda}+1+\lambda}$.
In particular, the irreducible highest weight representation
with highest weight~$\hwFR{\lambda}$ is the quotient
\begin{align}
\voaFR{\lambda} := 
    \VermaCH{c}{\hwFR{\lambda}} \, \big/ \, \VermaCH{c}{\,\hwFR{\lambda}+1+\lambda} .
\end{align}
\end{lem}
\begin{proof}
See, e.g., 
\cite{FF-representations}
or~\cite[Section~5.3.1]%
{IK-representation_theory_of_the_Virasoro_algebra}.
With the above assumptions, 
the Verma module $\VermaCH{c}{\hwFR{\lambda}}$
falls in ``class~$I$'' following the terminology 
of Iohara and Koga~\cite{IK-representation_theory_of_the_Virasoro_algebra}.
\end{proof}

The singular vector in $\VermaCH{c}{\hwFR{\lambda}}$
that generates the maximal proper subrepresentation
$\VermaCH{c}{\,\hwFR{\lambda}+1+\lambda} \subset \VermaCH{c}{\hwFR{\lambda}}$
is given by the explicit formula of~\cite{BSA-degenerate_CFTs_and_explicit_expressions},
\begin{align}
\nonumber 
\singvecop_{\lambda} \, \hwv{c,h(\lambda)} \, \in \; & \VermaCH{c}{\hwFR{\lambda}} ,
\qquad \text{where} \\
\label{eq:singular_vector}
\singvecop_{\lambda}
\, := \; & \sum_{k=1}^{\lambda+1}
    \sum_{\substack{p_{1}, \dots, p_{k} \geq 1 \\
        p_{1}+\cdots+p_{k} = \lambda+1}}
    \frac{(-4/\kappa)^{\lambda+1-k} \; (\lambda!)^{2}}{%
        \prod_{u=1}^{k-1}(\sum_{i=1}^{u} p_{i}) 
            (\sum_{i=u+1}^{k} p_{i})}
            L_{-p_{1}} \cdots L_{-p_{k}}
\; \in \; \UEA(\vir) .
\end{align}

\subsubsection*{Virasoro vertex operator algebras}

A vertex operator algebra is an $\Znn$-graded
vector space
\begin{align*}
\voaV = \bigoplus_{d \in \Znn} \voaV_{(d)} 
\end{align*}
equipped with two distinguished 
non-zero vectors,
\begin{align*}
\text{the vacuum vector } \, \voavac \in \voaV_{(0)}
\qquad \text{ and } \qquad 
\text{the conformal vector } \, \voaconfvec \in \voaV_{(2)} ,
\end{align*}
as well as a vertex operator map
\begin{align*}
\voaY(\cdot, \ipvar) \colon \voaV \to \End(\voaV)[[\ipvar^{\pm 1}]] ,
\end{align*}
all subject to axioms which are explicitly given 
in, e.g.,~\cite{Lepowsky_Li-VOA}.

Let $c \in \bC$ be given.
The vector $L_{-1}\hwv{c,0}$ is a singular vector
in the Verma module~$\VermaCH{c}{0}$ of highest weight~$h=0$.
Consider the highest weight representation obtained as the quotient
of the Verma module by the
proper subrepresentation 
$\UEA(\vir)L_{-1}\hwv{c,0} \subset \VermaCH{c}{0}$ generated by
this singular vector,
\begin{equation*}
	\voaV_{c}:=\VermaCH{c}{0} \, \big/ \, \UEA(\vir)L_{-1}\hwv{c,0}.
\end{equation*}
It is known~\cite[Theorem~6.1.5]{Lepowsky_Li-VOA}
that the vector space $\voaV_{c}$ can be equipped
with the structure of a vertex operator algebra (VOA),
uniquely fixed by the following.
The vacuum vector is~$\voavac = [\hwv{c,0}]$, the conformal vector
is $\voaconfvec = L_{-2} \voavac = [L_{-2} \hwv{c,0}]$,
and the Laurent modes of the vertex operator
\begin{align*}
\voaY(\voaconfvec , \ipvar) = \sum_{n \in \bZ} \ipvar^{-2-n} L_{n} 
\; \in \, \End(\voaV_c)[[\ipvar,\ipvar^{-1}]] 
\end{align*}
corresponding to the conformal vector
are the Virasoro generators~$L_n$, $n\in\Z$ (acting as endomorphisms of~$\voaV_c$).
This vertex operator algebra~$\voaV_c$ is known as the 
\term{universal Virasoro vertex operator algebra}
with central charce~$c$.

The irreducible highest weight representation of the Virasoro 
algebra with highest weight~$h=0$ can be also
obtained as the quotient of $\voaV_c$
by its maximal proper submodule. 
A simple vertex operator algebra can be formed
as the corresponding quotient of the 
universal Virasoro VOA~$\voaV_c$~\cite[Theorem~6.1.5]{Lepowsky_Li-VOA},
and we refer to this as the \term{simple Virasoro 
vertex operator algebra}
with central charce~$c$.
For generic central charges, $c = c(\kappa)$ with $\kappa \notin \bQ$,
the representation~$\voaV_c$ 
is in fact already irreducible by itself
(case~$\lambda=0$ of Lemma~\ref{lem: first row generic case}), 
so the simple Virasoro VOA 
coincides with the universal Virasoro VOA~$\voaV_c$.
In this article we focus on the generic case:
specifically we assume 
that $c=c(\kappa)$ as in~\eqref{eq: central charge parametrization} with
\begin{align}\label{eq: genericity}
\kappa \in (0,+\infty) \setminus \bQ .
\end{align}
For clarity, we then call~$\voaV_c$
the \term{generic Virasoro vertex operator algebra},
although it is also both the universal Virasoro VOA 
and the simple Virasoro VOA.
As a warning we already point out that, despite 
being a simple VOA, the generic Virasoro VOA~$\voaV_c$ fails
many key properties assumed in most VOA theory:
$\voaV_c$ is not $C_2$-cofinite, it has infinitely many simple modules,
it has modules which are not completely reducible, etc.

It is instructive to contrast the 
generic case we consider with what happens
for rational central charges~$c$~--- which are relevant, e.g., for minimal 
models of CFT.
For typical rational central charges the representation~$\voaV_c$
is not irreducible,
and correspondingly the universal Virasoro VOA and the simple Virasoro VOA are 
different. The simple Virasoro VOA has been studied extensively in these
cases, and is a model case of a well-behaved VOA:
it is in particular $C_2$-cofinite, 
has finitely many simple modules, 
has semisimple category of modules closed under fusion,
satisfies Verlinde's formula, etc.
The failure of such good properties for 
the generic Virasoro VOA is probably one of the main reasons
the generic Virasoro VOA has not been extensively studied yet.

\subsection{Modules and intertwining operators for VOAs}
\label{sub: modules and intertwining operators generally}

We introduce the notions of modules and intertwining operators for
a general vertex operator algebra $\voaV$.

\subsubsection*{Modules for vertex operator algebras}

A \term{module} for a vertex operator algebra~$\voaV$ is a vector 
space~$\voaM$ equipped with a linear map
\begin{align*}
\modY{\voaM} (\cdot , \ipvar) 
    \colon \voaV \to \End(\voaM)[[\ipvar^{\pm 1}]]
\end{align*}
subject to the following conditions.
For $\Vvec \in \voaV$,
let us denote by~$\Vvec^{\voaM}_{(n)} \in \End(\voaM)$ the coefficient
of~$\ipvar^{-1-n}$ in the formal
series~$\modY{\voaM}(\Vvec , \ipvar)$, so that
\begin{align}\label{eq: modes of module vertex operator}
\modY{\voaM} (\Vvec , \ipvar) 
   = \sum_{m \in \bZ} \ipvar^{-1-m} \, \Vvec^{\voaM}_{(m)} .
\end{align}
For the coefficients of the series
associated to the conformal vector~$\voaconfvec \in \voaV$,
we use the notation
$L^{\voaM}_{n} := \voaconfvec^{\voaM}_{(1+n)}$, $n\in\bZ$ so that
$\modY{\voaM} (\voaconfvec , \ipvar) 
    = \sum_{n \in \bZ} \ipvar^{-2-n} L^{\voaM}_{n} $.
When the module~$\voaM$ is sufficiently clear from the context,
we even omit the superscript and denote simply $L_{n} = L^{\voaM}_{n}$.
The required conditions then read:
\begin{itemize}
\item We have $\modY{\voaM} (\voavac , \ipvar) = \id_{\voaM}$.
\item For any $\Vvec_1 , \Vvec_2 \in \voaV$, the following Jacobi identity 
holds:
\begin{align}
\label{eq:Jacobi_module_map}
& 
    \ipvarO^{-1} \, \delta \Big( \frac{\ipvarI - \ipvarII}{\ipvarO} \Big) \,
    \modY{\voaM} (\Vvec_1 , \ipvarI) \, \modY{\voaM} (\Vvec_2 , \ipvarII) 
- \ipvarO^{-1} \, \delta \Big( \frac{\ipvarII - \ipvarI}{-\ipvarO} \Big) \,
    \modY{\voaM} (\Vvec_2 , \ipvarII) \, \modY{\voaM} (\Vvec_1 , \ipvarI) \\
= \; & \ipvarII^{-1} \, \delta \Big( \frac{\ipvarI - \ipvarO}{\ipvarII} \Big) \,
    \modY{\voaM} \big( \voaY (\Vvec_1 , \ipvarO) \Vvec_2 , \, \ipvarII \big) . \notag
\end{align}
\item The operator~$L^{\voaM}_{0} \in \End(\voaM)$ is diagonalizable,
its eigenspaces
$\voaM_{(\eta)} := \Kern \big( L^{\voaM}_{0} - \eta \, \id_{\voaM} \big)$ 
are finite-dimensional, and 
$\voaM_{(\eta)} = \set{0}$ when~$\re(\eta) \ll 0$.
\end{itemize}
A module~$\voaM$ thus has an $L^{\voaM}_0$-eigenspace decomposition
\begin{align}\label{eq: VOA module eigenspace decomposition}
\voaM = \bigoplus_{\eta \in \bC} \voaM_{(\eta)} .
\end{align}
The coefficients of the module
vertex operator~\eqref{eq: modes of module vertex operator}
respect this decomposition in the following sense.
\begin{lem}
\label{lem: homogeneity of the modes in modules}
For any homogeneous element $\Vvec \in \voaV_{(d)}$ of the
vertex operator algebra and any~$m \in \bZ$ and $\eta \in \bC$ we have
\begin{align*}
\Vvec^{\voaM}_{(m)} \, \voaM_{(\eta)} \subset \voaM_{(\eta-m+d-1)} .
\end{align*}
\end{lem}
\begin{proof}
This is a standard result; it follows from 
\begin{align*}
	[L_{0},\Vvec^{\voaM}_{(m)}]=(-m+d-1)\Vvec^{\voaM}_{(m)},\quad m\in \bZ ,
\end{align*}
which in turn follows from 
applying the Jacobi identity~(\ref{eq:Jacobi_module_map})
in the case that $\Vvec_{1}=\voaconfvec$ and $\Vvec_{2}=\Vvec$.
(Similar calculations will be
done in some more detail in a more general case
in Lemma~\ref{lem: intertwining property with rational functions}
and Corollary~\ref{cor:formula_intertw}, for example.)
\end{proof}

\subsubsection*{Contragredient modules}

Suppose that~$\voaM$ is a module for a vertex 
operator algebra~$\voaV$.
Considering the duals~$\voaM_{(\eta)}^* = \Hom(\voaM_{(\eta)},\bC)$ 
of the finite-dimensional
$L^\modM_{0}$-eigenspaces~$\voaM_{(\eta)}$ in the
decomposition~\eqref{eq: VOA module eigenspace decomposition},
the restricted dual (also called graded dual) is the space
\begin{align}
\voaM' = \bigoplus_{\eta \in \bC} \voaM_{(\eta)}^* .
\end{align}
It can be equipped with the structure of a
$\voaV$-module \cite{FrenkelHuangLepowsky1993, Xu1998}
so that the module vertex operator
$\modY{\voaM'}( \cdot , \ipvar)$ on $\voaM'$ is defined by
the formula
\begin{align*}
\Braket{\modY{\voaM'}(\Vvec,\ipvar) \Mvec' , \Mvec}
= \Braket{ \Mvec' , \modY{\voaM}(e^{\ipvar L_{1}}
   (-\ipvar^{-2})^{L_{0}} \Vvec , \ipvar^{-1}) \Mvec}, & \\
\text{ for } \Vvec \in \voaV, \ \Mvec' \in \voaM', \ \Mvec \in \voaM . &
\end{align*}
This $\voaV$-module~$\voaM'$ is called the \term{contragredient module} 
of~$\voaM$. Double contragredients are isomorphic
to the original modules, $\voaM'' \isom \voaM$.

\begin{lem}
\label{lem: Virasoro action on the contragredient}
We have $L_{n}^{\voaM'} = (L_{-n}^{\voaM})^\top$ for any $n \in \bZ$,
i.e., for any $\Mvec' \in \voaM'$ and $\Mvec \in \voaM$ we have
\begin{align*}
\langle {L_{n}^{\voaM'} \, \Mvec' , \Mvec } \rangle
= \langle {\Mvec' , L_{-n}^{\voaM} \, \Mvec } \rangle .
\end{align*}
\end{lem}
\begin{proof}
Taking $\Vvec = \voaconfvec$ and using
$e^{\ipvar L_{1}} (-\ipvar^{-2})^{L_{0}} \voaconfvec = \ipvar^{-4} \, \voaconfvec$,
this is just the equality of the coefficients of $\ipvar^{-2-n}$
in the defining formula above.
\end{proof}

\subsubsection*{Intertwining operators}

As an analogy to that the tensor product of modules over a Hopf algebra is governed by the coproduct structure,
the {\it fusion product} of modules of a VOA $\voaV$ is governed by the notion of intertwining operators.

\begin{defn}
Let $\modMmid, \modMin, \modMout$ be three modules for a VOA~$\voaV$,
with respective module vertex operators $\modY{\modMmid} (\cdot, \ipvar), 
\modY{\modMin} (\cdot, \ipvar) , 
\modY{\modMout} (\cdot, \ipvar)$.
An intertwining operator of 
type~$\voaFusion{\modMmid}{\modMin}{\modMout}$
is a linear map
\begin{align*}
\intertw (\cdot,\apvar) \colon
    \modMmid \to \Hom( \modMin , \modMout) \{\apvar\}
\end{align*}
satisfying the Jacobi identity
\begin{align}\label{eq:intertw_assoc}
& \ipvarB^{-1} \,
        \delta\left(\frac{\ipvar-\apvar}{\ipvarB}\right) \,
    \modY{\modMout} (\Vvec, \ipvar) \,
    \intertw(\Mvec , \apvar)
- \ipvarB^{-1} \,
        \delta\left(\frac{\apvar-\ipvar}{-\ipvarB}\right) \,
    \intertw (\Mvec , \apvar) \,
    \modY{\modMin}(\Vvec , \ipvar) \\
\notag
= \; & \apvar^{-1} \,
        \delta\left(\frac{\ipvar-\ipvarB}{\apvar}\right)
    \intertw \Big( \modY{\modMmid}(\Vvec , \ipvarB) \Mvec 
        , \, \apvar \Big)
\end{align}
and the translation property
\begin{equation}
\label{eq:intertw_translation}
\intertw \big( L_{-1}^{\modMmid} \Mvec, \apvar \big)
= \der{\apvar} \intertw (\Mvec , \apvar)
\end{equation}
for all $\Vvec \in \voaV$ and $\Mvec \in \modMmid$.	
\end{defn}

The space of intertwining operators of 
type~$\voaFusion{\modMmid}{\modMin}{\modMout}$ is a vector
space, which we will denote by
\begin{align*}
\spintertw \voaFusion{\modMmid}{\modMin}{\modMout} . 
\end{align*}
It is known that spaces of intertwining operators
admit the following symmetry.
\begin{prop}[\cite{Huang--LepowskyII}]
\label{prop:symmetry_fusion_rules}
Let $\modMin$, $\modMmid$, and $\modMout$
be modules of a VOA~$\voaV$. Then, we have linear isomorphisms
\begin{equation*}
	\spintertw \voaFusion{\modMmid}{\modMin}{\modMout}
\; \simeq \; \spintertw \voaFusion{\modMin}{\modMmid}{\modMout}
\; \simeq \; \spintertw \voaFusion{\modMmid}{\modMout'}{\modMin'}.
\end{equation*}
\end{prop}

The next lemma is useful for calculations with intertwining operators,
and it is in fact equivalent to 
the property~\eqref{eq:intertw_assoc}.
\begin{lem}\label{lem: intertwining property with rational functions}
Let $\intertw (\cdot, \apvar)$ be an intertwining operator of type 
$\voaFusion{\modMmid}{\modMin}{\modMout}$.
Then, for any $p,q \in{\bZ}$,
$\Vvec \in \voaV$, $\Mvec \in \modMmid$, we have
\begin{align*}
& \Res_{\ipvar} \Big(
    \modY{\modMout}(\Vvec,\ipvar) \,
    \intertw(\Mvec , \apvar) \;
    \ipvar^{p} \,
    (\ipvar-\apvar)^{q} 
    \Big)
  - \Res_{\ipvar} \Big(
    \intertw(\Mvec , \apvar) \,
    \modY{\modMin}(\Vvec,\ipvar) \;
    \ipvar^{p} \,
    (-\apvar+\ipvar)^{q} 
    \Big) \\
= \; & \Res_{\ipvarO} \Big(
    \intertw \big( \modY{\modMmid}(\Vvec , \ipvarO) \Mvec , \apvar \big) \;
    (\apvar + \ipvarO)^{p} \,
    \ipvarO^{q} 
    \Big) .
\end{align*}
\end{lem}
\begin{proof}
Take the terms in~\eqref{eq:intertw_assoc} proportional 
to~$\ipvarB^{-q-1}$ to obtain
\begin{align*}
& \modY{\modMout} (\Vvec , \ipvar)
    \intertw (\Mvec , \apvar)
    (\ipvar-\apvar)^{q}
  - \intertw (\Mvec , \apvar)
    \modY{\modMin} (\Vvec , \ipvar)
    (-\apvar+\ipvar)^{q} \\
= \; & 
    \Res_{\ipvarB} \bigg( \ipvarB^{q} \apvar^{-1} \,
    \delta \Big(\frac{\ipvar-\ipvarB}{\apvar} \Big) \,
    \intertw \big( \modY{\modMmid} (\Vvec , \ipvarB) \Mvec , \apvar \big) 
    \bigg) \\
= \; & \sum_{m\in{\bZ}} \sum_{k=0}^{\infty}
    (-1)^{k} \binom{m}{k} \ipvar^{m-k} \apvar^{-m-1}
    \intertw (\Vvec_{(q+k)} \Mvec,\apvar) .
\end{align*}
We further multiply $\ipvar^{p}$ and take the residue with respect to $\ipvar$ 
to obtain
\begin{align*}
& \Res_{\ipvar} \Big( \modY{\modMout} (\Vvec , \ipvar) \,
    \intertw (\Mvec , \apvar) \;
    \ipvar^{p} 
    (\ipvar-\apvar)^{q} 
    \Big)
 - \Res_{\ipvar} \Big(
    \intertw (\Mvec , \apvar) \,
    \modY{\modMin}(\Vvec , \ipvar) \;
    \ipvar^{p}
    (\ipvar-\apvar)^{q} 
    \Big) \\
= \; & \sum_{k=0}^{\infty} \binom{p}{k} \apvar^{p-k}
    \intertw (\Vvec_{(q+k)} \Mvec , \apvar) \\
= \; & \Res_{\ipvarO} \Big(
    \intertw \big( \modY{\modMmid} (\Vvec , \ipvarO) \Mvec , \, \apvar \big) \;
    (\apvar+\ipvarO)^{p} \ipvarO^{q} 
    \Big) .
\end{align*}
This proves the assertion.
\end{proof}

From the above formulas, we get the following equations 
for intertwining operators, which give the basis of many
recursive constructions with respect to the PBW
filtration~\eqref{eq: PBW filtration} that we will use.
\begin{cor}
\label{cor:formula_intertw}
Let $\intertw (\cdot, \apvar)$ be an intertwining operator of type 
$\voaFusion{\modMmid}{\modMin}{\modMout}$.
\begin{enumerate}
\item	For any $\Mvec \in \modMmid$, $n\in{\bZ}$, we have
\begin{align*}
& \intertw (L_{n}^{\modMmid} \, \Mvec, \, \apvar) \\
= & \; \Res_{\ipvar} \Big(
    \modY{\modMout} (\voaconfvec , \ipvar) \,
    \intertw (\Mvec , \apvar) \,
	(\ipvar-\apvar)^{n+1} 
  - 
    \intertw (\Mvec , \apvar) \,
    \modY{\modMin}(\voaconfvec , \ipvar) \,
    (-\apvar+\ipvar)^{n+1} \Big) \\
= \; & \sum_{k=0}^{\infty} 
    \binom{n+1}{k} \, (-\apvar)^{k} \,
    L_{n-k}^{\modMout} \, \intertw (\Mvec , \apvar)
  - \sum_{k=0}^{\infty}
    \binom{n+1}{k} \, (-\apvar)^{n-k+1} \,
    \intertw (\Mvec , \apvar) \, L_{k-1}^{\modMin}
    .
\end{align*}
\item	For any $\Mvec \in \modMmid$, $n\in{\bZ}$, we have
\begin{align*}
L_{n}^{\modMout} \, \intertw (\Mvec , \apvar)
- \intertw (\Mvec , \apvar) \, L_{n}^{\modMin}
= \; & \sum_{k=0}^{\infty}
    \binom{n+1}{k} \, \apvar^{n-k+1} \,
    \intertw (L_{k-1}^{\modMmid} \, \Mvec , \apvar)
    .
\end{align*}
\item 	For any $\Mvec \in \modMmid$, we have
\begin{align*}
L_{-1}^{\modMout} \, \intertw (\Mvec , \apvar)
- \intertw (\Mvec , \apvar) \, L_{-1}^{\modMin}
= \der{\apvar} \intertw (\Mvec,\apvar).
\end{align*}
\end{enumerate}
\end{cor}
\begin{rmk}
From here on, we will not use as careful notation as above
to indicate in which modules the different $L_n$'s act.
We will instead abuse the notation slightly and write
\begin{align*}
L_{n}^{\modMout} \, \intertw (\Mvec , \apvar)
- \intertw (\Mvec , \apvar) \, L_{n}^{\modMin}
= \; & \big[ L_{n} , \intertw (\Mvec,\apvar) \big]
\end{align*}
in, e.g., parts (2) and (3) above.
\end{rmk}
\begin{proof}[Proof of Corollary~\ref{cor:formula_intertw}]
Recall that $L_{n}=\voaconfvec_{(n+1)}$, $n\in{\bZ}$.
By setting $\Vvec=\voaconfvec$, $p=0$, $q=n+1$, we see (1).
By setting $\Vvec=\voaconfvec$, $p=n+1$, $q=0$, we see (2).
In particular the latter gives
\begin{align*}
\big[ L_{-1} , \intertw (\Mvec , \apvar) \big]
= \; & \intertw (L_{-1} \Mvec , \apvar) ,
\end{align*}
and the right hand side here
equals $\der{\apvar} \intertw (\Mvec,\apvar)$ by
the translation property~(\ref{eq:intertw_translation}) of intertwining operators.
\end{proof}

\subsection{Modules and intertwining operators for the generic Virasoro VOA}
\label{sub: first row modules and intertwining operators}
We now discuss modules and intertwining operators focusing on 
the case of the generic Virasoro VOA $\voaV_{c}$,
and describe the subcategory of modules that is the topic of this article.
We continue to parametrize the central charge~$c\leq 1$ by $\kappa>0$
as in~\eqref{eq: central charge parametrization}.
Throughout we make the genericity assumption that~$\kappa \notin \bQ$.

\subsubsection*{Modules}
Suppose that $\virgenrep$ is a
representation of the Virasoro algebra
where~$C$ acts as $c \, \id_{\virgenrep}$, and
$L_0$ acts diagonalizably with finite-dimensional eigenspaces
and eigenvalues with real part bounded from below.
Then~$\virgenrep$ 
has a unique structure of a module 
for the universal Virasoro VOA~$\voaV_{c}$ such that
\begin{align*}
\modY{\virgenrep}(\voaconfvec , \ipvar) 
    = \sum_{n \in \bZ} \ipvar^{-n-2} L_{n} , 
\end{align*}
i.e., $\voaconfvec^{\virgenrep}_{(m)} = L_{m-1} \in \End(\virgenrep)$,
see, e.g.,~\cite[Theorem~6.1.7]{Lepowsky_Li-VOA}.

In particular any of the Verma 
modules~$\VermaFR{\lambda}=\VermaCH{c}{\hwFR{\lambda}}$, $\lambda\in \Znn$,
and their irreducible quotient representations
\[
\voaFR{\lambda} = 
\VermaFR{\lambda} \, \big/ \, \VermaCH{c}{\, \hwFR{\lambda}+1+\lambda} ,
    \qquad \lambda\in \Znn
\]
in Lemma~\ref{lem: first row generic case} becomes a module for the
generic Virasoro VOA~$\voaV_c$. 
We call these $\voaFR{\lambda}$, $\lambda\in\Znn$ \term{first row modules},
and we denote the module vertex operators in them simply by
\begin{align*}
\modY{\lambda} (\cdot , \ipvar) 
    \colon \voaV_c \to \End(\voaFR{\lambda}) [[\ipvar^{\pm 1}]] .
\end{align*}

Irreducible highest weight representations are 
self-dual in the following sense.
\begin{lem}
\label{lem: self duality for irreducible hwreps}
Let $\virgenrep$ be an irreducible highest weight representation
of Virasoro algebra, viewed as a module for the
universal Virasoro VOA~$\voaV_{c}$.
Then the contragradient module $\virgenrep^{\prime}$ is 
isomorphic to~$\virgenrep$.	
\end{lem}
\begin{proof}
Let $\hwvsym \in \virgenrep_{(h)}$ be a highest weight vector
in~$\virgenrep$, and choose $\hwvsym' \in \virgenrep_{(h)}^*$
such that $\ldual \hwvsym' , \hwvsym \rdual = 1$.
From Lemma~\ref{lem: Virasoro action on the contragredient}
we see that $\hwvsym' \in \virgenrep'$ is 
a singular vector with highest weight~$h$ in~$\virgenrep'$.
Then, $\virgenrep'$ contains the subrepresentation~$\UEA(\vir) \hwvsym' \subset \virgenrep'$,
which is a highest weight representation with the same highest weight~$h$.
Since $\virgenrep$ is irreducible, there is a surjective module homomorphism $\UEA (\vir)\hwvsym^{\prime} \twoheadrightarrow \virgenrep$,
which in particular implies that $\dmn ((\UEA (\vir)\hwvsym^{\prime})_{(\eta)})\geq \dmn (\virgenrep_{(\eta)})$ for all $\eta\in \bC$.
On the other hand, by construction, we have $\dmn(\virgenrep'_{(\eta)}) = \dmn(\virgenrep_{(\eta)})$, $\eta\in \bC$.
Therefore we get that
\begin{align*}
	\dmn (\virgenrep^{\prime}_{(\eta)})=\dmn (\virgenrep_{(\eta)}) \leq \dmn ((\UEA (\vir)\hwvsym^{\prime})_{(\eta)}) \leq \dmn (\virgenrep^{\prime}_{(\eta)}),
\qquad \text{ for } \eta\in \bC,
\end{align*}
and we can 
conclude that $\UEA (\vir)\hwvsym^{\prime}=\virgenrep^{\prime}\simeq \virgenrep$.
\end{proof}

\begin{cor}\label{cor: self duality of FRmod}
The first-row modules are self-dual, 
$\voaFR{\lambda}' \isom \voaFR{\lambda}$ for any~$\lambda \in \Znn$.
\end{cor}

\subsubsection*{Intertwining operators among highest weight modules}

Let us now suppose that $\modMmid$, $\modMin$, $\modMout$
are three modules for the universal Virasoro VOA~$\voaV_{c}$,
and that each of $\modMmid$, $\modMin$, and~$\modMout'$
are highest weight modules
(note also that by Lemma~\ref{lem: self duality for irreducible hwreps},
the contragredient  $\modMout'$ is a highest weight module for example if
$\modMout$ itself is an irreducible highest weight module).
Let $\hwvmid$, $\hwvin$, $\hwvout'$ be highest weight vectors in
$\modMmid$, $\modMin$, $\modMout'$, respectively,
and denote the corresponding highest weights 
by~$\hmid, \hin, \hout$.

For an intertwining operator
operator~$\intertw \in \spintertw \voaFusion{\modMmid}{\modMin}{\modMout}$,
define
\begin{align*}
\initterm{\intertw} (\apvar) 
:= \braket{ \hwvout' , \; \intertw(\hwvmid , \apvar) \, \hwvin } 
\end{align*}
and call this the 
\term{initial term} of the intertwining operator~$\intertw$.
The initial term defines a linear map
\begin{equation*}
\inittermsym \colon \;
    \spintertw \voaFusion{\modMmid}{\modMin}{\modMout} \to \, {\bC}\{ \apvar \} .
\end{equation*}
We make a few simple general observations in this setup.

\begin{lem}\label{lem: the initial term is a power}
The initial term
of any $\intertw \in \spintertw \voaFusion{\modMmid}{\modMin}{\modMout}$
is of the form
\begin{align*}
\initterm{\intertw} (\apvar) 
= A \, \apvar^{\hout - \hmid - \hin} ,
\qquad \text{ for some $A \in \bC$.}
\end{align*}
\end{lem}
\begin{proof}
Using Corollary~\ref{cor:formula_intertw}(2) for $n = 0$,
Lemma~\ref{lem: Virasoro action on the contragredient}, and
the eigenvector equations
$L_0 \hwvmid = \hmid \, \hwvmid$,
$L_0 \hwvin = \hin \, \hwvin$,
$L_0 \hwvout' = \hout \, \hwvout'$,
we calculate
\begin{align*}
\hout \; \initterm{\intertw} (\apvar)
= \; & \braket{ L_0 \hwvout' , \; \intertw(\hwvmid , \apvar) \, \hwvin } \\
= \; & \braket{ \hwvout' , \; L_0 \, \intertw(\hwvmid , \apvar) \, \hwvin } \\
= \; & \braket{ \hwvout' , \; 
        \big( \apvar \, \intertw(L_{-1} \hwvmid , \apvar)
            + \intertw(L_{0} \hwvmid , \apvar) \big) \, \hwvin } \\
& 
    + \braket{ \hwvout' , \; \intertw(\hwvmid , \apvar) \, L_0 \hwvin } \\
= \; & \Big( \apvar \, \der{\apvar} + \hmid + \hin \Big) \; 
        \initterm{\intertw} (\apvar) .
\end{align*}
Thus the initial term satisfies
$\apvar \der{\apvar} \initterm{\intertw} (\apvar) 
= (\hout - \hmid - \hin) \, \initterm{\intertw} (\apvar)$.
The solution space of this differential equation is one dimensional
and spanned by~$\apvar^{\hout - \hmid - \hin}$, so the assertion 
follows.
\end{proof}

The following straightforward proposition contains a key idea,
that of an induction based on PBW-filtrations, so we do it in
detail here. In later proofs (including our main results), 
we will then not always write out explicitly all of the cases,
as the ideas are similar.
\begin{prop}
\label{prop: uniqueness of intertwiners up to constant}
An intertwining operator~%
$\intertw \in \spintertw \voaFusion{\modMmid}{\modMin}{\modMout}$
is uniquely determined by its initial term
$\initterm{\intertw} (\apvar)$.
In particular,
by Lemma~\ref{lem: the initial term is a power} we thus have
\begin{align*}
\dmn \; \spintertw \voaFusion{\modMmid}{\modMin}{\modMout} \; \leq \; 1 ,
\end{align*}
i.e., if a non-zero intertwining operators exists, it is unique 
up to a multiplicative constant.
\end{prop}
\begin{proof}

Suppose that $\initterm{\intertw}(\apvar) = 0$. 
We will prove that then
$\braket{ \Mvecout' , \; \intertw(\Mvecmid , \apvar) \, \Mvecin } = 0$
for all $\Mvecout' \in \modMout'$, $\Mvecmid \in \modMmid$, 
$\Mvecin \in \modMin$. It will follow that $\intertw = 0$.
From this we can conclude that the initial term indeed
determines the interwining operator.

The proof of the above is done by induction with respect to 
the total PBW word length
for the PBW filtrations of the
three highest weight representations
$\modMmid$, $\modMin$, $\modMout^{\prime}$.
So assume that
$\braket{ \Mvecout' , \; \intertw(\Mvecmid , \apvar) \, \Mvecin } = 0$
whenever $\Mvecout' \in \PBWfil{p_1} \modMout'$, 
$\Mvecmid \in \PBWfil{p_2} \modMmid$, 
$\Mvecin \in \PBWfil{p_3} \modMin$ are such that
the total word lengths satisfy $p_1 + p_2 + p_3 \leq p$.
Now let $n>0$, and consider any such $\Mvecout', \Mvecmid, \Mvecin$.
From Corollary~\ref{cor:formula_intertw}(2) we get that 
\begin{align*}
\braket{L_{-n} \Mvecout', \intertw (\Mvecmid , \apvar) \Mvecin}
= \; & \braket{\Mvecout' , 
        \intertw (\Mvecmid , \apvar) L_{n} \Mvecin }
    + \apvar^{n+1} \der{\apvar}
	    \braket{\Mvecout' , \intertw (\Mvecmid , \apvar) \Mvecin} \\
&   + \sum_{k=1}^{\infty} \binom{n+1}{k} \apvar^{n-k+1} \,
        \braket{\Mvecout' , 
        \intertw (L_{k-1} \Mvecmid,\apvar) \Mvecin } \\
= \; & 0 ,
\end{align*}
where each term of the second expression vanished by the 
induction hypothesis
(note that also $L_{n} \Mvecin \in \PBWfil{p_3-1} \modMin$
and $L_{k-1} \Mvecmid \in \PBWfil{p_2} \modMmid$ above).
Similarly we get
\begin{align*}
\braket{\Mvecout' , \intertw (\Mvecmid , \apvar) L_{-n} \Mvecin}
= \; & \braket{L_{n} \Mvecout' ,
        \intertw(\Mvecmid , \apvar) \Mvecin}
  - \apvar^{-n+1} \der{\apvar} \;
    \braket{\Mvecout' , \intertw(\Mvecmid , \apvar) \Mvecin } \\
& - \sum_{k=1}^{\infty} \binom{-n+1}{k} \apvar^{-n-k+1} \;
    \braket{\Mvecout' , 
        \intertw (L_{k-1}\Mvecmid , \apvar) \Mvecin } \\
= \; & 0 .
\end{align*}
Finally, using Corollary~\ref{cor:formula_intertw}(1--3), we
similarly get
\begin{align*}
\braket{\Mvecout' , \intertw(L_{-n}\Mvecmid , \apvar) \Mvecin}
= \; & \sum_{k=0}^{\infty} \binom{-n+1}{k} (-\apvar)^{k} \;
    \braket{L_{n+k} \Mvecout' , \intertw(\Mvecmid , \apvar) \Mvecin} \\
& - \sum_{k=1}^{\infty} \binom{-n+1}{k} (-\apvar)^{-n-k+1} \;
    \braket{ \Mvecout' , 
        \intertw (\Mvecmid , \apvar) L_{k-1} \Mvecin} \\
& - (-\apvar)^{-n+1} \, \braket{ L_{1} \Mvecout' , 
        \intertw (\Mvecmid,\apvar) \Mvecin} \\
& + (-\apvar)^{-n+1} \der{\apvar} \,
    \braket{\Mvecout' ,\intertw (\Mvecmid , \apvar) \Mvecin } \\
= \; & 0 .
\end{align*}
These three cases complete the induction step,
by establishing that
$\braket{ \Mvecout' , \; \intertw(\Mvecmid , \apvar) \, \Mvecin }$
vanishes also whenever $\Mvecout' \in \PBWfil{p_1} \modMout'$, 
$\Mvecmid \in \PBWfil{p_2} \modMmid$, 
$\Mvecin \in \PBWfil{p_3} \modMin$ are such that
the total word lengths satisfy $p_1 + p_2 + p_3 \leq p+1$.
\end{proof}

\subsubsection*{Intertwining operators among first row modules}

To find all intertwining operators among the
first row modules~$\voaFR{\lambda}$, in view of 
Proposition~\ref{prop: uniqueness of intertwiners up to constant}
it suffices to determine when non-zero intertwining operators
can exist. We call the conditions for the existence
\term{selection rules}.\footnote{These selection rules
turn out to take exactly the same form as the
selection rules which determine when an irreducible 
representation of the quantum group~$\Uqsltwo$ appears in
the tensor product of two others.}

The 
singular vectors~\eqref{eq:singular_vector}
lead to necessary conditions.
Fix~$\lambda \in \Znn$, and introduce the corresponding
polynomial of $\hout, \hin$
\begin{align}\label{eq: selection polynomial two var}
& P_\lambda(\hin,\hout) \\ 
\nonumber
:= \; &
    \sum_{k=1}^{\lambda+1}
    \sum_{\substack{p_{1}, \dots, p_{k} \geq 1 \\
        p_{1}+\cdots+p_{k} = \lambda+1}}
    \frac{(-4/\kappa)^{\lambda+1-k} \; (\lambda!)^{2}}{%
        \prod_{u=1}^{k-1}(\sum_{i=1}^{u} p_{i})
            (\sum_{i=u+1}^{k} p_{i})}
        \prod_{j = 1}^k (-1)^{p_j}
            \Big( \sum_{i>j} p_i + \hwFR{\lambda} - \hout + p_j \hin \Big) .
\end{align}

\begin{lem}\label{lem: necessary selection rule}
Suppose that $\modMin$, $\modMout$
are modules for~$\voaV_{c}$ such that
$\modMin$ and~$\modMout'$
are highest weight modules
with highest weights~$\hin$ and $\hout$, respectively.
For each $\lambda\in \Znn$, the linear map
\begin{equation*}
	\spintertw \voaFusion{\voaFR{\lambda}}{\modMin}{\modMout}
\to \spintertw \voaFusion{\VermaFR{\lambda}}{\modMin}{\modMout} ;
	\qquad \intertw (\cdot,\apvar)\mapsto \intertw(\pi_{\lambda}(\cdot), \apvar),
\end{equation*}
where $\pi_{\lambda}:\VermaFR{\lambda}\to \voaFR{\lambda}$ is the canonical projection, is an embedding of the space of intertwining operators.
Furthermore, this embedding is an isomorphism if and only if the highest weights satisfy the polynomial equation
\begin{align*}
P_\lambda (\hin, \hout) = 0.
\end{align*}
\end{lem}
\begin{proof}
It is clear that for any
$\intertw \in \spintertw \voaFusion{\voaFR{\lambda}}{\modMin}{\modMout}$,
the formula $\intertw(\pi_{\lambda}(\cdot), \apvar)$ indeed
gives an intertwining operator of type
$\voaFusion{\VermaFR{\lambda}}{\modMin}{\modMout}$, so the
map is well-defined.

For its injectivity, observe the following.
The initial terms satisfy $\initterm{\intertw(\cdot,\apvar)}=\initterm{\intertw(\pi_{\lambda}(\cdot),\apvar)}$.
Now if $\intertw (\pi_{\lambda}(\cdot),\apvar)=0$, we have $\initterm{\intertw(\cdot,\apvar)}=0$, which
by Proposition~\ref{prop: uniqueness of intertwiners up to constant}
implies that also $\intertw (\cdot,\apvar)=0$.
Injectivity follows.

The embedding is an isomorphism if and only if
all intertwining operators $\intertw \in \spintertw \voaFusion{\VermaFR{\lambda}}{\modMin}{\modMout}$
factor through the irreducible quotient
in the sense that we have $\intertw (\Mvec, \apvar)=0$
for all $\Mvec$ in the maximal proper submodule of $\VermaFR{\lambda}$.
As in Proposition~\ref{prop: uniqueness of intertwiners up to constant},
we see that the factorization is equivalent to the
single condition 
$\braket{\hwvout' , \, \intertw( \singvecop_\lambda \hwv{c,h(\lambda)} , \, \apvar) \, \hwvin} = 0$,
where $\singvecop_{\lambda} \hwv{c,h(\lambda)}$ is the singular vector defined
in~\eqref{eq:singular_vector}; 
just note that the singular vector 
generates
the maximal proper submodule of $\VermaFR{\lambda}$, which
is isomorphic to $\VermaCH{c}{\hwFR{\lambda}+\lambda+1}$.

Let $\hwvin$ and $\hwvout'$ be highest weight vectors 
in~$\modMin$ and $\modMout'$, respectively, and 
let~$\hwvsym_{\lambda}$ be a highest weight vector in~$\VermaFR{\lambda}$.
Denote the initial term of $\intertw$ by
\begin{align*}
f(\apvar) = \initterm{\intertw}(\apvar) =
\braket { \hwvout' , \, \intertw (\hwvsym_\lambda , \apvar) \, \hwvin} . 
\end{align*}
By Lemma~\ref{lem: the initial term is a power}, we must have
\begin{align*}
f(\apvar) = A \, \apvar^{\Delta} ,
\qquad \text{ where} \qquad
\Delta = \hout - \hin - \hwFR{\lambda} .
\end{align*}
Again by a calculation using the formulas of 
Corollary~\ref{cor:formula_intertw}
(in fact the special case $\Mvecout'=\hwvout'$ and $\Mvecin = \hwvin$ 
of the last calculation in the proof of
Proposition~\ref{prop: uniqueness of intertwiners up to constant}),
we get for any~$n>0$
\begin{align*}
& \braket{\hwvout' , \intertw(L_{-n} \Mvecmid , \apvar) \hwvin} \\
= \; & 
    (-\apvar)^{-n+1} \der{\apvar} \,
    \braket{\hwvout' ,\intertw (\Mvecmid , \apvar) \hwvin } 
  - (1-n) \hin (-\apvar)^{-n} \;
    \braket{ \hwvout' , 
        \intertw (\Mvecmid , \apvar) \hwvin} \\
= \; & \dopL{-n} \, \braket{ \hwvout' , 
        \intertw (\Mvecmid , \apvar) \hwvin} ,
\end{align*}
where we introduced the differential operator
\begin{align*}
\dopL{-n} = \; &
    (-\apvar)^{-n} \Big( (n-1) \hin - \apvar \der{\apvar} \Big) .
\end{align*}
Recursively used, this allows to reduce the following expression to
a differential operator acting on the initial term
\begin{align}\label{eq: acting inside the intertwining operator}
\braket{\hwvout' , 
    \intertw(L_{-n_{k}} \cdots L_{-n_{1}} \hwvsym_\lambda , \apvar)
            \hwvin} 
= \; & \dopL{-n_{k}} \cdots \dopL{-n_{1}}
    \braket{ \hwvout' , 
        \intertw (\hwvin , \apvar) L_{k-1} \hwvin} \\
\nonumber
= \; & \dopL{-n_{k}} \cdots \dopL{-n_{1}} \; f(\apvar) .
\end{align}
With induction on~$k$,
starting from~$f(\apvar) = A \, \apvar^{\Delta}$ and
using the explicit differential operators~$\dopL{-n}$, we find
\begin{align*}
\dopL{-n_{k}} \cdots \dopL{-n_{1}} \; f(\apvar)
= \; & A \, \apvar^{\Delta - \sum_j n_j} \,
        \prod_{j = 1}^k (-1)^{p_j} 
        \Big( \sum_{i<j} n_i - \Delta + (n_j - 1) \hin \Big)
\end{align*}

From the formula~\eqref{eq:singular_vector} for the
singular vector~$\singvecop_\lambda \hwv{c,h(\lambda)}$,
using~\eqref{eq: acting inside the intertwining operator}, we get
\begin{align*}
\braket{\hwvout' , \intertw( \singvecop_\lambda \hwvmid  , \apvar) \hwvin} 
= \; & \sum_{k=1}^{\lambda+1}
    \sum_{\substack{p_{1}, \dots, p_{k} \geq 1 \\
        p_{1}+\cdots+p_{k} = \lambda+1}}
    \frac{(-4/\kappa)^{\lambda+1-k} \; (\lambda!)^{2}}{%
        \prod_{u=1}^{k-1}(\sum_{i=1}^{u} p_{i}) 
            (\sum_{i=u+1}^{k} p_{i})}
     \dopL{-p_{1}} \cdots \dopL{-p_{k}} \; f(\apvar) .
\end{align*}
The vanishing
$\braket{\hwvout' , \intertw( \singvecop_\lambda \hwvmid , \apvar) \hwvin} = 0$
therefore amounts to a differential equation for the
initial term~$f(\apvar) = \initterm{\intertw}(\apvar)$.
With the explicit formula $f(\apvar) = A \, \apvar^{\Delta}$,
this differential equation simplifies to
\begin{align*}
0 = \; & A \, \apvar^{\Delta - 1 - \lambda} \; P_\lambda (\hin,\hout)
\end{align*}
The intertwining operator~$\intertw$ is non-zero only if~$A \neq 0$,
and in this case the desired factorization is equivalent to
$P_\lambda (\hin,\hout) = 0$.
\end{proof}

For fixed~$\lambda \in \Znn$ and $\hin \in \bC$,
the equation~$P_\lambda(\hin, \hout)$ is a 
degree~$\lambda+1$ polynomial equation for~$\hout$.
It is possible to find the roots by a direct calculation \cite{FrenkelZhu2012}.
With the quantum group method, the following easier argument
works as well.
\begin{prop}
\label{prop:factorization_selection_rule_polynomial}
Let~$\lambda \in \Znn$ and $\mu \in \bQ$.
Then we have the factorization
\begin{align*} 
P_\lambda \big( \hwFR{\mu}, \hout \big)
= 
    \prod_{\ell = 0}^{\lambda} \big( \hout - \hwFR{\lambda + \mu - 2 \ell} \big) .
\end{align*}
\end{prop}
\begin{proof}
Recall from~\eqref{eq: first row highest weights}
that~$\hwFR{\mu}$ is a quadratic polynomial in~$\mu$, specifically
$\hwFR{\mu} = \frac{1}{\kappa} \mu^2 
    + \big( \frac{2}{\kappa}-\frac{1}{2} \big) \mu$.
Note also that 
\begin{align*}
\hwFR{\rho+m} - \hwFR{\rho}
= \; & \frac{m}{\kappa} \Big( m + 2 \rho
    + \big( 2-\frac{\kappa}{2} \big) \Big) .
\end{align*}
Recalling that $\kappa \notin \bQ$, we see in particular 
that if $\rho \in \bQ$, then
all $\hwFR{\rho+m}$, $m \in \bZ$, are distinct.

First fix~$\lambda \in \Znn$ and~$\ell \in \set{0,1,\ldots,\lambda}$.
Now if $\mu \in \Znn$ is such that $\mu \geq \lambda$,
then we have
$\sigma = \lambda+\mu-2\ell \in \selRule{\lambda}{\mu}$,
so we may consider the non-zero vector
$u = \QGembed{\sigma}{\lambda}{\mu} \big( \QGvec{\sigma}{0} \big)
\in \QGrep{\lambda} \tens \QGrep{\mu}$.
The associated function
\begin{align*} 
F = \sF[u] \colon \chamber_2 \to \bC 
\end{align*}
satisfies the BSA PDEs $\BSAoper{j} F(x_0,x_1) = 0$, for $j=0,1$,
where~$\BSAoper{j}$ is given by~\eqref{eq: BSA differential operator}.
From the translation and scaling covariance and asymptotics given in
Theorem~\ref{thm:properties_of_correlations} it follows that
in this case the function must be simply
\begin{align*} 
F(x_0,x_1) = B \, (x_1-x_0)^{\Delta} , 
\end{align*}
where~$\Delta = \hwFR{\lambda+\mu-2\ell} - \hwFR{\lambda} - \hwFR{\mu}$
and $B = \betacoef{\lambda + \mu - 2\ell}{\lambda}{\;\;\mu} \neq 0$.

But the BSA PDE for a function of the above form reads
\begin{align*}
0 = \; & B \, (x_1 - x_0)^{\Delta-\lambda-1} \; \sum_{k=1}^{\lambda+1} 
    \sum_{\substack{p_1 , \ldots , p_k \geq 1 \\ 
                    p_1 + \cdots + p_k = \lambda + 1}}
        \frac{(-4/\kappa)^ {1+\lambda-k} \, \lambda !^2}%
        {\prod_{u=1}^{k-1} (\sum_{i=1}^{u} p_i) (\sum_{i=u+1}^{k} p_i)} \\
& \qquad\qquad\qquad\qquad
    \prod_{j=1}^k (-1)^{p_j} \Big( 
          -\Delta + \sum_{i>j} p_i + (p_j-1) \hwFR{\mu} \Big) \\
= \; & B \, (x_1 - x_0)^{\Delta-\lambda-1} \, 
    P_\lambda \big( \hwFR{\mu} , \hwFR{\lambda+\mu-2\ell} \big) .
\end{align*}
We conclude that $P_\lambda \big( \hwFR{\mu} , \hwFR{\lambda+\mu-2\ell} \big) = 0$.
Now observe that
both $\mu \mapsto \hwFR{\mu}$ and $\mu \mapsto \hwFR{\lambda+\mu-2\ell}$ 
are quadratic polynomials in~$\mu$, so also 
$\mu \mapsto P_\lambda \big( \hwFR{\mu} , \hwFR{\lambda+\mu-2\ell} \big)$
is a polynomial in~$\mu$. We have just shown that this polynomial
vanishes for all integers~$\mu \geq \lambda$, so it must be identically zero:
\begin{align*}
P_\lambda \big( \hwFR{\mu} , \hwFR{\lambda+\mu-2\ell} \big) = 0
\qquad \text{ for all 
    $\lambda \in \Znn$, $\ell \in \set{0,1,\ldots,\lambda}$, $\mu \in \bC$.}
\end{align*}

Now let $\lambda \in \Znn$ and $\mu \in \bQ$ be fixed.
Consider the polynomial 
$\hout \mapsto P_\lambda \big( \hwFR{\mu} , \hout \big)$
of degree $\lambda+1$. From the above we see that
$\hout = \hwFR{\lambda+\mu-2\ell}$ are roots of this polynomial,
when $\ell = 0, 1, \ldots, \lambda$. Since~$\kappa \not\in \bQ$,
these~$\lambda+1$ roots are distinct, and so we must have
\begin{align*}
P_\lambda \big(\hwFR{\mu} , \hout \big)
= \const \times 
    \prod_{\ell=0}^{\lambda} \big( \hout - \hwFR{\lambda+\mu-2\ell} \big) .
\end{align*}
From the defining formula~\eqref{eq: selection polynomial two var}
one sees that the leading term of this polynomial of~$\hout$
is $(\hout)^{1+\lambda}$, so the constant of proportionality above
is in fact~$1$.
\end{proof}

The following theorem states that there is always a
nontrivial intertwining operator among an arbitrary
triple of Verma modules. The theorem can be seen as
a particular case of a more general one in
\cite{Li99}.
For self-containedness, we give the proof for the
case of our interest in 
Appendix~\ref{app: construction for Verma modules}.
\begin{thm}
\label{thm:intertw_Vermas}
For $\hmid,\hin ,\hout \in\C$,
\begin{equation*}
	\dim \spintertw {\voaFusion{\VermaCH{c}{\hmid}}{\VermaCH{c}{\hin}}{\VermaCH{c}{\hout}^{\prime}}}=1.
\end{equation*}
\end{thm}

The fusion rules among first row modules of the generic Virasoro VOA
now follow from Lemma~\ref{lem: necessary selection rule}, Proposition~\ref{prop:factorization_selection_rule_polynomial} and Theorem~\ref{thm:intertw_Vermas}.

\begin{cor}
\label{cor:fusion_rules_first_row_modules}
Let $\lambda, \mu, \nu \in\Znn$. Then,
\begin{align*}
	 \dmn \spintertw \voaFusion{\voaFR{\lambda}}{\voaFR{\mu}}{\voaFR{\nu}}=
	 \begin{cases}
	 1 & \nu\in \selRule{\lambda}{\mu},\\
	 0 & \mbox{otherwise}.
	 \end{cases}
\end{align*}
\end{cor}
\begin{proof}
Recall that the first row modules are self-dual (Corollary~\ref{cor: self duality of FRmod}).
We consider the following sequence of embeddings
\begin{align*}
	&\spintertw \voaFusion{\voaFR{\lambda}}{\voaFR{\mu}}{\voaFR{\nu}}
	\to \spintertw \voaFusion{\VermaFR{\lambda}}{\voaFR{\mu}}{\voaFR{\nu}} 
	\simeq \spintertw \voaFusion{\voaFR{\mu}}{\VermaFR{\lambda}}{\voaFR{\nu}}
	\to \spintertw \voaFusion{\VermaFR{\mu}}{\VermaFR{\lambda}}{\voaFR{\nu}}
	\simeq \spintertw \voaFusion{\voaFR{\nu}}{\VermaFR{\mu}}{\VermaFR{\lambda}^{\prime}} \\
	&\to \spintertw \voaFusion{\VermaFR{\nu}}{\VermaFR{\mu}}{\VermaFR{\lambda}^{\prime}}
	\simeq \spintertw \voaFusion{\VermaFR{\lambda}}{\VermaFR{\mu}}{\VermaFR{\nu}^{\prime}},
\end{align*}
where we used the symmetry in Proposition~\ref{prop:symmetry_fusion_rules}. Then, due to Lemma~\ref{lem: necessary selection rule}, the composed embedding is an isomorphism if and only if the highest weights satisfy the conditions
\begin{align*}
	P_{\lambda}(\hwFR{\mu},\hwFR{\nu})=P_{\nu}(\hwFR{\lambda},\hwFR{\mu})=P_{\mu}(\hwFR{\nu},\hwFR{\lambda})=0.
\end{align*}
It is an easy manipulation of the factorization in Proposition~\ref{prop:factorization_selection_rule_polynomial} to see that these conditions%
\footnote{
It can also be seen that two of these conditions imply the other one.
} are equivalent to $\nu\in\selRule{\lambda}{\mu}$.
\end{proof}

\section{Compositions of intertwining operators}%
\label{sec: composition}
Intertwining operators, defined in the previous section,
give $3$-point correlation functions of conformal field theories
(or alternatively, in a geometric interpretation, amplitudes in a
pair-of-pants Riemann surface with three parametrized boundary
components~\cite{Huang-CFT_and_VOA}).
Importantly, they serve as the building blocks of more general
correlation functions. Namely, for
a multipoint correlation function (or an amplitude on
a Riemann surfaces with more than three parametrized boundary components),
one forms an appropriate composition of intertwining operators.

An intertwining operator is a formal series, with coefficients that
are linear operators between modules of the VOA.
Compositions of intertwining operators are then a priori formal series
in several formal variables, also with coefficients that 
are (composed) linear operators between modules.
To get actual correlation functions, however, one considers
suitable matrix elements of the formal series of linear operators, and 
crucially, one must then address the convergence of the corresponding
series.

In this section, we first study properties of the compositions of
the intertwining operators between modules of the first row
subcategory for the generic Virasoro VOA as formal power series.
Then we establish analoguous properties for the functions
obtained by the quantum group method~---
specifically the ones corresponding to the conformal block
vectors of Section~\ref{sec: explicit conformal block vectors}.
The main result of this section 
(Theorem~\ref{thm: quantum group functions and compositions of intertwining ops})
is that the formal series of the composition of intertwining
operators coincide with suitable power series expansions of the
actual conformal block functions, and that these series
are convergent in the appropriate domains.

More precisely, the section is organized as follows.
In Section~\ref{sub: normalized intertwining operators}
we fix our normalization of the intertwining operators,
and in Section~\ref{sub: space of formal series}
we introduce the specific spaces of formal series that
will be used.
Section~\ref{sub: differential equations for matrix elements}
contains the proofs of two key properties;
that (the highest weight matrix elements of)
compositions of the intertwining operators
satisfy a system of partial differential equations,
and that the formal series solutions to this PDE system are unique
(up to scalar multiples).
In Section~\ref{sub: compositions and the quantum group method}
we detail a recursive series expansion procedure for the
quantum group functions, and use this and the partial
differential equations to prove the main result
(Theorem~\ref{thm: quantum group functions and compositions of intertwining ops})
that the formal series of the composition of intertwining operators
converge to the functions corresponding to the conformal
block vectors.
In Section~\ref{sub: applications of the composition result}
we comment on some first applications of the 
result.

\subsection{Fusion rules and normalized intertwining operators}
\label{sub: normalized intertwining operators}

A comparison of
Corollary~\ref{cor:fusion_rules_first_row_modules}
and Lemma~\ref{lem: embedding of QG reps in tensor product of two}
directly gives
\begin{align*}
\dim 
\spintertw \voaFus{\voaFR{\labmid}}{\voaFR{\labin}}{\voaFR{\labout}}
= \; & \dim \mathrm{Hom}_{\Uqsltwo}\big(\QGrep{\labout}, \;
    \QGrep{\labmid} \tens \QGrep{\labin} \big) \\
= \; & \begin{cases}
    1 & \text{ if } \labout \in \selRule{\labmid}{\labin} \\
    0 & \text{ otherwise.}
  \end{cases}
\end{align*}

\begin{defn}\label{def: normalized intertwining operators}
When $\labout \in \selRule{\labmid}{\labin}$ so that nonzero
intertwining operators exist, we denote by
\begin{align*}
\intertwnorm{\labin}{\labmid}{\labout}
\in \spintertw \voaFus{\voaFR{\labin}}{\voaFR{\labmid}}{\voaFR{\labout}} 
\end{align*}
the unique intertwining operator normalized so that
\begin{align}
\label{eq: normalization of intertwining operators}
\initterm{\intertwnorm{\labin}{\labmid}{\labout}} (\apvar)
= \; B \, \apvar^{\Delta} ,
\end{align}
where
\begin{align*}
B = \betacoef{\labmid}{\labin}{\;\labout} \neq 0
\qquad \text{ and } \qquad
\Delta = \deltaexp{\labmid}{\labin}{\;\labout}
    = \hwFR{\labout} - \hwFR{\labmid} - \hwFR{\labin}
\end{align*}
are as in 
Theorem~\ref{thm:properties_of_correlations}.
\end{defn}

The above normalization of the intertwining 
operator~$\intertwnorm{\labin}{\labmid}{\labout}$
is chosen so that that the ``matrix element''
$\Braket{\hwvFR{\labout},\intertwnorm{\labin}{\labmid}{\labout} ( 
\hwvFR{\labmid} , \apvar ) \, \hwvFR{\labin}} = 
\initterm{\intertwnorm{\labin}{\labmid}{\labout}} (\apvar)
= B \, \apvar^{\Delta}$
between highest weight vectors formally
coincides with the function~$x \mapsto \sF[u_{\cbfullseq}](0,x)$
constructed with the quantum group method from the
conformal block 
vector~$u_{\cbfullseq} \in 
\QGrep{\labmid} \tens \QGrep{\labin}$
corresponding to ${\cbfullseq = (\labin , \labmid , \labout)}$
as in Section~\ref{sec: explicit conformal block vectors}.

\subsection{Space of formal series}
\label{sub: space of formal series}
For a sequence
$\lambdafullseq =
(\lambda_{0},\lambda_{1}\dots,\lambda_{N},\lambda_{\infty})
\in \Znn^{N+2}$, and a $\lambdafullseq$-admissible 
sequence $\cbfullseq = (\cbseq_0 , \cbseq_1 , \ldots, \cbseq_{N-1}, \cbseq_{N})$,
we consider the composition of intertwining operators
\begin{align*}
\intertwnorm{\cbseq_{N-1}}{\lambda_{N}}{\cbseq_{N}}(\Mvec_N , {\apvar_{N}})
\intertwnorm{\cbseq_{N-2}}{\lambda_{N-1}}{\cbseq_{N-1}}(\Mvec_{N-1} , {\apvar_{N-1}})
\cdots
\intertwnorm{\cbseq_{1}}{\lambda_{2}}{\cbseq_{2}}(\Mvec_2 , {\apvar_{2}}) 
\intertwnorm{\cbseq_{0}}{\lambda_{1}}{\cbseq_{1}}(\Mvec_1 , {\apvar_{1}}) ,
\end{align*}
which is a priori a formal series in $\apvar_1 , \ldots , \apvar_N$
with coefficients that are linear 
maps ${\voaFR{\cbseq_0} \to \voaFR{\cbseq_N}}$.
We consider in particular the ``matrix element'' between highest weight
states
\begin{align}\label{eq: highest weight matrix element of a composition}
C^{\lambdafullseq}_{\cbfullseq}(\apvar_{1},\dots, \apvar_{N}):=
\Braket{\hwvFR{\cbseq_{N}}^{\prime} , \,
    \intertwnorm{\cbseq_{N-1}}{\lambda_{N}}{\cbseq_{N}}(\hwvFR{\lambda_{N}} , \apvar_{N})
\cdots
\intertwnorm{\cbseq_{0}}{\lambda_{1}}{\cbseq_{1}}(\hwvFR{\lambda_{1}} , \apvar_{1}) \,
    \hwvFR{\lambda_0}} .
\end{align}
As in the proof of 
Proposition~\ref{prop: uniqueness of intertwiners up to constant},
the general matrix elements
\begin{align*}
\Braket{\Mvec_{\infty}^{\prime} , \,
    \intertwnorm{\cbseq_{N-1}}{\lambda_{N}}{\cbseq_{N}}(\Mvec_{N} , \apvar_{N})
\cdots
\intertwnorm{\cbseq_{0}}{\lambda_{1}}{\cbseq_{1}}(\Mvec_{1} , \apvar_{1})
    \, \Mvec_{0}}
\end{align*}
with $\Mvec_{i}\in \voaFR{\lambda_{i}}$, $i=0,1,\dots, N$,
and $\Mvec_{\infty}^{\prime}\in \voaFR{\lambda_{\infty}}^{\prime}$,
can be determined by a recursion
using the Jacobi identity~\eqref{eq:intertw_assoc} (in the 
specific form of Corollary~\ref{cor:formula_intertw})
from this matrix element~\eqref{eq: highest weight matrix element of a composition};
this particular matrix element serves a role analogous
to the initial term of an intertwining operator. A priori, 
\eqref{eq: highest weight matrix element of a composition} is
a formal series in~$\apvar_1 , \ldots , \apvar_n$ with complex coefficients,
i.e., an element of~$\bC\{\apvar_1 , \ldots , \apvar_n\}$.
But in fact the series 
has a more particular structure:
it is essentially a formal
power series (with non-negative powers!) in the ratios of the variables.

\begin{lem}
\label{lem:matrix_elements_formal_series}
The series~\eqref{eq: highest weight matrix element of a composition} 
belongs to the space
\begin{align}\label{eq: the appropriate space of formal series}
\mathbb{C}[[\apvar_{N-1}/\apvar_{N}]][[\apvar_{N-2}/\apvar_{N-1}]]\cdots [[\apvar_{1}/\apvar_{2}]]\apvar_{N}^{\Delta_{N}}\cdots \apvar_{1}^{\Delta_{1}},
\end{align}
with $\Delta_{i}=\hwFR{\cbseq_{i}}-\hwFR{\lambda_{i}}-\hwFR{\cbseq_{i-1}}$ for $i=1,\dots, N$.
\end{lem}

We make a preliminary to prove Lemma~\ref{lem:matrix_elements_formal_series}.
For each $\lambda\in\Znn$ and $n\in\Znn$, we set $\voaFR{\lambda}(n):=(\voaFR{\lambda})_{(\hwFR{\lambda}+n)}$ as the eigenspace of $L_{0}$ corresponding to the eigenvalue $\hwFR{\lambda}+n$.
Suppose that a triple $(\labout,\labmid,\labin)$ 
satisfies the selection rule $\labout \in \selRule{\labmid}{\labin}$, and take the intertwining operator
$\intertwnorm{\labin}{\labmid}{\labout}$
of type $\voaFus{\voaFR{\labmid}}{\voaFR{\labin}}{\voaFR{\labout}} $. 
For $\Mvec\in \voaFR{\labmid}(k)$, $k\in\Znn$,
using Corollary~\ref{cor:formula_intertw}(2)
for~$L_0$ like in
Lemma~\ref{lem: homogeneity of the modes in modules},
we find that $\intertwnorm{\labin}{\labmid}{\labout} (\Mvec,\apvar)$ 
takes the form
\begin{equation*}
\intertwnorm{\labin}{\labmid}{\labout}
   (\Mvec,\apvar)=\sum_{m\in\Z}\Mvec_{(m)}\apvar^{\Delta-m-1}
\qquad \text{ with }
\Delta=\hwFR{\labout}-\hwFR{\labmid}-\hwFR{\labin} ,
\end{equation*}
where $\Mvec_{(m)}\voaFR{\labin}(n)\subset \voaFR{\labout}(k-m+n-1)$
for $m\in\Z$, $n \in \Znn$.

We prove a slightly more general result than Lemma~\ref{lem:matrix_elements_formal_series}:
\begin{lem}
Let $\Mvec\in \voaFR{\cbseq_{0}}(n)$, $n\in\Znn$.
Then, the formal series defined by
\begin{align*}
C^{\lambdafullseq}_{\cbfullseq}(\Mvec;\apvar_{1},\dots, \apvar_{N})
:= \Braket{\hwvFR{\cbseq_{N}}^{\prime},
    \intertwnorm{\cbseq_{N-1}}{\lambda_{N}}{\cbseq_{N}}(\hwvFR{\lambda_{N}} , \apvar_{N})
\cdots
\intertwnorm{\cbseq_{0}}{\lambda_{1}}{\cbseq_{1}}(\hwvFR{\lambda_{1}} , \apvar_{1})
    \Mvec} 
\end{align*}
lies in the space
\begin{align*}
	\mathbb{C}[[\apvar_{N-1}/\apvar_{N}]][[\apvar_{N-2}/\apvar_{N-1}]]\cdots [[\apvar_{1}/\apvar_{2}]] \,
	\apvar_{N}^{\Delta_{N}} \cdots 
	\apvar_{2}^{\Delta_{2}} \apvar_{1}^{\Delta_{1}-n}.
\end{align*}
\end{lem}
\begin{proof}
We prove this by induction on~$N$.
When $N=1$, we have
\begin{align*}
\braket{\hwvFR{\cbseq_{1}}^{\prime},
	\intertwnorm{\cbseq_{0}}{\lambda_{1}}{\cbseq_{1}}
	(\hwvFR{\lambda_{1}},\apvar_{1})\Mvec}
& \; = \sum_{m\in\Z} \braket{\hwvFR{\cbseq_{1}}^{\prime},(\hwvFR{\lambda_{1}})_{(m)}\Mvec}\apvar_{1}^{\Delta_{1}-m-1}.
\end{align*}
Here, since $(\hwvFR{\lambda_{1}})_{(m)} \Mvec 
\subset \voaFR{\cbseq_{1}}(-m+n-1)$, the matrix element 
$\braket{\hwvFR{\cbseq_{1}}^{\prime},
 (\hwvFR{\lambda_{1}})_{(m)}\Mvec}$ vanishes unless $m=n-1$.
Therefore, we can see that
\begin{align*}
\braket{\hwvFR{\cbseq_{1}}^{\prime},
	\intertwnorm{\cbseq_{0}}{\lambda_{1}}{\cbseq_{1}}(\hwvFR{\lambda_{1}},\apvar_{1})\Mvec}
&=\braket{\hwvFR{\cbseq_{1}}^{\prime},(\hwvFR{\lambda_{1}})_{(n-1)}\Mvec}\apvar_{1}^{\Delta_{1}-n} ,
\end{align*}
which is the $N=1$ case.

Then assume that the assertion holds for $N\leq k-1$ with some $k>1$.
Then, we have
\begin{align*}
C^{\lambdafullseq}_{\cbfullseq}(\Mvec;\apvar_{1},\dots, \apvar_{k})
& \; = \sum_{m\in\Z}\Braket{\hwvFR{\cbseq_{k}}^{\prime},
    \intertwnorm{\cbseq_{k-1}}{\lambda_{k}}{\cbseq_{k}}(\hwvFR{\lambda_{k}} , \apvar_{k})
\cdots
\intertwnorm{\cbseq_{1}}{\lambda_{2}}{\cbseq_{2}}(\hwvFR{\lambda_{2}} , \apvar_{2})(\hwvFR{\lambda_{1}})_{(m)}
    \Mvec}\apvar_{1}^{\Delta_{1}-m-1}.
\end{align*}
Again, since $(\hwvFR{\lambda_{1}})_{(m)}\Mvec \in\voaFR{\cbseq_{1}}(-m+n-1)$, the matrix element vanishes unless $m\leq n-1$.
In other words, we have the expansion
\begin{align*}
C^{\lambdafullseq}_{\cbfullseq}(\Mvec;\apvar_{1},\dots, \apvar_{k})
	&=\sum_{\ell=0}^{\infty}\Braket{\hwvFR{\cbseq_{k}}^{\prime},
    \intertwnorm{\cbseq_{k-1}}{\lambda_{k}}{\cbseq_{k}}
        (\hwvFR{\lambda_{k}} , \apvar_{k})
\cdots
\intertwnorm{\cbseq_{1}}{\lambda_{2}}{\cbseq_{2}}(\hwvFR{\lambda_{2}} , \apvar_{2})(\hwvFR{\lambda_{1}})_{(n-\ell-1)}
    \Mvec}\apvar_{1}^{\Delta_{1}-n+\ell} \\
&=\sum_{\ell=0}^{\infty}C^{\lambdafullseq^{\prime}}_{\cbfullseq^{\prime}}\left((\hwvFR{\lambda_{1}})_{(n-\ell-1)}\Mvec;\apvar_{2},\dots,\apvar_{k}\right)\apvar_{1}^{\Delta_{1}-n+\ell},
\end{align*}
where we set
$\lambdafullseq^{\prime} 
= (\cbseq_{1},\lambda_{2},\ldots,\lambda_{k},\lambda_{\infty})$
and $\cbfullseq^{\prime}=(\cbseq_{1},\dots,\cbseq_{k})$.
By the induction hypothesis, each coefficient $C^{\lambdafullseq^{\prime}}_{\cbfullseq^{\prime}}\left((\hwvFR{\lambda_{1}})_{(n-\ell-1)}\Mvec;\apvar_{2},\dots,\apvar_{k}\right)$ lies in
\begin{align*}
	\C[[\apvar_{k-1}/\apvar_{k}]]\cdots [[\apvar_{2}/\apvar_{3}]]
	\, \apvar_{2}^{\Delta_{2}-\ell} .
\end{align*}
Consequently, we have the desired result at $N=k$.
\end{proof}

\subsection{System of differential equations}
\label{sub: differential equations for matrix elements}

Due to the quotienting out of the
singular vector
in the first row 
module~$\voaFR{\lambda} = \VermaFR{\lambda} \, \big/ \, \UEA(\vir) \singvecop_\lambda \hwv{c,h(\lambda)}$,
the matrix 
element~$C^{\lambdafullseq}_{\cbfullseq}(\apvar_{1},\dots, \apvar_{N})$
satisfies a certain system of differential equations.
Given a sequence
$\lambdafullseq=(\lambda_{0},\dots,\lambda_{N},\lambda_{\infty})\in\Znn^{N+2}$,
we introduce the differential operators
\begin{align}\label{eq: Witt differential operator transl inv}
\varWittL{n}{j}
	:=&(-x_{j})^{n}\left(-x_{j} \sum_{1 \leq i \leq N}\frac{\dee}{\dee x_{i}}
	        -(1+n)\hwFR{\lambda_{0}}\right) \\
\nonumber
	& - \sum_{\substack{1 \leq i \leq N \\ i \neq j}} 
	    (x_{i}-x_{j})^{n} \left((x_{i}-x_{j}) \frac{\dee}{\dee x_{i}}+(1+n)\hwFR{\lambda_{i}}\right)
\end{align}
for $j=1,\dots, N$ and $n\in\Z$, and
\begin{align}\label{eq: BSA differential operator transl inv}
\varBSAoper{j}
	=\sum_{k=1}^{\lambda_{j}+1}\sum_{\substack{p_{1},\dots, p_{k}\geq 1 \\ p_{1}+\cdots+p_{k}=\lambda_{j}+1}}\frac{(-4/\kappa)^{1+\lambda_{j}-k}\lambda_{j}!}{\prod_{u=1}^{k-1}(\sum_{i=1}^{u}p_{i})(\sum_{i=u+1}^{k}p_{i})}\varWittL{-p_{1}}{j} \cdots \varWittL{-p_{k}}{j}
\end{align}
for $j=1,\dots, N$.
These differential operators~$\varWittL{n}{j}$ and~$\varBSAoper{j}$
essentially correspond to how
the differential operators~$\WittL{n}{j}$ and~$\BSAoper{j}$
of \eqref{eq: Witt differential operator} and~\eqref{eq: BSA differential operator}
act on translation invariant functions; see 
Lemma~\ref{lem: action of diff op on transl inv fun} for a
precise statement.
Notice that the actions of~$\varWittL{n}{j}$ and~$\varBSAoper{j}$
on the space
\begin{align*}
\C[\apvar_N] \, [[\apvar_{N-1}/\apvar_{N}]][[\apvar_{N-2}/\apvar_{N-1}]]\cdots [[\apvar_{1}/\apvar_{2}]] \, [\apvar_1^{-1}] \; \apvar_{N}^{\Delta_{N}}\cdots \apvar_{1}^{\Delta_{1}}
\end{align*}
of formal series are canonically determined
in such a way that
\begin{align*}
x_{j}\mapsto \apvar_{j} 
\qquad \text{ and } \qquad 
\frac{\dee}{\dee x_{i}}\mapsto \frac{\dee}{\dee \apvar_{i}}
    \quad \text{ for } i=1,\dots, N ,
\end{align*}
and the factors $(x_i-x_j)^n$
are expanded as the following formal power series
\begin{align*}
(x_i-x_j)^n \; \mapsto \; \begin{cases}
    (-\apvar_j)^n \, \sum_{k=0}^\infty \binom{n}{k} (-1)^k \, (\apvar_i/\apvar_j)^k & \text{ for } i<j \\
    \; \apvar_i^n \; \sum_{k=0}^\infty \binom{n}{k} (-1)^k \, (\apvar_j/\apvar_i)^k & \text{ for } i>j ,
\end{cases}
\end{align*}
i.e., $(x_i-x_j)^n \mapsto (\apvar_i-\apvar_j)^n$ for $i>j$ but 
$(x_i-x_j)^n \mapsto (-\apvar_j+\apvar_i)^n$ for $i<j$.

The following standard lemma will be instrumental for us.
We include the details here to concretely demonstrate the
source of the specific power series expansions above.
\begin{lem}
\label{lem:PDE_matrix_elements}
Let $\lambdafullseq=(\lambda_{0},\dots,\lambda_{N},\lambda_{\infty})\in\Znn^{N+2}$ be a sequence and $\cbfullseq=(\cbseq_{0},\dots, \cbseq_{N})$ be $\lambdafullseq$-admissible.
Then, the matrix element $C^{\lambdafullseq}_{\cbfullseq}(\apvar_{1},\dots, \apvar_{N})$ solves the following system of differential equations
\begin{equation*}
\varBSAoper{j} C^{\lambdafullseq}_{\cbfullseq}(\apvar_{1},\dots, \apvar_{N})=0 
\qquad \text{ for } j=1,\dots, N.
\end{equation*}
\end{lem}
\begin{proof}
Fix $j \in \set{1,\ldots,N}$.
For $\Mvec \in \voaFR{\lambda_j}$, consider
\begin{align*}
X(\Mvec;\apvar_{1},\dots, \apvar_{N})
:= \Braket{\hwvFR{\cbseq_{N}}^{\prime},
    \intertwnorm{\cbseq_{N-1}}{\lambda_{N}}{\cbseq_{N}}(\hwvFR{\lambda_{N}} , \apvar_{N})
\cdots
\intertwnorm{\cbseq_{j-1}}{\lambda_{j}}{\cbseq_{j}}(\Mvec , \apvar_{j})
\cdots
\intertwnorm{\cbseq_{0}}{\lambda_{1}}{\cbseq_{1}}(\hwvFR{\lambda_{1}} , \apvar_{1})
    \hwvFR{\lambda_{0}}} . 
\end{align*}
Then we can use Corollary~\ref{cor:formula_intertw}(1)
to calculate $X(L_{-p} \Mvec;\apvar_{1},\dots, \apvar_{N})$, for $p>0$,
\begin{align*}
 & X(L_{-p}\Mvec;\apvar_{1},\dots, \apvar_{N}) \\
= \; & \phantom{+} \sum_{k=0}^\infty \binom{1-p}{k} (-\apvar_j)^k \\
& \qquad\qquad \Braket{\hwvFR{\cbseq_{N}}^{\prime},
    \intertwnorm{\cbseq_{N-1}}{\lambda_{N}}{\cbseq_{N}}(\hwvFR{\lambda_{N}} , \apvar_{N})
\cdots L_{-p-k} \,
\intertwnorm{\cbseq_{j-1}}{\lambda_{j}}{\cbseq_{j}}(\Mvec , \apvar_{j})
\cdots
\intertwnorm{\cbseq_{0}}{\lambda_{1}}{\cbseq_{1}}(\hwvFR{\lambda_{1}} , \apvar_{1})
    \hwvFR{\lambda_{0}}} \\
& - \sum_{k=0}^\infty \binom{1-p}{k} (-\apvar_j)^{1-p-k} \\
& \qquad\qquad \Braket{\hwvFR{\cbseq_{N}}^{\prime},
    \intertwnorm{\cbseq_{N-1}}{\lambda_{N}}{\cbseq_{N}}(\hwvFR{\lambda_{N}} , \apvar_{N})
\cdots 
\intertwnorm{\cbseq_{j-1}}{\lambda_{j}}{\cbseq_{j}}(\Mvec , \apvar_{j})
\, L_{k-1} \cdots
\intertwnorm{\cbseq_{0}}{\lambda_{1}}{\cbseq_{1}}(\hwvFR{\lambda_{1}} , \apvar_{1})
    \hwvFR{\lambda_{0}}} .
\end{align*}
In the first term, we use 
\begin{align}\label{eq: primary field Virasoro commutation}
\big[ L_{n} ,
    \intertwnorm{\cbseq_{i-1}}{\lambda_{i}}{\cbseq_{i}}(\hwvFR{\lambda_i} , \apvar_{i}) \big]
\; = \; \Big( \apvar_i^{1+n} \pder{\apvar_{i}} + \apvar_i^{n} \, (1+n) \, \hwFR{\lambda_i} \Big) \; 
    \intertwnorm{\cbseq_{i-1}}{\lambda_{i}}{\cbseq_{i}}(\hwvFR{\lambda_i} , \apvar_{i}) .
\end{align}
from Corollary~\ref{cor:formula_intertw}(2), to commute $L_{-p-k}$ to the left,
where it annihilates the highest weight 
vector~$\hwvFR{\cbseq_{N}}^{\prime} \in \voaFR{\lambda_N}'$.
The first term thus becomes
\begin{align*}
& - \sum_{i>j}
\sum_{k=0}^\infty \binom{1-p}{k} (-\apvar_j)^k \,
    \Big( \apvar_i^{1-p-k} \pder{\apvar_{i}} + \apvar_i^{-p-k} \, (1-p-k) \, \hwFR{\lambda_i} \Big) \; X(\Mvec;\apvar_{1},\dots, \apvar_{N}) \\
= \; & - \sum_{i>j} \Big( (\apvar_i-\apvar_j)^{1-p} \pder{\apvar_{i}} +
        (\apvar_i-\apvar_j)^{-p} \, (1-p) \, \hwFR{\lambda_i} \Big)
    \; X(\Mvec;\apvar_{1},\dots, \apvar_{N}) ,
\end{align*}
which is indeed expanded in non-negative
powers of~$\apvar_{j}/\apvar_{i}$, $i>j$.
Similarly in the second term, by commuting $L_{k-1}$ to the right and simplifying,
the commutator contributions become 
\begin{align*}
& - \sum_{i<j} \Big( (-\apvar_j+\apvar_i)^{1-p} \pder{\apvar_{i}} +
        (-\apvar_j+\apvar_i)^{-p} \, (1-p) \, \hwFR{\lambda_i} \Big)
    \; X(\Mvec;\apvar_{1},\dots, \apvar_{N}) ,
\end{align*}
which is expanded in non-negative
powers of~$\apvar_{i}/\apvar_{j}$, $i<j$. However, for $k=0$ and $k=1$,
the highest weight vector $\hwvFR{\cbseq_{0}} \in \voaFR{\lambda_0}$
is not annihilated by~$L_{k-1}$, so also the following two further terms
remain
\begin{align*}
& - (-\apvar_j)^{1-p}
  \Braket{\hwvFR{\cbseq_{N}}^{\prime},
    \intertwnorm{\cbseq_{N-1}}{\lambda_{N}}{\cbseq_{N}}(\hwvFR{\lambda_{N}} , \apvar_{N})
\cdots
\intertwnorm{\cbseq_{j-1}}{\lambda_{j}}{\cbseq_{j}}(\Mvec , \apvar_{j})
\cdots
\intertwnorm{\cbseq_{0}}{\lambda_{1}}{\cbseq_{1}}(\hwvFR{\lambda_{1}} , \apvar_{1})
    L_{-1} \hwvFR{\lambda_{0}}} \\
& - (-\apvar_j)^{-p} \, (1-p) \;
  \Braket{\hwvFR{\cbseq_{N}}^{\prime},
    \intertwnorm{\cbseq_{N-1}}{\lambda_{N}}{\cbseq_{N}}(\hwvFR{\lambda_{N}} , \apvar_{N})
\cdots
\intertwnorm{\cbseq_{j-1}}{\lambda_{j}}{\cbseq_{j}}(\Mvec , \apvar_{j})
\cdots
\intertwnorm{\cbseq_{0}}{\lambda_{1}}{\cbseq_{1}}(\hwvFR{\lambda_{1}} , \apvar_{1})
    L_{0} \, \hwvFR{\lambda_{0}}} .
\end{align*}
After simplifications (the first one still using
Corollary~\ref{cor:formula_intertw}), these two can be written as
\begin{align*}
(-\apvar_j)^{1-p} \sum_{1 \leq i \leq N} \pder{\apvar_{i}}
\; X(\Mvec;\apvar_{1},\dots, \apvar_{N})
- (-\apvar_j)^{-p} \, (1-p) \, \hwFR{\lambda_0}
\; X(\Mvec;\apvar_{1},\dots, \apvar_{N})
\end{align*}
By combining everything, we get
\begin{align*}
X( L_{-p} \Mvec;\apvar_{1},\dots, \apvar_{N})
= \varWittL{-p}{j} \, X(\Mvec;\apvar_{1},\dots, \apvar_{N}) .
\end{align*}

Using this auxiliary observation, the assertion then follows easily
from the formula~\eqref{eq:singular_vector} for the singular
vector~$\singvecop_{\lambda_j} \hwv{c,h(\lambda_j)} \in \VermaFR{\lambda_j}$,
and the fact that its canonical
projection in~$\voaFR{\lambda_j}$ vanishes,
$\singvecop_{\lambda_j} \hwvFR{\lambda_j} = 0$.
\end{proof}

The coefficient of the monomial
$\apvar_{N}^{\Delta_{N}}\cdots \apvar_{1}^{\Delta_{1}}$ in
\begin{align*}
C^{\lambdafullseq}_{\cbfullseq}(\apvar_{1},\dots, \apvar_{N}) 
\; \in \; \C[[\apvar_{N-1}/\apvar_{N}]]\cdots [[\apvar_{1}/\apvar_{2}]]\,
    \apvar_{N}^{\Delta_{N}}\cdots \apvar_{1}^{\Delta_{1}}
\end{align*}
is easy to trace in the calculations above. Ultimately
because of the chosen normalizations of the intertwining operators, this
coefficient is
\begin{align*}
\prod_{j=1}^N \betacoef{\lambda_j}{\cbseq_{j-1}}{\;\cbseq_{j}} \neq 0 .
\end{align*}
In particular 
$C^{\lambdafullseq}_{\cbfullseq}(\apvar_{1},\dots, \apvar_{N})$
is a non-zero solution to the system of differential equations above.

\begin{thm}
\label{thm:uniqueness_differential_equations_formal_series}
In the space of formal series of the 
form~\eqref{eq: the appropriate space of formal series},
the solution space
\begin{align*}
\bigcap_{j=1}^N \Kern \, \varBSAoper{j}
\; \subset \; 
\C[\apvar_N] \, [[\apvar_{N-1}/\apvar_{N}]]\cdots [[\apvar_{1}/\apvar_{2}]] \, [\apvar_1^{-1}]
    \; \apvar_{N}^{\Delta_{N}}\cdots \apvar_{1}^{\Delta_{1}}
\end{align*}
of the above system of differential
equations is one-dimensional,
\begin{align*}
\dmn \Big(
\bigcap_{j=1}^N \Kern \, \varBSAoper{j} \Big)
\; = \; 1  .
\end{align*}
\end{thm}
\begin{proof}
Since Lemma~\ref{lem:PDE_matrix_elements} ensures the
existence of a nonzero solution, it suffices to show
that solutions are unique up to a multiplicative constant.
We prove this by induction on~$N$. 
When $N=1$, 
observe that for any $n \in \bZ$ we have
\begin{align*}
\varBSAoper{1} \cdot \apvar_{1}^{\Delta_{1} + n}
= P_{\lambda_{1}} \big(\hwFR{\lambda_{0}},\hwFR{\cbseq_{1}}+n \big) \,
    \apvar_{1}^{\Delta_{1} + n - \lambda_1 - 1} ,
\end{align*}
where
$P_{\lambda_1}$
is the polynomial~\eqref{eq: selection polynomial two var}.
Recall that by the genericity~$\kappa \notin \bQ$, we have
\begin{align*}
P_{\lambda_{1}} \big(\hwFR{\lambda_{0}},\hwFR{\cbseq_{1}}+n \big) \; = \; 0
\end{align*}
if and only if~$n=0$. This shows that $\Kern \varBSAoper{1} = \bC \, \apvar_{1}^{\Delta_{1}}$,
and proves one-dimensionality.

For the induction step,
assume the uniqueness of solutions up to multiplicative
constants for~$N-1$ variables.
We regard
\begin{align*}
& \C[\apvar_N] \, [[\apvar_{N-1}/\apvar_{N}]]\cdots [[\apvar_{1}/\apvar_{2}]] \, [\apvar_1^{-1}]
    \; \apvar_{N}^{\Delta_{N}}\cdots \apvar_{1}^{\Delta_{1}} 
\end{align*}
as a subspace of the space
\begin{align*}
\bigg(\C[\apvar_N] \, [[\apvar_{N-1}/\apvar_{N}]]\cdots [[\apvar_{2}/\apvar_{3}]] \, [\apvar_2^{-1}] \; \apvar_{N}^{\Delta_{N}}\cdots \apvar_{2}^{\Delta_{2}}\bigg)
    [[\apvar_1]]\, [\apvar_1^{-1}] \, \apvar_{1}^{\Delta_{1}}
\end{align*}
and on the latter we introduce a $\bZ$-grading so that 
the degree of a monomial is
\begin{align*}
\deg \big( \apvar_{N}^{\Delta_{N}+n_N}\cdots \apvar_{1}^{\Delta_{1}+n_1} \big)
    \; = \; n_1
\qquad \text{ for } n_1, \ldots, n_N \in \bZ .
\end{align*}
The grading also gives rise to the corresponding notion of a degree of
an operator acting on the latter space.

Now let $C(\apvar_{1},\dots, \apvar_{N}) \in \bigcap_{j=1}^N \Kern \, \varBSAoper{j}$
be a non-zero solution. Expand it according to the grading above as
\begin{equation*}
	C(\apvar_{1},\dots, \apvar_{N})=\sum_{n=d}^{\infty}\apvar_{1}^{\Delta_{1}+n}
	    C_{n}(\apvar_{2},\dots, \apvar_{N}),
\end{equation*}
where $d \in \bZ$ is the lowest degree with a non-vanishing coefficient,
$C_{d} (\apvar_{2},\dots, \apvar_{N}) \neq 0$.
We express the action 
of~$\varWittL{-p}{1}$, $p\in\Z$, explicitly as
\begin{align*}
\varWittL{-p}{1} 
= \; & (-1)^{p+1}\apvar_{1}^{-p}\left(\apvar_{1}\frac{\dee}{\dee \apvar_{1}}+(1-p)\hwFR{\lambda_{0}}\right) \\
& -\sum_{i=2}^{N} \left(
  (-1)^{p}\apvar_{1}^{-p+1} \frac{\dee}{\dee \apvar_{i}}
  + \sum_{k=0}^{\infty}\binom{-p+1}{k}\apvar_{i}^{-p-k}(-\apvar_{1})^{k}\left(\apvar_{i}\frac{\dee}{\dee \apvar_{i}}+(-p-k+1)\hwFR{\lambda_{i}}\right) \right)
\end{align*}
in order to manifest the degree of each contribution.
Let us denote the first term in this expansion by
\begin{align*}
	\ldvarWittL{-p}{1}:=(-1)^{p+1}\apvar_{1}^{-p}\left(\apvar_{1}\frac{\dee}{\dee \apvar_{1}}+(1-p)\hwFR{\lambda_{0}}\right),
\end{align*}
which is an operator of degree $-p$ and 
involves no other variables besides~$\apvar_1$.
It is readily seen that, if $p\geq 0$, the
difference $\varWittL{-p}{1}-\ldvarWittL{-p}{1}$ is a
sum of operators of degrees strictly greater than $-p$.
Therefore, when we set
\begin{align*}
	\ldvarBSAoper{1}
	=\sum_{k=1}^{\lambda_{1}+1}\sum_{\substack{p_{1},\dots, p_{k}\geq 1 \\ p_{1}+\cdots+p_{k}=\lambda_{1}+1}}\frac{(-4/\kappa)^{1+\lambda_{1}-k}\lambda_{1}!}{\prod_{u=1}^{k-1}(\sum_{i=1}^{u}p_{i})(\sum_{i=u+1}^{k}p_{i})}\ldvarWittL{-p_{1}}{1}\cdots \ldvarWittL{-p_{k}}{1},
\end{align*}
this is of degree $-\lambda_{1}-1$ and the difference $\varBSAoper{1}-\ldvarBSAoper{1}$ is a sum of operators of degree strictly greater than $-\lambda_{1}-1$.
Specifically, let us expand $\varBSAoper{1}$ as
\begin{align*}
	\varBSAoper{1}=\ldvarBSAoper{1}+\sum_{k=1}^{\infty}\extraterm^{(1)}_{k},\quad \deg \extraterm^{(1)}_{k}=-\lambda_{1}-1+k.
\end{align*}
Recall the polynomial~\eqref{eq: selection polynomial two var}.
It is again straightforward that the action
of $\ldvarBSAoper{1}$ on the monomial $\apvar_{1}^{\Delta_{1}+n}$, $n\in\Z$, gives
\begin{align*}
\ldvarBSAoper{1} \cdot \apvar_{1}^{\Delta_{1}+n}
= P_{\lambda_{1}} \big(\hwFR{\lambda_{0}},\hwFR{\cbseq_{1}}+n \big)
        \; \apvar_{1}^{\Delta_{1}+n-\lambda_{1}-1}.
\end{align*}
Having assumed the generic case $\kappa \notin \bQ$,
from the factorization in Proposition \ref{prop:factorization_selection_rule_polynomial},
we see that
the factor
${P_{\lambda_{1}} \big(\hwFR{\lambda_{0}},\hwFR{\cbseq_{1}}+n \big)}$
vanishes if and only if~$n=0$.
This observation gives first of all an indicial equation:
the lowest degree term in~$\varBSAoper{1} \, C = 0$ is
\begin{align*}
C_{d} 
  \, \ldvarBSAoper{1} \cdot \apvar_{1}^{\Delta_{1}+d}
\; = \; C_{d} 
  \, P_{\lambda_{1}} \big(\hwFR{\lambda_{0}},\hwFR{\cbseq_{1}}+d \big)
        \; \apvar_{1}^{\Delta_{1}+d-\lambda_{1}-1} ,
\end{align*}
and its vanishing is only possible if~$d = 0$.
Moreover, the same observation gives a recursion to determine the
coefficients of higher degree, $C_n$ for $n > d = 0$.
Indeed, plugging in the series expansions for both~$C$
and~$\varBSAoper{1}$, the differential equation~$\varBSAoper{1} C = 0$
for the solution~$C$ becomes
\begin{align*}
0 \; 
&=\left(\ldvarBSAoper{1}+\sum_{k=1}^{\infty}\extraterm^{(1)}_{k}\right)\sum_{n=0}^{\infty}\apvar_{1}^{\Delta_{1}+n} \, C_{n}(\apvar_{2},\dots, \apvar_{N}) \\
&=\sum_{n=0}^{\infty}P_{\lambda_{1}}(\hwFR{\lambda_{0}},\hwFR{\cbseq_{1}}+n)
     \, \apvar_{1}^{\Delta_{1}+n-\lambda_{1}-1} \,  C_{n}(\apvar_{2},\dots, \apvar_{N}) \\
&\qquad\qquad +\sum_{k=1}^{\infty}\sum_{n=0}^{\infty}\extraterm^{(1)}_{k}\left(\apvar_{1}^{\Delta_{1}+n} \, C_{n}(\apvar_{2},\dots, \apvar_{N})\right) \\
&=\sum_{k=1}^{\infty}P_{\lambda_{1}}(\hwFR{\lambda_{0}},\hwFR{\cbseq_{1}}+k)
  \, \apvar_{1}^{\Delta_{1}+k-\lambda_{1}-1}  \, C_{k}(\apvar_{2},\dots, \apvar_{N}) \\
&\qquad\qquad +\sum_{k=1}^{\infty}\sum_{n=0}^{k-1}\extraterm^{(1)}_{k-n}\left(\apvar_{1}^{\Delta_{1}+n} \, C_{n}(\apvar_{2},\dots, \apvar_{N})\right).
\end{align*}
For each $k=1,2,\dots$, 
the equality of the components of degree $-\lambda_{1}-1+k$ gives
\begin{equation*}
	C_{k}(\apvar_{2},\dots,\apvar_{N})=-\frac{\apvar_{1}^{-\Delta_{1}-k+\lambda_{1}+1}}{P_{\lambda_{1}}(\hwFR{\lambda_{0}},\hwFR{\cbseq_{1}}+k)}\sum_{n=0}^{k-1}\extraterm^{(1)}_{k-n}\left(\apvar_{1}^{\Delta_{1}+n} \, C_{n}(\apvar_{2},\dots, \apvar_{N})\right) ,
\end{equation*}
where we used the property that
$P_{\lambda_{1}}(\hwFR{\lambda_{0}},\hwFR{\cbseq_{1}}+k)\neq 0$
for $k=1,2,\dots$.
The apparent dependence of the right hand side on $\apvar_{1}$ is checked to be cancelled by considering the degrees of the operators involved.
This formula implies that, for each $k=1,2\dots$,
the coefficient $C_{k}$ is determined by the finitely
many previous coefficients $C_{0}, C_{1},\dots, C_{k-1}$.
Recursively, we conclude that each $C_{k}$, $k=1,2,\dots, $
is determined by the initial coefficient~$C_{0}$.

Note that $C_{0}$ belongs to the space
\begin{align*}
\C[\apvar_N] \, [[\apvar_{N-1}/\apvar_{N}]]\cdots [[\apvar_{2}/\apvar_{3}]] \, [\apvar_2^{-1}] \; \apvar_{N}^{\Delta_{N}}\cdots \apvar_{2}^{\Delta_{2}}
\end{align*}
The next task is to show that the formal series $C_{0}$ solves the
desired system differential equations. Those differential equations
involve the differential operators
\begin{align*}
\varvarWittL{-p}{j}
:= \; &(-x_{j})^{-p}\left(-x_{j} \sum_{2 \leq i \leq N}\frac{\dee}{\dee x_{i}}
	        -(1-p)\hwFR{\cbseq_{1}}\right) \\
	&-\sum_{\substack{2 \leq i \leq N \\ i \neq j}} 
	    (x_{i}-x_{j})^{-p} \left((x_{i}-x_{j}) \frac{\dee}{\dee x_{i}}+(1-p)\hwFR{\lambda_{i}}\right)
\end{align*}
for $2 \leq j \leq N$. These differential operators
$\varvarWittL{-p}{j}$ do not involve the variable $\apvar_1$,
and they can be viewed either as operators on the space
$\C[\apvar_N] \, [[\apvar_{N-1}/\apvar_{N}]]\cdots [[\apvar_{2}/\apvar_{3}]] \, [\apvar_2^{-1}] \; \apvar_{N}^{\Delta_{N}}\cdots \apvar_{2}^{\Delta_{2}}$
or as operators on the space
$\big(\C[\apvar_N] \, [[\apvar_{N-1}/\apvar_{N}]]\cdots [[\apvar_{2}/\apvar_{3}]] \, [\apvar_2^{-1}] \, \apvar_{N}^{\Delta_{N}}\cdots \apvar_{2}^{\Delta_{2}}\big)
    [[\apvar_1]]\, [\apvar_1^{-1}] \, \apvar_{1}^{\Delta_{1}}$,
in which case they obviously have degree~$0$.
In the latter space, the difference $\varWittL{-p}{j}-\varvarWittL{-p}{j}$
can be straightforwardly simplified to the form
\begin{align*}
\varWittL{-p}{j} - \varvarWittL{-p}{j}
= \; & (1-p)(-\apvar_j)^{-p} \Big( \Delta_1 - \apvar_1 \pder{\apvar_1} \Big) \\ 
& - \sum_{k=2}^\infty \binom{1-p}{k} \apvar_1^k (-\apvar_j)^{1-p-k} \pder{\apvar_1}
    - (1-p) \hwFR{\lambda_{1}} \sum_{k=1}^\infty \binom{-p}{k} \apvar_1^k (-\apvar_j)^{-p-k} ,
\end{align*}
where the first line is a degree zero operator, and all the
terms on the second line have strictly positive degrees.
We thus see that
\begin{equation*}
\varWittL{-p}{j} \bigg( \sum_{n=0}^\infty
    \apvar_{1}^{\Delta_{1}+n} \, C_{n}(\apvar_{2},\dots, \apvar_{N}) \bigg)
= \apvar_{1}^{\Delta_{1}}
   \Big( \varvarWittL{-p}{j} \, C_{0}(\apvar_{2},\dots,\apvar_{N}) \Big) + \hot ,
\end{equation*}
where $\hot$ contains only terms of strictly positive degree.
Now the vanishing of the degree zero term in the
differential equation
$\ldvarBSAoper{j} \Big( \sum_{n=0}^\infty
    \apvar_{1}^{\Delta_{1}+n} \, C_{n}(\apvar_{2},\dots, \apvar_{N}) \Big) = 0$
implies the differential equation
\begin{align*}
\varvarBSAoper{j} \, C_0(\apvar_2 , \ldots, \apvar_N) = 0
\end{align*}
for $C_0$, where
\begin{align*}
\varvarBSAoper{j} :=
  \sum_{k=1}^{\lambda_{j}+1}\sum_{\substack{p_{1},\dots, p_{k}\geq 1 \\ p_{1}+\cdots+p_{k}=\lambda_{j}+1}}\frac{(-4/\kappa)^{1+\lambda_{j}-k}\lambda_{j}!}{\prod_{u=1}^{k-1}(\sum_{i=1}^{u}p_{i})(\sum_{i=u+1}^{k}p_{i})}\varvarWittL{-p_{1}}{j} \cdots \varvarWittL{-p_{k}}{j} .
\end{align*}

The initial coefficient $C_{0}(\apvar_2,\ldots,\apvar_N)$ therefore solves the system of
differential equations ${\varvarBSAoper{j}\, C_{0}=0}$ for~$j=2,\dots, N$,
which is a system of the original form with one fewer variable.
By the induction hypothesis, such a $C_{0}$ is unique up to multiplicative constants.
Since a solution~$C$ is determined by its initial coefficient~$C_0$,
we conclude that the solution $C$ is also unique one up to multiplicative constants.
\end{proof}

\subsection{Functions from the quantum group method}
\label{sub: compositions and the quantum group method}
In this subsection we do the exactly analoguous
steps for the functions from the quantum group method as
were done for the highest weight matrix elements of
compositions of intertwining operators in the previous
subsection. The conclusions have an identical appearance, and
this is in fact what will allow us to show the equality of the two.
While Section~\ref{sub: differential equations for matrix elements}
dealt with formal power series, we will now be working with functions.

We use the quantum group method of
Section~\ref{sec: quantum group method}, with the difference
that we first use $N+1$ variables, labeled $x_0, x_1, \ldots, x_{N}$.
After some initial observations, we will in fact set $x_0 = 0$.

\subsubsection*{Translation invariance}

We first note that the differential operators~$\WittL{n}{j}$
of~\eqref{eq: Witt differential operator}
preserve translation invariance.
\begin{lem}
\label{lem:translation_invariance_under_Witt_action}
If $F 
\in \ContF^\infty(\chamber_{N+1})$
is translation invariant,
\begin{equation*}
F(x_{0}+s,x_{1}+s,\dots, x_{N}+s)=F(x_{0},x_{1},\dots, x_{N})
\quad \text{ for } (x_{0},x_{1},\dots, x_{N})\in \chamber_{N+1} \text{ and }
s\in\mathbb{R} ,
\end{equation*}
then so is $\WittL{n}{j}F$,
for any $j \in \set{0,1,\dots, N}$ and $n\in\Z$.
\end{lem}
\begin{proof}
Let us set $\mathbb{L}_{-1}:=\sum_{i=0}^{N}\frac{\dee}{\dee x_{i}}$. Then, a
smooth function $F$ is translation invariant if and only if $\mathbb{L}_{-1}F=0$.
The desired result follows from the fact that each operator $\WittL{n}{j}$, $j=0,1,\dots, N$, $n\in\Z$ commutes with $\mathbb{L}_{-1}$.
\end{proof}

We introduce the \term{restricted chamber}
\begin{equation*}
	\chamber^{+}_{N}=\Set{(x_{1},\dots,x_{N})\in \chamber_{N} \; \big| \; x_{1}>0}
\end{equation*}
of $N$ variables.
To a given translation invariant function
$F 
\in \ContF^{\infty} (\chamber_{N+1})$,
we associate a corresponding function on $\chamber_{N}^{+}$ by
\begin{equation*}
\widetilde{F}(x_{1},\dots, x_{N}):=F(0,x_{1},\dots, x_{N}),
\quad \text{ for } (x_{1},\dots, x_{N}) \in \chamber_{N}^{+}.
\end{equation*}
For typographical reasons, when the function $F$
is given by a lengthy expression, we
place the tilde symbol as a superscript,
$\widetilde{F} = F^{\thicksim}$.

\begin{lem}
\label{lem: action of diff op on transl inv fun}
Fix $j \in \set{1,\ldots,N}$.
Let $F\in \ContF^{\infty}(\chamber_{N+1})$ be a translation
invariant function. Then we have
\begin{align*}
\left(\WittL{n_{1}}{j}\cdots \WittL{n_{k}}{j} F\right)^{\thicksim}
= \varWittL{n_{1}}{j} \cdots \varWittL{n_{k}}{j} \widetilde{F},
\end{align*}
for any $k\in\Znn$ and $n_{1},\dots, n_{k}\in\Z$.
\end{lem}
\begin{proof}
If $F$ is translation invariant, then its 
partial derivative w.r.t.~$x_{0}$ can be rewritten as
\begin{equation*}
\pder{x_{0}} F = - \sum_{i=1}^{N} \pder{x_{i}} F .
\end{equation*}
Consequently, we observe that
\begin{align*}
\WittL{n}{j}F
= \; & \phantom{-} (x_{0}-x_{j})^{n} \,
    \Big( (x_{0}-x_{j})\sum_{1 \leq i \leq N} \pder{x_{i}}
        - (1+n) \, \hwFR{\lambda_{0}} \Big) F \\
& - \sum_{\substack{1 \leq i \leq N\\ i\neq j}}
    (x_{i}-x_{j})^{n} \, \Big( (x_{i}-x_{j})\pder{x_{i}}
        +(1+n) \, \hwFR{\lambda_{i}} \Big) F,
\end{align*}
which admits the specialization at $x_{0}=0$ and gives
$\big( \WittL{n}{j} F \big)^{\thicksim} = \varWittL{n}{j} \widetilde{F}$.
The assertion is then obtained from  
Lemma~\ref{lem:translation_invariance_under_Witt_action}
by induction on~$k$.
\end{proof}

The following simple corollary of this lemma
yields differential equations of the same form
as those in Section~\ref{sub: differential equations for matrix elements}.
\begin{cor}
Fix a sequence 
$\lambdafullseq =
(\lambda_{0},\lambda_{1},\dots, \lambda_{N},\lambda_{\infty})\in \Znn^{N+2}$
and a $\lambdafullseq$-admissible sequence $\cbfullseq=(\cbseq_{0},\dots,\cbseq_{N})$.
Let $u_{\cbfullseq} \in
\big( \bigotimes_{j=1}^N \QGrep{\lambda_j} \big) \tens \QGrep{\lambda_{0}}$
be the corresponding conformal block 
vector~\eqref{eq: construction of conformal block vector},
and let
$F=\sF[u_{\cbfullseq}]\in \ContF^{\infty}(\chamber_{N+1})$ 
denote the corresponding function obtained by the quantum group method
of Section~\ref{sub: quantum group method main statement}.
Then $F$ is translation invariant, and 
the associated function $\widetilde{F}$ on $\chamber_{N}^{+}$
solves the system of differential equations
$\varBSAoper{j}\widetilde{F}=0$, $j=1,\dots, N$.
\end{cor}

\subsubsection*{Series expansions of the functions}
\label{ssec: quantum group functions as power series}
Fix again a sequence 
$\lambdafullseq =
(\lambda_{0},\lambda_{1},\dots, \lambda_{N},\lambda_{\infty})\in \Znn^{N+2}$
and a $\lambdafullseq$-admissible sequence $\cbfullseq=(\cbseq_{0},\dots,\cbseq_{N})$,
and let $\widetilde{F} \colon \chamber_{N}^{+} \to \bC$ be the function given by
\begin{align*}
\widetilde{F}(x_1, \ldots, x_N) \; = \; \sF[u_{\cbfullseq}](0,x_1,\ldots,x_N)
\qquad \text{ for } (x_1, \ldots, x_N) \in \chamber_{N}^{+}
\end{align*}
where $u_{\cbfullseq} \in
\big( \bigotimes_{j=1}^N \QGrep{\lambda_j} \big) \tens \QGrep{\lambda_{0}}$
is the corresponding conformal block vector
of~\eqref{eq: construction of conformal block vector}.
We will expand this function as a series recursively, one variable at a time,
starting from~$x_1$.
The series expansion in~$x_1$ is the Frobenius series at~$0$
given by Lemma~\ref{lem: Frobenius series of QG functions},
\begin{align*}
\widetilde{F}(x_1, \ldots, x_N) 
\; = \; x_1^{\Delta_1} \, \sum_{k=0}^\infty c_k(x_2, \ldots, x_N) \, x_1^k ,
\end{align*}
and the other variables
$(x_2, \ldots, x_N) \in \chamber_{N-1}^+$
are treated as parameters. When the other variables stay in the open subset
\begin{align}
\chamber_{N-1}^{+;R} :=
\set{(x_2, \ldots, x_N) \in \chamber_{N-1}^+ \; \Big| \; x_2 > R} ,
\end{align}
the Frobenius series in locally uniformly $R$-controlled.

The differential operators~$\varWittL{n}{j}$
and~$\varBSAoper{j}$ 
of~\eqref{eq: Witt differential operator transl inv}
and~\eqref{eq: BSA differential operator transl inv}
are composed of
\begin{itemize}
\item differentiations~$\pder{x_1}$ with respect to 
the power series variable~$x_1$;
\item differentiations~$\pder{x_i}$
with respect to the parameters~$x_i$, for $i=2,\ldots,N$;
\item multiplication operators by~$(x_i-x_j)^n$
for $i \neq j$ and by $x_i^n$.
\end{itemize}
If $i,j>1$, then the multiplication by~$(x_i-x_j)^n$ acts on
the coefficients $c_k \in \ContF^\infty(\chamber_{N-1}^{+;R})$
by a multiplication by a smooth (and in particular locally
bounded) function. By contrast ${(x_i-x_1)^n}$ for~$i = 2, \ldots, N$
is expanded as a power series
$(x_i-x_1)^n = \sum_{j=0}^\infty (-1)^j \, \binom{n}{j} \, x_i^{n-j} \, x_1^j$,
which is itself locally uniformly $R$-controlled power series in~$x_1$ on 
$\chamber_{N-1}^{+;R}$, and the 
multiplication operator
involves a convolution of the coefficients of the power series.

By Lemma~\ref{lem: operations on R controlled series}, each one
of the above constituent operators preserves the space of
locally uniformly $R$-controlled series and acts naturally
coefficientwise when~$(x_2, \ldots, x_N) \in \chamber_{N-1}^{+;R}$.
Therefore the constituent
operators can be composed, and the operators~$\varWittL{n}{j}$
and~$\varBSAoper{j}$ also act on the space of
locally uniformly $R$-controlled series, and they also act naturally
coefficientwise, when~$(x_2, \ldots, x_N) \in \chamber_{N-1}^{+;R}$.
For different values of $R>0$, the actions of~$\varWittL{n}{j}$
and~$\varBSAoper{j}$ on
power series parametrized by~$\chamber_{N-1}^{+;R}$
are consistent (coefficient functions are obtained by restrictions
to the smaller subset), so they give rise to
a natural action on power series
parametrized simply by~$\chamber_{N-1}^{+} = \bigcup_{R>0} \chamber_{N-1}^{+;R}$.

It will turn out in the analysis below that the coefficient
functions~$c_k \in \ContF^\infty(\chamber_{N-1}^{+})$
are themselves (essentially) given by functions from the quantum group
method, and they therefore admit Frobenius series expansions of their
own. Recursively in the number of variables, this allows us to
uniquely associate to~$\widetilde{F}$ a power series in
the space~\eqref{eq: the appropriate space of formal series},
\begin{align*}
\mathbb{C}[\apvar_N]\,[[\apvar_{N-1}/\apvar_{N}]] \cdots [[\apvar_{1}/\apvar_{2}]]\,[\apvar_1^{-1}] \;
    \apvar_{N}^{\Delta_{N}} \cdots \apvar_{1}^{\Delta_{1}}.
\end{align*}

The following is our main result about the equality of the
quantum group functions and (the highest weight matrix elements of) the
compositions of intertwining
operators of the first row subcategory of the generic Virasoro VOA.
\begin{thm}
\label{thm: quantum group functions and compositions of intertwining ops}
For any
$\lambdafullseq =
(\lambda_{0},\lambda_{1},\dots, \lambda_{N},\lambda_{\infty})\in \Znn^{N+2}$
and any $\lambdafullseq$-admissible sequence $\cbfullseq=(\cbseq_{0},\dots,\cbseq_{N})$,
the function~$\sF[u_{\cbfullseq}]$ admits a series expansion
\begin{align*}
\sF[u_{\cbfullseq}](0,x_1,\ldots,x_N)
= \; & \sum_{k_1, \ldots, k_N \in \Z} c_{k_1, \ldots, k_N} \,
    x_1^{\Delta_1+k_1} \cdots x_N^{\Delta_N+k_N} ,
\end{align*}
which is convergent for any~$(x_1, \ldots, x_N) \in \chamber_N^+$.
As a formal series, this coincides with
the matrix element~\eqref{eq: highest weight matrix element of a composition}
of the composition of intertwining operators,
\begin{align*}
\sF[u_{\cbfullseq}](0,\apvar_1,\ldots,\apvar_N)
= \; & C^{\lambdafullseq}_{\cbfullseq}(\apvar_{1},\dots, \apvar_{N})
\; \in \; \mathbb{C}[[\apvar_{N-1}/\apvar_{N}]] \cdots [[\apvar_{1}/\apvar_{2}]] \,
     \apvar_{N}^{\Delta_{N}} \cdots \apvar_{1}^{\Delta_{1}} .
\end{align*}
\end{thm}

\begin{proof}
We employ an induction over the number of variables $N$.
The case of $N=1$ is obvious; the function is
\begin{equation*}
	\widetilde{\sF}[u_{\cbfullseq}](x)=\betacoef{\lambda_{1}}{\lambda_{0}}{\lambda_{\infty}}x^{\Delta_{1}},\quad x>0.
\end{equation*}
This function gives the same one term power series
in~$\C \apvar^{\Delta_{1}}$ as the matrix
element~$C^{\lambdafullseq}_{\cbfullseq}(\apvar)$.

Assume that the assertion is true
when the number of variables is strictly less than~$N$.
Then use Lemma~\ref{lem: Frobenius series of QG functions}
to obtain a Frobenius expansion 
of~$\widetilde{\sF}[u_{\cbfullseq}]$ at $x_{1}=0$,
\begin{equation}
\label{eq:Frobenius_expansion_reduced_qg_function}
	\widetilde{\sF}[u_{\cbfullseq}](x_{1},\dots,x_{N})=\sum_{n=0}^{\infty}\widetilde{F}_{n}(x_{2},\dots, x_{N})x_{1}^{\Delta_{1}+n}
\end{equation}
Note that by Remark~\ref{rmk: leading coefficient of Frobenius series},
the initial coefficient~$\widetilde{F}_{0}$ coincides with
$\betacoef{\lambda_{1}}{\lambda_{0}}{\cbseq_{1}} \widetilde{\sF}[u_{\cbfullseq'}]$,
where ${\cbfullseq'=(\cbseq_{1},\dots,\cbseq_{N})}$
is a $(\cbseq_{1},\lambda_{2},\dots,\lambda_{N},\lambda_{\infty})$-admissible
sequence.

We apply $\varBSAoper{1}$ to both sides of the expansion (\ref{eq:Frobenius_expansion_reduced_qg_function}).
Because the operators act coefficientwise, we may
employ similar arguments as in the proof of Theorem \ref{thm:uniqueness_differential_equations_formal_series}.
We find that $\widetilde{F}_{n}$, $n\geq 1$,
are determined recursively as
\begin{equation*}
\widetilde{F}_{n}(x_{2},\dots,x_{N})
= - \frac{\apvar_{1}^{-\Delta_{1}-n+\lambda_{1}+1}}{P_{\lambda_{1}}(\hwFR{\lambda_{0}},\hwFR{\varsigma_{1}}+n)}
  \sum_{m=0}^{n-1} \extraterm^{(1)}_{n-m}
  \left(\apvar_{1}^{\Delta_{1}+m}\widetilde{F}_{m}(x_{2},\dots, x_{N})\right) ,
\end{equation*}
where $\extraterm^{(1)}_{j}$, $j \in \Znn$, are differential
operators given by exactly the same formulas as in the proof of
Theorem~\ref{thm:uniqueness_differential_equations_formal_series},
but now acting on the
space~$\ContF^\infty(\chamber_{N-1}^{+})[[\apvar_1^{\pm 1}]] \, \apvar_1^{\Delta_1}$.

In particular, $\widetilde{F}_{n}$, $n\geq 1$, are
determined by $\widetilde{F}_{0}$ via the same relations as $C_{n}$,
$n\geq 1$, are determined by $C_{0}$.
It also follows that all higher coefficients $F_{n}$, $n\geq 1$, are
analytic on~$\chamber_{N-1}^{+}$.
By the induction hypothesis, $\widetilde{F}_{0}$ is expanded in
\begin{equation*}
	\C [[\apvar_{N-1}/\apvar_{N}]]\cdots [[\apvar_{2}/\apvar_{3}]]
	\, \apvar_{N}^{\Delta_{N}}\cdots \apvar_{2}^{\Delta_{2}}
\end{equation*}
and coincides with $C_{0}$.
Therefore, each $\widetilde{F}_{n}$, $n\geq 0$ is expanded in
\begin{equation*}
\C [[\apvar_{N-1}/\apvar_{N}]]\cdots [[\apvar_{2}/\apvar_{3}]]
  \, \apvar_{N}^{\Delta_{N}}\cdots \apvar_{2}^{\Delta_{2}-n}
\end{equation*}
and coincides with $C_{n}$.
Finally, from the expansion (\ref{eq:Frobenius_expansion_reduced_qg_function}), we conclude that $\widetilde{\sF}[u_{\cbfullseq}]$ is expanded in
\begin{equation*}
\C [[\apvar_{N-1}/\apvar_{N}]]\cdots [[\apvar_{1}/\apvar_{2}]]
	\, \apvar_{N}^{\Delta_{N}}\cdots \apvar_{1}^{\Delta_{1}}
\end{equation*}
and coincides with $C^{\lambdafullseq}_{\cbfullseq}$.
The convergence of this series at any
$(x_1 , \ldots, x_N) \in \chamber_{N}^{+}$
is clear by an inductive application of
Lemma~\ref{lem: Frobenius series of QG functions}.
\end{proof}

\subsection{Some applications}
\label{sub: applications of the composition result}
Let us quickly comment on applications of the above result to
the analysis of the PDE system and to the quantum group method itself.

\subsubsection*{On the solution spaces to BPZ differential equations}

By Theorem~\ref{thm: quantum group functions and compositions of intertwining ops},
for every $u \in
\big( \bigotimes_{j=1}^N \QGrep{\lambda_j} \big) \tens \QGrep{\lambda_{0}}$
such that $E u = 0$ we associate a formal series
representing the function~$\widetilde{\sF}[u] \colon \chamber_{N}^+ \to \bC$,
\begin{align*}
\widetilde{\sF}[u] \; \in \; \bC \{ \apvar_1 , \ldots, \apvar_N \} .
\end{align*}
By convergence of the series
(pointwise) in~$\chamber_{N}^+$, the series uniquely determines
the function~$\widetilde{\sF}[u]$, and by translation invariance it therefore also
determines~$\sF[u] \colon \chamber_{N+1} \to \bC$.

For fixed 
$\lambdafullseq =
(\lambda_{0},\lambda_{1},\dots, \lambda_{N},\lambda_{\infty})\in \Znn^{N+2}$,
any $\lambdafullseq$-admissible sequence $\cbfullseq=(\cbseq_{0},\dots,\cbseq_{N})$
gives a solution to the (same) system of BPZ partial differential equations
of BSA form,
\begin{align}\label{eq: BPZ PDE system of BSA type}
\BSAoper{j} 
F = 0 \qquad \text{ for } j = 0,\ldots,N,
\end{align}
with homogeneity
\begin{align}\label{eq: homogeneity for BPZ PDE system of BSA type}
F(s x_0 , s x_1 , \ldots , s x_N)
= s^{\hwFR{\lambda_\infty} - \sum_{j=0}^N \hwFR{\lambda_j}} \, 
    F(x_0 , x_1 , \ldots , x_N) 
\end{align}
for any~$s>0$ and~$(x_0, x_1, \ldots, x_N) \in \chamber_{N+1}$.

For different~$\cbfullseq$,
the corresponding formal series in
$\bC \{ \apvar_1 , \ldots, \apvar_N \}$
are clearly linearly independent, since the sequences
$(\Delta_1, \ldots,\Delta_N)$ of exponents
$\Delta_{i}=\hwFR{\cbseq_{i}}-\hwFR{\lambda_{i}}-\hwFR{\cbseq_{i-1}}$,
$i=1,\ldots,N$, are distinct in the generic
case. In particular the conformal block type
solutions~$\sF[u_{\cbfullseq}]$, for $\cbfullseq$ admissible,
are linearly independent solutions to the system of differential
equations.

As an application, we therefore get that the dimension of the solution
space is at least the number of conformal blocks,
which is a combinatorial and well understood quantity.

\begin{cor}
Let
$\lambdafullseq =
(\lambda_{0},\lambda_{1},\dots, \lambda_{N},\lambda_{\infty})\in \Znn^{N+2}$.
Then we have
\begin{align*}
\dmn \set{ F \in \ContF^\infty (\chamber_{N+1}) \; \bigg| \;
   F \text{ satisfies \eqref{eq: BPZ PDE system of BSA type}
    and \eqref{eq: homogeneity for BPZ PDE system of BSA type}}
    } 
\; \geq \;\; & \dmn \; \Hom_{\Uqsltwo} \bigg(
    \QGrep{\lambda_{\infty}} , \;
    \bigotimes_{j=0}^N \QGrep{\lambda_j} \bigg) .
\end{align*}
\end{cor}

\subsubsection*{Differential equations at infinity}

From the perspective of CFT, 
the (PDE) part of Theorem~\ref{thm:properties_of_correlations}
has the interpretation that 
for~$u\in \HWspace{\lambdafullseq}$, the function~$\sF [u]$
satisfies partial differential equations stemming
from degeneracies of primary fields at the
points~$x_0, x_1, \ldots, x_N$.
When we take $u$ to moreover satisfy $K.u = q^{\lambda_\infty} \, u$,
it is natural to expect
the function~$\sF [u]$ to satisfy a further 
differential equation of order~$\lambda_\infty + 1$,
associated with the field at infinity.
It is possible to give a direct proof of this property from
the quantum group method, but such a proof is not entirely trivial.
By contrast, a very simple proof can be obtained as an application
of the results of this section.

For the statement, we introduce the following
differential operators at infinity
\begin{align*}
	\WittL{n}{\infty}=\sum_{i=0}^{N}\left(x_{i}^{n+1}\frac{\dee}{\dee x_{i}}+(n+1)\hwFR{\lambda_{i}}x_{i}^{n}\right), \quad n\in \Z.
\end{align*}

\begin{cor}
\label{cor:diff_eq_at_infty}
Let $u\in \HWspace{\lambdafullseq}\cap \Kern (K-q^{\lambda_{\infty}})$, $\lambda_{\infty}\in \Znn$. The function $\sF [u]$ satisfies the differential equation
$\BSAoper{\infty} \, \sF [u] = 0$, where
\begin{equation*}
	\BSAoper{\infty}=\sum_{k=1}^{\lambda_{\infty}+1}\sum_{\substack{p_{1},\dots, p_{k}\geq 1 \\ p_{1}+\cdots +p_{k}=\lambda_{\infty}+1}}\frac{(-4/\kappa)^{\lambda_{\infty}+1-k}(\lambda_{\infty}!)^{2}}{\prod_{u=1}^{k-1}(\sum_{i=1}^{u}p_{i})(\sum_{i=u+1}^{k}p_{i})}\WittL{p_{1}}{\infty} \cdots \WittL{p_{k}}{\infty}.
\end{equation*}
\end{cor}
\begin{proof}
Note that 
any $u\in \HWspace{\lambdafullseq}\cap \Kern (K-q^{\lambda_{\infty}})$
is a linear combination of $u_{\cbfullseq}$,
where $\cbfullseq$ ranges over the $\lambdafullseq$-admissible sequences.
By linearity, it suffices to prove the statement for $u = u_{\cbfullseq}$.

To simplify, we make use of the fact that the representation~$\QGrep{0}$ is
the unit for tensor products of representations of $\Uqsltwo$,
as seen from~\eqref{eq:tensor_prod_Uq_mods}.
Explicitly, we use the isomorphism
\begin{align*}
\big( \bigotimes_{j=1}^N \idof{\QGrep{\lambda_j}} \big) \tens \QGembed{\lambda_0}{\lambda_0}{0} \; \colon \;
\big( \bigotimes_{j=1}^N \QGrep{\lambda_j} \big) \tens \QGrep{\lambda_{0}}
\; \overset{\isom}{\longrightarrow} \;
\big( \bigotimes_{j=1}^N \QGrep{\lambda_j} \big) \tens \QGrep{\lambda_{0}} \tens \QGrep{0} .
\end{align*}
We set
$u'_{\cbfullseq} = \big( \big( \bigotimes_{j=1}^N \idof{\QGrep{\lambda_j}} \big) \tens \QGembed{\lambda_0}{\lambda_0}{0} \big) (u_{\cbfullseq})$, and note that the function
\begin{align*}
\sF [u'_{\cbfullseq}] (z , x_0 , x_1 , \ldots , x_N)
\end{align*}
satisfies the first order differential equation
$0 = \pder{z} \, \sF [u'_{\cbfullseq}]$ by the (PDE) property
of Theorem~\ref{thm:properties_of_correlations}.
Therefore it is constant as a function of~$z$.
When we specialize to~$z=0$, 
Theorem~\ref{thm: quantum group functions and compositions of intertwining ops}
states that this function is given by
\begin{align*}
\sF [u'_{\cbfullseq}] (0 , x_0 , x_1 , \ldots , x_N)
= \; & 
    \Braket{ \hwvFR{\cbseq_{N}}^{\prime} , \,
    \intertwnorm{\cbseq_{N-1}}{\lambda_{N}}{\cbseq_{N}}(\hwvFR{\lambda_{N}} , x_{N})
\cdots
\intertwnorm{\cbseq_{0}}{\lambda_{1}}{\cbseq_{1}}(\hwvFR{\lambda_{1}} , x_{1}) \,
\intertwnorm{0}{\lambda_{0}}{\cbseq_{0}}(\hwvFR{\lambda_{0}} , x_{0}) \,
    \hwvFR{0}} .
\end{align*}
On the other hand, by the (ASY) property 
of Theorem~\ref{thm:properties_of_correlations}
and the observation that ${\betacoef{\lambda_{0}}{0}{\,\lambda_{0}} = 1}$,
we can determine the value of the constant, and we find
\begin{align*}
\sF [u'_{\cbfullseq}] (0 , x_0 , x_1 , \ldots , x_N)
\; = \; \sF [u_{\cbfullseq}] (x_0 , x_1 , \ldots , x_N) .
\end{align*}

It now suffices to prove the differential equation for the above expression
involving a composition of intertwining operators.
Recall that $\voaFR{\lambda_\infty}' \isom \voaFR{\lambda_\infty}$,
so $\singvecop_{\lambda_\infty} \hwvFR{\lambda_\infty}' = 0$, and therefore
\begin{align*}
0 \; = \;\; &
\Braket{ \singvecop_{\lambda_\infty} \hwvFR{\cbseq_{N}}^{\prime} , \,
    \intertwnorm{\cbseq_{N-1}}{\lambda_{N}}{\cbseq_{N}}(\hwvFR{\lambda_{N}} , \apvar_{N})
\cdots
\intertwnorm{0}{\lambda_{0}}{\cbseq_{0}}(\hwvFR{\lambda_{0}} , \apvar_{0}) \,
    \hwvFR{0}} \\
= \;\; &
    \sum_{k=1}^{\lambda_{\infty}+1}\sum_{\substack{p_{1},\dots, p_{k}\geq 1 \\ p_{1}+\cdots +p_{k}=\lambda_{\infty}+1}}\frac{(-4/\kappa)^{\lambda_{\infty}+1-k}(\lambda_{\infty}!)^{2}}{\prod_{u=1}^{k-1}(\sum_{i=1}^{u}p_{i})(\sum_{i=u+1}^{k}p_{i})} \\
& \qquad\qquad \Braket{ \hwvFR{\cbseq_{N}}^{\prime} , \; L_{p_k} \cdots L_{p_1} \,
    \intertwnorm{\cbseq_{N-1}}{\lambda_{N}}{\cbseq_{N}}(\hwvFR{\lambda_{N}} , \apvar_{N})
\cdots
\intertwnorm{0}{\lambda_{0}}{\cbseq_{0}}(\hwvFR{\lambda_{0}} , \apvar_{0}) \,
    \hwvFR{0}} .
\end{align*}
We can commute the Virasoro generators~$L_{p_m}$ to the right
by Corollary~\ref{cor:formula_intertw} in the
form~\eqref{eq: primary field Virasoro commutation} again.
For $p_m>0$ we have $L_{p_m} \hwvFR{\lambda_0} = 0$, so after commutation we find
the desired differential equation.
\end{proof}

By contrast, 
to prove Corollary~\ref{cor:diff_eq_at_infty} directly using the quantum
group method, one would have to introduce an auxiliary variable
$x_{N+1}$ and an auxiliary vector
\begin{align*}
u'' \in \QGrep{\lambda_{\infty}} \tens \big( \bigotimes_{j=1}^N \QGrep{\lambda_j} \big) \tens \QGrep{\lambda_{0}}
\end{align*}
as in~\cite[Section 5.2]{KP-conformally_covariant_bdry_correlations},
conjugate by an appropriate exponentiatial of~$\WittL{1}{N+1}$,
and take a complicated limit $x_{N+1} \to +\infty$.
The proof above is a significant simplification.

\subsubsection*{Full series expansions of the functions from the quantum group method}

The property (ASY) in Theorem~\ref{thm:properties_of_correlations}
gives the explicit leading coefficient of the Frobenius series expansion 
of a function obtained by the quantum group method, in the
limit~$x_j-x_{j-1} \to 0$.
The main result of this section
allows to express also all of the
higher order coefficients of this Frobenius series
as matrix elements of
compositions of intertwining operators\footnote{Specifically,
the method as presented in this section, applies to the
coefficients of the Frobenius series corresponding
to~$x_1-x_0 \to 0$.
The general case of expansions as
$x_j-x_{j-1} \to 0$ for~$j>1$ will be 
obtained by associativity, which is addressed in the next section.},
making their calculation tractable by
algebraic and combinatorial techniques.
This is perhaps the most important direct
application of the results of this section to the quantum group method.

\section{Associativity of intertwining operators}%
\label{sec: associativity}
In this section we will prove one of the main results of the paper,
the associativity of the intertwining operators among the modules
of the first row subcategory.
This result is essentially the associativity in an appropriate
tensor category.

From the point of view of correlation functions of conformal field theories,
associativity amounts to the fact that the same correlation function
admits series expansions in different regimes, and the resulting
series represent the same function on the overlap of the domains where they
converge.
A geometric interpretation of this property is that different
pair-of-pants decompositions of the same Riemann surface can be used
interchangeably.

In VOA theory, associativity is a statement about the case~$N=2$ only,
but involving the full 
intertwining operators, not merely a matrix element that gives
a particular correlation function.
Starting from the $N=2$ case, it is then possible to inductively obtain
coincidence of various expansions of multipoint correlation functions
(or interchangeability of pair-of-pants decompositions of multiply
punctured spheres).

From the point of view of VOA intertwining operators, 
we will be comparing a priori entirely different formal series,
and it is far from obvious that they should represent the same function.
Indeed, to obtain coincidence, we cannot separate a given
conformal block from the rest: the expansion of a single conformal
block in another regime will involve all possible conformal blocks,
so a carefully devised linear combination is needed.
The appropriate coefficients of the linear combination
will be the $6j$-symbols of the underlying quantum group.

A key subtlety is that the analytic functions represented by the
series are multivalued (when extended to their natural complex domains),
and the correct coincidence statements necessarily
involve branch choices. In our setup, the quantum group method and the
ordering of the variables on the real line leads to convenient
branch choices which facilitate the statement.
The multivaluedness, in turn, is closely related to braiding
properties, which feature crucially in the analysis of CFT correlation
functions, and form a key structure of the underlying tensor category.

In the previous sections, the quantum group has
been used mainly through the 
fusion rules of the VOA modules, which matched the selection rules
for $\Uqsltwo$ representations. The selection rules themselves, however,
are not sensitive to the deformation parameter~$q$.
By contrast, the results of this section will
involve the specific $6j$ symbols, and they will lead to the specific
braiding properties. The associativity result therefore gives more profound
evidence for the equivalence of the first row subcategory of the
generic Virasoro VOA and the category of (type-one) finite-dimensional
representations of~$\Uqsltwo$.

This section is organized as follows.
In Section~\ref{sub: statement of associativity} we introduce the setup
and give the precise formulation
of the associativity statement, Theorem~\ref{thm:associativity}.
The proof of that statement is divided in two parts.
Section~\ref{sub: associativity for the highest weight matrix element}
contains the proof of a particular case involving a highest weight
matrix element.
In Section~\ref{sub: associativity general case}, an inductive
construction starting from that particular case is used to
finish the proof of the general case.

\subsection{The setup and statement of associativity}
\label{sub: statement of associativity}

Fix $\lambda_{0},\lambda_{1},\lambda_{2},\lambda_{\infty}\in \Znn$.
For arbitrary
$\Mvec_{0} \in \voaFR{\lambda_{0}}$,
$\Mvec_{1} \in \voaFR{\lambda_{1}}$,
$\Mvec_{2} \in \voaFR{\lambda_{2}}$,
$\Mvec_{\infty}' \in \voaFR{\lambda_{\infty}}^{\prime}$,
and $\cbseq\in \selRule{\lambda_{0}}{\lambda_{1}}$,
we will consider the formal series
\begin{align}\label{eq: associativity series A}
\Braket{\Mvec_{\infty}', \; \intertwnorm{\cbseq}{\lambda_{2}}{\lambda_{\infty}}(\Mvec_{2}, \apvar_{2})
    \,  \intertwnorm{\lambda_{0}}{\lambda_{1}}{\cbseq}(\Mvec_{1}, \apvar_{1}) \, \Mvec_{0}} ,
\end{align}
which is a special case $N=2$ of the series in Section~\ref{sec: composition},
and we will compare it with the formal series
\begin{align}\label{eq: associativity series B}
\Braket{\Mvec_{\infty}', \; \intertwnorm{\lambda_{0}}{\cbseqB}{\lambda_{\infty}}
    \big(\intertwnorm{\lambda_{1}}{\lambda_{2}}{\cbseqB}(\Mvec_{2}, \apvarB) \, \Mvec_{1}, \, \apvar \big)
    \, \Mvec_{0}} 
\end{align}
for $\cbseqB \in\selRule{\lambda_{1}}{\lambda_{2}}\cap\selRule{\lambda_{0}}{\lambda_{\infty}}$.
The idea will be to substitute actual values $x_1 , x_2 \in \bR$ so that
\begin{align*}
&\begin{cases}
\apvar_1 = x_1 \\
\apvar_2 = x_2 
\end{cases} \qquad \text{ and} \qquad\quad
\begin{cases}
\apvar = x_1  \\
\apvarB = x_2-x_1 .
\end{cases}
\end{align*}
The first series~\eqref{eq: associativity series A} will correspond to an
expansion in the regime~$0 < x_1 \ll x_2$,
and the second~\eqref{eq: associativity series B}
to and expansion in the regime~$0 < x_2-x_1 \ll x_1$.
As discussed in Section~\ref{sec: composition}, the
series~\eqref{eq: associativity series A} is in the space
\begin{align*}
\mathbb{C}[\apvar_2]\,[[\apvar_{1}/\apvar_{2}]]\,[\apvar_1^{-1}]
    \; \apvar_{1}^{\deltasym_1} \, \apvar_{2}^{\deltasym_2}
\qquad \text{ with } \qquad
\begin{cases}
\deltasym_1 = \hwFR{\cbseq} - \hwFR{\lambda_1} - \hwFR{\lambda_0} \\
\deltasym_2 = \hwFR{\lambda_\infty} - \hwFR{\lambda_2} - \hwFR{\cbseq} .
\end{cases}
\end{align*}
Similarly we will see that the
series~\eqref{eq: associativity series B} is in the space
\begin{align*}
\mathbb{C}[\apvar]\,[[\apvarB/\apvar]] \, [\apvarB^{-1}] \; \apvar^{\deltasymB'(\cbseqB)} 
    \, \apvarB^{\deltasymB(\cbseqB)}
\qquad \text{with } \qquad
\begin{cases}
\deltasymB(\cbseqB) = \hwFR{\cbseqB} - \hwFR{\lambda_2} - \hwFR{\lambda_1} \\
\deltasymB'(\cbseqB) = \hwFR{\lambda_\infty} - \hwFR{\cbseqB} - \hwFR{\lambda_0} .
\end{cases}
\end{align*}

The goal of this section is to prove the following
associativity of intertwining operators.
\begin{thm}
\label{thm:associativity}
Fix $\lambda_{0},\lambda_{1},\lambda_{2},\lambda_{\infty}\in \Znn$.
For arbitrary
$\Mvec_{0} \in \voaFR{\lambda_{0}}$,
$\Mvec_{1} \in \voaFR{\lambda_{1}}$,
$\Mvec_{2} \in \voaFR{\lambda_{2}}$,
$\Mvec_{\infty}' \in \voaFR{\lambda_{\infty}}^{\prime}$,
and $\cbseq\in \selRule{\lambda_{0}}{\lambda_{1}}$,
the series
\begin{align*}
\Braket{\Mvec_{\infty}', \; \intertwnorm{\cbseq}{\lambda_{2}}{\lambda_{\infty}}(\Mvec_{2}, x_{2})
    \,  \intertwnorm{\lambda_{0}}{\lambda_{1}}{\cbseq}(\Mvec_{1}, x_{1}) \, \Mvec_{0}} ,
\end{align*}
converges when $0 < x_1 < x_2$, and the series
\begin{align*}
\sum_{\cbseqB \, \in \, \selRule{\lambda_{1}}{\lambda_{2}}
        \, \cap \, \selRule{\lambda_{0}}{\lambda_{\infty}}}
	\sixj{\lambda_{\infty}}{\lambda_2}{\lambda_1}{\lambda_0}{\cbseq}{\cbseqB}
\Braket{\Mvec_{\infty}', \; \intertwnorm{\lambda_{0}}{\cbseqB}{\lambda_{\infty}}
    \big(\intertwnorm{\lambda_{1}}{\lambda_{2}}{\cbseqB}(\Mvec_{2}, x_2-x_1) \, \Mvec_{1}, \, x_1 \big)
    \, \Mvec_{0}} 
\end{align*}
converges when $0 < x_2-x_1 < x_1$, and in the (nontrivial) overlap
of these domains, the analytic functions represented by these two series coincide.
\end{thm}

The proof of Theorem~\ref{thm:associativity} consists of two main steps:
the special case of the matrix element of highest weight vectors,
and the recursion on the PBW filtration to establish the general
case. We address the two in separate subsections below.

In the rest of the section, sums 
over~$\cbseqB$ are always taken over the same set of values satisfying the two
selection rules above.
For brevity, we omit this from the notation and write just 
\begin{align*}
\sum_{\cbseqB} \Big( \cdots \Big) \; := \;
\sum_{\cbseqB \, \in \, \selRule{\lambda_{1}}{\lambda_{2}}
        \, \cap \, \selRule{\lambda_{0}}{\lambda_{\infty}}} \Big( \cdots \Big) .
\end{align*}

\subsection{The special case with highest weight vectors}
\label{sub: associativity for the highest weight matrix element}

The first step in the proof of Theorem~\ref{thm:associativity} is
to consider the special case where instead of general vectors
$\Mvec_{0} \in \voaFR{\lambda_{0}}$,
$\Mvec_{1} \in \voaFR{\lambda_{1}}$,
$\Mvec_{2} \in \voaFR{\lambda_{2}}$,
$\Mvec_{\infty}' \in \voaFR{\lambda_{\infty}}^{\prime}$,
we use the highest weight vectors
$\hwvsym_{\lambda_{0}} \in \voaFR{\lambda_{0}}$,
$\hwvsym_{\lambda_{1}} \in \voaFR{\lambda_{1}}$,
$\hwvsym_{\lambda_{2}} \in \voaFR{\lambda_{2}}$,
$\hwvsym_{\lambda_{\infty}}' \in \voaFR{\lambda_{\infty}}^{\prime}$.
This is analogous to the initial term of an intertwining
operator, and the highest weight matrix element considered
in the context of compositions of intertwining operators.
The precise statement we want to prove is the following.
\begin{prop}
\label{prop:associativity_initial_term}
Let $\lambda_{0},\lambda_{1},\lambda_{2},\lambda_{\infty}\in \Znn$.
For arbitrary $\cbseq \in \selRule{\lambda_{0}}{\lambda_{1}}$, the two series
\begin{align*}
\Braket{\hwvsym_{\lambda_{\infty}}', \; \intertwnorm{\cbseq}{\lambda_{2}}{\lambda_{\infty}}(\hwvsym_{\lambda_{2}}, x_{2})
    \,  \intertwnorm{\lambda_{0}}{\lambda_{1}}{\cbseq}(\hwvsym_{\lambda_{1}}, x_{1}) \, \hwvsym_{\lambda_{0}}} ,
\end{align*}
converges when $0 < x_1 < x_2$, and the series
\begin{align*}
\sum_{\cbseqB}
	\sixj{\lambda_{\infty}}{\lambda_2}{\lambda_1}{\lambda_0}{\cbseq}{\cbseqB}
\Braket{\hwvsym_{\lambda_{\infty}}', \; \intertwnorm{\lambda_{0}}{\cbseqB}{\lambda_{\infty}}
    \big(\intertwnorm{\lambda_{1}}{\lambda_{2}}{\cbseqB}(\hwvsym_{\lambda_{2}}, x_2-x_1) \, \hwvsym_{\lambda_{1}}, \, x_1 \big)
    \, \hwvsym_{\lambda_{0}}} 
\end{align*}
converges when $0 < x_2-x_1 < x_1$, and in the (nontrivial) overlap
of these domains, the analytic functions represented by these two series coincide.
\end{prop}

The first of the two series is exactly of the form considered in
Section~\ref{sec: composition}, and by
Theorem~\ref{thm: quantum group functions and compositions of intertwining ops}
we have
\begin{align*}
\sF [u_{\cbfullseq}] (0,x_1, x_2)
\; = \;  \Braket{\hwvsym_{\lambda_{\infty}}^{\prime}, \;
    \intertwnorm{\cbseq}{\lambda_{2}}{\lambda_{\infty}}(\hwvsym_{\lambda_{2}},x_{2}) \,
    \intertwnorm{\lambda_{0}}{\lambda_{1}}{\cbseq}(\hwvsym_{\lambda_{1}},x_{1}) \,
    \hwvsym_{\lambda_{0}}}
\qquad \text{for $0<x_1<x_2$,}
\end{align*}
where $u_{\cbfullseq} \in \QGrep{\lambda_2} \tens \QGrep{\lambda_1} \tens \QGrep{\lambda_0}$
is the conformal block vector~\eqref{eq: construction of conformal block vector}
\begin{align*}
u_{\cbfullseq}
\; = \; \Big( \big( \id_{\QGrep{\lambda_2}} \tens \QGembed{\cbseq}{\lambda_1}{\lambda_0} \big)
        \circ \QGembed{\lambda_\infty}{\lambda_2}{\cbseq} \Big) \, (\QGhwv{\lambda_{\infty}}) .
\end{align*}
For the purposes of comparing with the other series in the associativity
statement, we decompose this vector using the defining 
property~\eqref{eq: 6j decomposition}
\begin{align*}
\Big( \big( \id_{\QGrep{\lambda_2}} \tens \QGembed{\cbseq}{\lambda_1}{\lambda_0} \big)
        \circ \QGembed{\lambda_\infty}{\lambda_2}{\cbseq} \Big)
= \; & \sum_{\cbseqB} \sixj{\lambda_\infty}{\lambda_2}{\lambda_1}{\lambda_0}{\cbseq}{\cbseqB}
\big( \QGembed{\cbseqB}{\lambda_2}{\lambda_1} \tens \id_{\QGrep{\lambda_0}} \big)
\circ \QGembed{\lambda_\infty}{\cbseqB}{\lambda_0}
\end{align*}
of $6j$-symbols, to get
\begin{align*}
u_{\cbfullseq} \; = \;
\sum_{\cbseqB} \sixj{\lambda_\infty}{\lambda_2}{\lambda_1}{\lambda_0}{\cbseq}{\cbseqB}
    \; \widehat{u}_{\cbfullseqB} ,
\qquad \text{ where} \qquad
\widehat{u}_{\cbfullseqB}
\; = \; \Big( \big( \QGembed{\cbseqB}{\lambda_2}{\lambda_1} \tens \id_{\QGrep{\lambda_0}} \big)
\circ \QGembed{\lambda_\infty}{\cbseqB}{\lambda_0} \Big) \, (\QGhwv{\lambda_{\infty}}) .
\end{align*}
For the corresponding function, we get the decomposition
\begin{align*}
\sF[u_{\cbfullseq}](0,x_1,x_2) \; = \;
\sum_{\cbseqB} \sixj{\lambda_\infty}{\lambda_2}{\lambda_1}{\lambda_0}{\cbseq}{\cbseqB}
    \; \sF[\widehat{u}_{\cbfullseqB}](0,x_1,x_2) ,
\end{align*}
where each term on the right hand side has a good Frobenius series
expansion in the variable~$x_2-x_1$, by virtue of the submodule
projection property~$\widehat{u}_{\cbfullseqB} 
= \QGlongprcan{1}{2}{\cbseqB} (\widehat{u}_{\cbfullseqB})$
and Lemma~\ref{lem: Frobenius series of QG functions}.
Specifically, for any~$\cbseqB$, in the region~$0 < x_2-x_1 < x_1$, we have
\begin{align}\label{eq: another Frobenius series}
\sF[\widehat{u}_{\cbfullseqB}](0,x_1,x_2)
\; = \; (x_2-x_1)^{\deltasymB(\cbseqB)} \,
    \sum_{k=0}^\infty \widehat{C}_k (x_1) \, (x_2-x_1)^k ,
\end{align}
where $\deltasymB(\cbseqB) = \hwFR{\cbseqB} - \hwFR{\lambda_2} - \hwFR{\lambda_1}$
and the leading coefficient is
\begin{align*}
\widehat{C}_0 (x_1) 
\; = \;\; & \betacoef{\lambda_{2}}{\lambda_{1}}{\;\;\cbseqB} \;
    \sF \big[ \QGembed{\lambda_\infty}{\cbseqB}{\lambda_0} (\QGhwv{\lambda_{\infty}}) \big] (0, x_1)
\; = \; \betacoef{\lambda_{2}}{\lambda_{1}}{\;\;\cbseqB} \,
    \betacoef{\cbseqB}{\lambda_{0}}{\,\lambda_\infty} \; x_1^{\deltasymB'(\cbseqB)} 
\end{align*}
with
$\deltasymB'(\cbseqB) = \hwFR{\lambda_\infty} - \hwFR{\cbseqB} - \hwFR{\lambda_0}$.
To prove Proposition~\ref{prop:associativity_initial_term}
therefore amounts to showing that the 
series~\eqref{eq: another Frobenius series}
coincides with the other expression with VOA intertwining operators,
\begin{align*}
\Braket{\hwvsym_{\lambda_{\infty}}', \; \intertwnorm{\lambda_{0}}{\cbseqB}{\lambda_{\infty}}
    \big(\intertwnorm{\lambda_{1}}{\lambda_{2}}{\cbseqB}(\hwvsym_{\lambda_{2}}, \apvarB) \, \hwvsym_{\lambda_{1}}, \, \apvar \big)
    \, \hwvsym_{\lambda_{0}}} ,
\end{align*}
where one substitutes~$\apvar=x_1$ and~$\apvarB = x_2-x_1$.
For this, the strategy is again to show that both expressions satisfy
the same differential equation, and that the series solutions to it are unique
up to multiplicative constants.
The most straightforward approach would be to use more than one differential
equation (e.g., inductively, as in Section~\ref{sec: composition}),
but in the current case we can in fact bypass the need for all but one
differential equation by using a priori homogeneity information.

The appropriate differential operator now is
\begin{align}\label{eq: BSA differential operator new var}
\newvarBSAoper{2}
	=\sum_{k=1}^{\lambda_{2}+1}\sum_{\substack{p_{1},\dots, p_{k}\geq 1 \\ p_{1}+\cdots+p_{k}=\lambda_{2}+1}}\frac{(-4/\kappa)^{1+\lambda_{2}-k}\lambda_{2}!}{\prod_{u=1}^{k-1}(\sum_{i=1}^{u}p_{i})(\sum_{i=u+1}^{k}p_{i})}\newvarWittL{-p_{1}}{2} \cdots \newvarWittL{-p_{k}}{2}
\end{align}
where for $n \in \bZ$ we set
\begin{align}\label{eq: Witt differential operator new var}
\newvarWittL{n}{2}
:= \;&  - (-x-y)^{n}\left((-x-y) \pder{x}
	        +(1+n)\hwFR{\lambda_{0}}\right) \\
\nonumber
	& - (-y)^{n} \left((-y) \Big( \pder{x} - \pder{y} \Big)+(1+n)\hwFR{\lambda_{1}}\right) .
\end{align}
Analogously to Section~\ref{sec: composition},
these differential operators act on the one hand on
spaces of smooth functions of~$x$ and~$y$, and on the other hand on
spaces
$\mathbb{C}[\apvar]\,[[\apvarB/\apvar]] \, [\apvarB^{-1}] \; \apvar^{\deltasymB'(\cbseqB)} 
    \, \apvarB^{\deltasymB(\cbseqB)}$
of formal series (for any~$\cbseqB$) via
\begin{align*}
x \mapsto \apvar , \qquad y \mapsto \apvarB,
\qquad 
\frac{\dee}{\dee x}\mapsto \frac{\dee}{\dee \apvar}, \qquad
\frac{\dee}{\dee y}\mapsto \frac{\dee}{\dee \apvarB}
\end{align*}
and the factors $(-x-y)^n$
are expanded as the following formal power series
\begin{align*}
(-x-y)^n \; \mapsto \; (-\apvar-\apvarB)^n \, = \,
    (-\apvar)^n \, \sum_{k=0}^\infty \binom{n}{k} \, (\apvarB/\apvar)^k .
\end{align*}
These differential operators~$\newvarBSAoper{2}$ and~$\newvarWittL{n}{2}$
are obtained from~$\varBSAoper{2}$ and~$\varWittL{n}{2}$ by changing
variables from $x_1, x_2$ to $x=x_1, y=x_2-x_1$. In particular the
following is obvious from the
property~$\varBSAoper{2} \, \sF[\widehat{u}_{\cbfullseqB}](0,x_1,x_2) \, = \, 0$,
which itself is obtained from the (PDE) part of
Theorem~\ref{thm:properties_of_correlations} using
Lemma~\ref{lem:translation_invariance_under_Witt_action}.
\begin{lem}\label{lem: associativity PDE for the quantum group function}
For any~$\cbseqB$ we have
\begin{align*}
\newvarBSAoper{2} \, \sF[\widehat{u}_{\cbfullseqB}](0,x,x+y) \; = \; 0 .
\end{align*}
\end{lem}

Next we show the analoguous property of the
formal series.

\begin{prop}
For each $\cbseqB \in \selRule{\lambda_{1}}{\lambda_{2}}\cap\selRule{\lambda_{0}}{\lambda_{\infty}}$,
the formal series 
\begin{align*}
\newhwme^{\lambdafullseq}_{\cbseqB}(\apvar,\apvarB)
:=\Braket{\hwvsym_{\lambda_{\infty}}^{\prime},
  \intertwnorm{\lambda_{0}}{\cbseqB}{\lambda_{\infty}} \big(
    \intertwnorm{\lambda_{1}}{\lambda_{2}}{\cbseqB}(\hwvsym_{\lambda_{2}}, \apvarB) \, \hwvsym_{\lambda_{1}},\apvar \big) \, \hwvsym_{\lambda_{0}}}
\end{align*}
lies in $\C[[\apvarB/\apvar]] \, \apvarB^{\deltasymB(\cbseqB)} \, \apvar^{\deltasymB'(\cbseqB)}$,
and it satisfies the differential equation
\begin{align*}
\newvarBSAoper{2} \, \newhwme^{\lambdafullseq}_{\cbseqB}(\apvar,\apvarB) \; = \; 0 .
\end{align*}
\end{prop}
\begin{proof}
We expand the intertwining operator $\intertwnorm{\lambda_{1}}{\lambda_{2}}{\cbseqB}(\hwvsym_{\lambda_{2}},\apvarB)$ so that
\begin{align*}
	\intertwnorm{\lambda_{1}}{\lambda_{2}}{\cbseqB}(\hwvsym_{\lambda_{2}}, \apvarB)
= \sum_{n\in\Z} \apvarB^{\deltasymB(\cbseqB)-n-1}(\hwvsym_{\lambda_{2}})_{(n)},
\end{align*}
where each $(\hwvsym_{\lambda_{2}})_{(n)}\in \Hom (\voaFR{\lambda_{1}},\voaFR{\cbseqB})$ is of degree $-n-1$, $n\in\Z$.
In particular, we have $(\hwvsym_{\lambda_{2}})_{(n)}\hwvsym_{\lambda_{1}}\in \voaFR{\cbseqB}(-n-1)$, $n\in\Z$, implying that $(\hwvsym_{\lambda_{2}})_{(n)}\hwvsym_{\lambda_{1}}=0$ unless $n\leq -1$.
Furthermore, we expand each part $\intertwnorm{\lambda_{0}}{\cbseqB}{\lambda_{\infty}}((\hwvsym_{\lambda_{2}})_{(n)}\hwvsym_{\lambda_{1}}, \apvar)$, $n\leq -1$ so that
\begin{align*}
	\intertwnorm{\lambda_{0}}{\cbseqB}{\lambda_{\infty}}((\hwvsym_{\lambda_{2}})_{(n)}\hwvsym_{\lambda_{1}}, \apvar)
	=\sum_{m\in\Z} \apvar^{\deltasymB'(\cbseqB)-m-1}\left((\hwvsym_{\lambda_{2}})_{(n)}\hwvsym_{\lambda_{1}}\right)_{(m)},
\end{align*}
where $\left((\hwvsym_{\lambda_{2}})_{(n)}\hwvsym_{\lambda_{1}}\right)_{(m)}\in \Hom (\voaFR{\lambda_{0}},\voaFR{\lambda_{\infty}})$, $m\in\Z$ is of degree $-n-m-2$.
Therefore, the matrix element
\begin{align*}
	\Braket{\hwvsym_{\lambda_{\infty}}^{\prime},\left((\hwvsym_{\lambda_{2}})_{(n)}\hwvsym_{\lambda_{1}}\right)_{(m)}\hwvsym_{\lambda_{0}}}
\end{align*}
vanishes unless $-n-m-2=0$.
Consequently, we see that the series $\newhwme^{\lambdafullseq}_{\cbseqB}(\apvar,\apvarB)$
is expanded as
\begin{align*}
\newhwme^{\lambdafullseq}_{\cbseqB}(\apvar,\apvarB)
\; = \;\; & \apvarB^{\deltasymB(\cbseqB)} \,\apvar^{\deltasymB'(\cbseqB)} \; \sum_{n=0}^{\infty}
	\Braket{\hwvsym_{\lambda_{\infty}}^{\prime},\left((\hwvsym_{\lambda_{2}})_{(-n-1)}\hwvsym_{\lambda_{1}}\right)_{(n-1)}\hwvsym_{\lambda_{0}}} (\apvarB/\apvar)^{n} \\
\in \;\; & \C [[\apvarB/\apvar]] \, \apvarB^{\deltasymB(\cbseqB)} \, \apvar^{\deltasymB'(\cbseqB)}.
\end{align*}

The differential equation follows from the quotienting
out of the singular vector
$\singvecop_{\lambda_2} \hwv{c,h(\lambda_2)}$ of the Verma module $\VermaCH{c}{\hwFR{\lambda_2}}$
given by~\eqref{eq:singular_vector};
in~$\voaFR{\lambda_2}$
we have
$\singvecop_{\lambda_2} \hwvsym_{\lambda_{2}} = 0$ and therefore
\begin{align*}
\Braket{\hwvsym_{\lambda_{\infty}}^{\prime},
	\intertwnorm{\lambda_{0}}{\cbseqB}{\lambda_{\infty}} \big(
	\intertwnorm{\lambda_{1}}{\lambda_{2}}{\cbseqB}
	    (\singvecop_{\lambda_2} \hwvsym_{\lambda_{2}}, \apvarB) \,
	\hwvsym_{\lambda_{1}}, \, \apvar \big) \, \hwvsym_{\lambda_{0}}} \; = \; 0.
\end{align*}
To see that this gives the differential equation of the asserted form,
the key is to observe that
for any $\Mvec_2\in \voaFR{\lambda_{2}}$ and $p>0$, we have
\begin{align*}
\Braket{\hwvsym_{\lambda_{\infty}}^{\prime},\intertwnorm{\lambda_{0}}{\cbseqB}{\lambda_{\infty}}
    \big( \intertwnorm{\lambda_{1}}{\lambda_{2}}{\cbseqB}(L_{-p} \Mvec_2 , \apvarB) \,
        \hwvsym_{\lambda_{1}}, \apvar \big) \, \hwvsym_{\lambda_{0}}}
\; = \; \newvarWittL{-p}{2} \Braket{\hwvsym_{\lambda_{\infty}}^{\prime},
    \intertwnorm{\lambda_{0}}{\cbseqB}{\lambda_{\infty}} \big(
    \intertwnorm{\lambda_{1}}{\lambda_{2}}{\cbseqB}(\Mvec_2 , \apvarB) 
        \,\hwvsym_{\lambda_{1}}, \, \apvar \big)\hwvsym_{\lambda_{0}}} ,
\end{align*}
which in turn follows by a calculation based on the formulas of
Corollary~\ref{cor:formula_intertw}. The calculation is otherwise
very similar to
that in the proof of Lemma~\ref{lem:PDE_matrix_elements},
except that part (a) of Corollary~\ref{cor:formula_intertw} has to be used
twice here.
\end{proof}
By the chosen normalization of the intertwining operators,
the coefficient of $\apvarB^{\deltasymB(\cbseqB)}\,\apvar^{\deltasymB'(\cbseqB)}$ in
$\newhwme^{\lambdafullseq}_{\cbseqB}(\apvar,\apvarB)$ is
$\betacoef{\cbseqB}{\lambda_{0}}{\,\lambda_\infty} \,
\betacoef{\lambda_{1}}{\lambda_{2}}{\;\cbseqB}$.
In particular, $\newhwme^{\lambdafullseq}_{\cbseqB}(\apvar,\apvarB)$
is non-zero.

The remaining core ingredient is a suitable uniqueness
statement for series form solutions of the PDE.
Two aspects of the statement here are worth noting.
First, we will have a uniqueness statement separately for
every~$\cbseqB$, in an appropriate space of formal series.
Different~$\cbseqB$ would give other linearly independent solutions,
but due to different characteristic exponents, the forms of the series
are different.
Second, although we require just
one differential equation, we obtain uniqueness, because
we additionally fix the total homogeneity degree. Alternatively
it would be possible to start without the homogeneity requirement,
using instead further differential equations that can also
be established in the present case.
\begin{prop}\label{prop: uniqueness of solutions for associativity}
For each $\cbseqB \in
\selRule{\lambda_{1}}{\lambda_{2}} \cap \selRule{\lambda_{0}}{\lambda_{\infty}}$,
subspace
\begin{align*}
\Kern \, \newvarBSAoper{2} \; \subset \; 
\mathbb{C}[[\apvarB/\apvar]] \; \apvar^{\deltasymB'(\cbseqB)} 
    \, \apvarB^{\deltasymB(\cbseqB)}
\end{align*}
consisting of the solutions to the above differential
equation is one-dimensional,
\begin{align*}
\dmn \Big( \Kern \, \newvarBSAoper{2} \Big)
\; = \; 1  .
\end{align*}
\end{prop}
\begin{proof}
Since $\newhwme^{\lambdafullseq}_{\cbseqB}$ is a non-zero solution
in this space of formal series,
it suffices to show that solutions are unique up to multiplicative constant.

For the present purpose, we introduce a $\bZ$-grading of the space
$\mathbb{C}[\apvar]\,[[\apvarB/\apvar]] \, [\apvarB^{-1}] \; 
\apvar^{\deltasymB'(\cbseqB)} \, \apvarB^{\deltasymB(\cbseqB)}$:
for $n,m\in\Z$, a monomial
$\apvarB^{\deltasymB(\cbseqB) + n} \, \apvar^{\deltasymB'(\cbseqB) + m}$ is
declared
to be of degree~$n$.
The degrees of operators are determined accordingly.

For each $p>0$,
the operator
\begin{align*}
\newldvarWittL{-p}{2}
\; = \; - (-\apvarB)^{-p} \, \left(\apvarB \pder{\apvarB} + (1-p) \, \hwFR{{\lambda_{1}}} \right)
\end{align*}
is of degree~$-p$, and the difference
\begin{align*}
&\newvarWittL{-p}{2} - \newldvarWittL{-p}{2} \\
= \;\; & -(-\apvarB)^{-p+1} \pder{\apvar}
	+\sum_{k=0}^{\infty} \binom{-p+1}{k} \, (-1)^{p+1} \, \apvarB^{k} \, \apvar^{-p-k} \,
	    \Big( \apvar \pder{\apvar} + (1-p-k) \, \hwFR{\lambda_{0}} \Big) 
\end{align*}
is a sum of terms of degrees strictly greater than~$-p$.
Therefore, we can decompose the action of~$\newvarBSAoper{2}$ 
on $\mathbb{C}[\apvar]\,[[\apvarB/\apvar]] \, [\apvarB^{-1}] \; 
\apvar^{\deltasymB'(\cbseqB)} \, \apvarB^{\deltasymB(\cbseqB)}$ so that
\begin{align*}
\newvarBSAoper{2}
\; = \; \newldvarBSAoper{2} +
    \sum_{k=1}^{\infty} \newextraterm^{(2)}_{k},
\end{align*}
where $\deg \newextraterm^{(2)}_{k} = -\lambda_{2}-1+k$ for $k \geq 1$, and
\begin{align*}
\newldvarBSAoper{2}
= \sum_{k=1}^{\lambda_{2}+1}
    \sum_{\substack{p_{1},\dots, p_{k}\geq 1 \\ p_{1}+\cdots+p_{k}=\lambda_{2}+1}}
    \frac{(-4/\kappa)^{1+\lambda_{2}-k}\lambda_{2}!}{\prod_{u=1}^{k-1}(\sum_{i=1}^{u}p_{i})
    (\sum_{i=u+1}^{k}p_{i})} \,
    \newldvarWittL{-p_1}{2} \cdots \newldvarWittL{-p_k}{2}
\end{align*}
is a differential operator of degree $-\lambda_{2}-1$.

We also have
\begin{align*}
\newldvarBSAoper{2} \ \apvarB^{\deltasymB(\cbseqB)+n}
\; = \; P_{\lambda_{2}} \big(\hwFR{\lambda_{1}} , \hwFR{\mu}+n \big)
    \; \apvarB^{\deltasymB(\cbseqB)+n-\lambda_{2}-1}
\end{align*}
for any~$n \in \bZ$, and we recall that
$P_{\lambda_{2}}(\hwFR{\lambda_{1}},\hwFR{\mu}+n)=0$ if and only if~$n=0$.

Let us expand a series $\newhwme \in 
\mathbb{C}[[\apvarB/\apvar]] \; 
\apvar^{\deltasymB'(\cbseqB)} \, \apvarB^{\deltasymB(\cbseqB)}$ so that
\begin{align*}
\newhwme \; = \;
    \sum_{n=0}^{\infty} \newhwme_{n} \; \apvar^{\deltasymB'(\cbseqB) - n} \apvarB^{\deltasymB(\cbseqB) + n} .
\end{align*}
Requiring the differential equation $\newvarBSAoper{2} \, \newhwme = 0$
and considering the terms of different degrees separately,
we obtain the recursion relations that determines all higher
coefficients $\newhwme_{k}$, $k>0$, by
\begin{align*}
\newhwme_{k}
\; = \; & -
  \frac{\apvarB^{-\deltasymB(\cbseqB)-k+\lambda_{2}+1} \apvar^{-\deltasymB'(\cbseqB)+k}}%
  {P_{\lambda_{2}} \big( \hwFR{\lambda_{1}} , \hwFR{\mu}+k\big)}
  \; \sum_{n=0}^{k-1} \newextraterm^{(2)}_{k-n}
      \left(\apvarB^{\deltasymB(\cbseqB)+n} \, \apvar^{\deltasymB'(\cbseqB)-n} \, \newhwme_{n} \right) . 
\end{align*}
Therefore, a solution is uniquely determined by its leading
coefficient~$\newhwme_0$.
\end{proof}

Putting the above ingredients together, we can
prove Proposition~\ref{prop:associativity_initial_term}.
\begin{proof}[Proof of Proposition~\ref{prop:associativity_initial_term}]
By linearity on~$\sF$ and decomposition of~$u_{\cbfullseq}$, we found
\begin{align*}
\sF[u_{\cbfullseq}](0,x_1,x_2) \; = \;
\sum_{\cbseqB} \sixj{\lambda_\infty}{\lambda_2}{\lambda_1}{\lambda_0}{\cbseq}{\cbseqB}
    \; \sF[\widehat{u}_{\cbfullseqB}](0,x_1,x_2) .
\end{align*}
According to Theorem~\ref{thm: quantum group functions and compositions of intertwining ops},
the left hand side is given by the series 
\begin{align*}
\Braket{\hwvsym_{\lambda_{\infty}}', \; \intertwnorm{\cbseq}{\lambda_{2}}{\lambda_{\infty}}(\hwvsym_{\lambda_{2}}, x_{2})
    \,  \intertwnorm{\lambda_{0}}{\lambda_{1}}{\cbseq}(\hwvsym_{\lambda_{1}}, x_{1}) \, \hwvsym_{\lambda_{0}}} .
\end{align*}
On the right hand side, change variables to $x=x_1, \, y=x_2-x_1$,
and recall from Lemma~\ref{lem: associativity PDE for the quantum group function}
that each term
$\sF[\widehat{u}_{\cbfullseqB}](0,x,x+y)$
satisfies the differential equation
$\newvarBSAoper{2} \, \sF[\widehat{u}_{\cbfullseqB}](0,x,x+y) = 0$
and by Lemma~\ref{lem: Frobenius series of QG functions}
has a Frobenius series of the form~\eqref{eq: another Frobenius series}.
By homogeneity, property~(COV) in Theorem~\ref{thm:properties_of_correlations},
it is easy to see that the coefficients are of the
form~$\widehat{C}_n (x) = c_n \, x^{\deltasymB'(\cbseqB)-n}$.
The power series part of the Frobenius series in variable~$y$
is locally uniformly
$R$-controlled in the domain defined by~$x>R$, so
by arguments similar to the proof of
Theorem~\ref{thm: quantum group functions and compositions of intertwining ops}
we get that the differential
operator~$\newvarBSAoper{2}$ acts naturally coefficientwise on the series.
The uniqueness up to multiplicative constant of series solutions stated in
Proposition~\ref{prop: uniqueness of solutions for associativity}
then shows that the series
expansion is
\begin{align*}
\sF[\widehat{u}_{\cbfullseqB}](0,x,x+y)
\; = \; \Braket{\hwvsym_{\lambda_{\infty}}^{\prime},
  \intertwnorm{\lambda_{0}}{\cbseqB}{\lambda_{\infty}} \big(
    \intertwnorm{\lambda_{1}}{\lambda_{2}}{\cbseqB}(\hwvsym_{\lambda_{2}}, y) \,
    \hwvsym_{\lambda_{1}}, x \big) \, \hwvsym_{\lambda_{0}}} ,
\end{align*}
since the leading coefficients on both sides are
$\widehat{C}_0 (x) = \betacoef{\lambda_{2}}{\lambda_{1}}{\;\;\cbseqB} \,
    \betacoef{\cbseqB}{\lambda_{0}}{\,\lambda_\infty} \; x^{\deltasymB'(\cbseqB)} $.
\end{proof}

\subsection{Reduction to the initial terms}
\label{sub: associativity general case}

Here, we present a proof of Theorem \ref{thm:associativity}.
The proof is by induction on the total PBW length,
broadly similarly to Proposition~\ref{prop: uniqueness of intertwiners up to constant}.
In fact, since the proof splits to many similar cases, we only provide the details
about one. Proposition~\ref{prop:associativity_initial_term} from above
will serve as the base case of the induction.

\begin{proof}[Proof of Theorem \ref{thm:associativity}]
Recall that all the modules $\voaFR{\lambda_{i}}$, $i=0,1,2$ and $\voaFR{\lambda_{\infty}}^{\prime}$ admit the PBW filtration;
\begin{align*}
\PBWfil{0} \voaFR{\lambda_{i}} \; \subset \; \cdots 
  & \; \subset \; \PBWfil{p_{i}}\voaFR{\lambda_{i}}
	\; \subset \; \PBWfil{p_{i}+1}\voaFR{\lambda_{i}}
	\; \subset \; \cdots, \qquad \text{ for } i=0,1,2, \\
\PBWfil{0} \voaFR{\lambda_{\infty}}^{\prime} \; \subset \; \cdots
  & \; \subset \; \PBWfil{p_{\infty}}\voaFR{\lambda_{\infty}}^{\prime}
    \; \subset \; \PBWfil{p_{\infty}+1}\voaFR{\lambda_{\infty}}^{\prime}
    \; \subset \; \cdots.
\end{align*}

Fix $p\in \N$ and assume that the claim in Theorem~\ref{thm:associativity}
is true for any $\Mvec_{i}\in \PBWfil{p_{i}}\voaFR{\lambda_{i}}$, $i=0,1,2$
and $\Mvec_{\infty}^{\prime}\in \PBWfil{p_{\infty}}\voaFR{\lambda_{\infty}}^{\prime}$
with $p_{0}+p_{1}+p_{2}+p_{\infty}\leq p$.
Note also that the base
case $p_0 = p_1 = p_2 = p_\infty = 0$ is covered
by Proposition~\ref{prop:associativity_initial_term}.

Let us first consider increasing~$p_0$ by one, by
applying $L_{-n}$, $n>0$, on $\Mvec_0 \in \PBWfil{p_{0}}\voaFR{\lambda_{0}}$.
The Jacobi identity, as formulated in 
Corollary~\ref{cor:formula_intertw}, yields the following
\begin{align*}
\phantom{\bigg|} & \Braket{\Mvec_{\infty}^{\prime},
    \; \intertwnorm{\cbseq}{\lambda_{2}}{\lambda_{\infty}} (\Mvec_{2},\apvar_{2}) \,
    \intertwnorm{\lambda_{0}}{\lambda_{1}}{\cbseq}(\Mvec_{1},\apvar_{1}) \, L_{-n}\Mvec_{0}} \\
= \; & \Braket{L_{n} \Mvec_{\infty}^{\prime}, \;
	\intertwnorm{\cbseq}{\lambda_{2}}{\lambda_{\infty}}(\Mvec_{2},\apvar_{2}) \,
	\intertwnorm{\lambda_{0}}{\lambda_{1}}{\cbseq}(\Mvec_{1},\apvar_{1}) \, \Mvec_{0}} \\
& - \apvar_{1}^{-n+1} \pder{\apvar_{1}} \Braket{\Mvec_{\infty}^{\prime}, \;
    \intertwnorm{\cbseq}{\lambda_{2}}{\lambda_{\infty}}(\Mvec_{2},\apvar_{2}) \,
    \intertwnorm{\lambda_{0}}{\lambda_{1}}{\cbseq}(\Mvec_{1},\apvar_{1}) \; \Mvec_{0} } \\
& - \sum_{k=1}^{\infty}\binom{-n+1}{k}\apvar_{1}^{-n-k+1}
    \Braket{\Mvec_{\infty}^{\prime}, \; 
    \intertwnorm{\cbseq}{\lambda_{2}}{\lambda_{\infty}}(\Mvec_{2},\apvar_{2}) \, 
    \intertwnorm{\lambda_{0}}{\lambda_{1}}{\cbseq}(L_{k-1}\Mvec_{1},\apvar_{1}) \, \Mvec_{0}} \\
& - \apvar_{2}^{-n+1} \pder{\apvar_{2}} \Braket{\Mvec_{\infty}^{\prime}, \;
    \intertwnorm{\cbseq}{\lambda_{2}}{\lambda_{\infty}}(\Mvec_{2},\apvar_{2}) \,
    \intertwnorm{\lambda_{0}}{\lambda_{1}}{\cbseq}(\Mvec_{1},\apvar_{1}) \, \Mvec_{0} } \\
& - \sum_{k=1}^{\infty}\binom{-n+1}{k}\apvar_{2}^{-n-k+1}
    \Braket{\Mvec_{\infty}^{\prime} , \;
    \intertwnorm{\cbseq}{\lambda_{2}}{\lambda_{\infty}}(L_{k-1}\Mvec_{2},\apvar_{2}) \,
    \intertwnorm{\lambda_{0}}{\lambda_{1}}{\cbseq}(\Mvec_{1},\apvar_{1}) \, \Mvec_{0} }.
\end{align*}
Each of the terms of the right hand side have a 
total PBW word length~$\le p$,
so the induction hypothesis can be applied to each of them.
By the induction hypothesis, the right hand side power series represents the same
analytic function as the following series
\begin{align}\nonumber
\sum_{\cbseqB}
\sixj{\lambda_{\infty}}{\lambda_2}{\lambda_1}{\lambda_0}{\cbseq}{\cbseqB} \bigg( \,
& \Braket{ L_{n} \Mvec_{\infty}^{\prime}, \;
    \intertwnorm{\lambda_{0}}{\cbseqB}{\lambda_{\infty}} \big(
        \intertwnorm{\lambda_{1}}{\lambda_{2}}{\cbseqB}(\Mvec_{2}, \apvarB)
        \Mvec_{1}, \apvar \big) \, \Mvec_{0}} \\ \nonumber
& -\apvar^{1-n}\left( \pder{\apvar} - \pder{\apvarB} \right)
    \Braket{\Mvec_{\infty}^{\prime}, \;
    \intertwnorm{\lambda_{0}}{\cbseqB}{\lambda_{\infty}} \big(
        \intertwnorm{\lambda_{1}}{\lambda_{2}}{\cbseqB}(\Mvec_{2},\apvarB)
        \Mvec_{1}, \apvar) \, \Mvec_{0}} \\  \nonumber
& -\sum_{k=1}^{\infty}\binom{1-n}{k} \apvar^{1-n-k}
    \Braket{\Mvec_{\infty}^{\prime}, \;
    \intertwnorm{\lambda_{0}}{\cbseqB}{\lambda_{\infty}} \big(
        \intertwnorm{\lambda_{1}}{\lambda_{2}}{\cbseqB}(\Mvec_{2}, \apvarB) L_{k-1} 
        \Mvec_{1}, \apvar \big) \, \Mvec_{0}} \\  \nonumber
& -(\apvar+\apvarB)^{1-n} \, \pder{\apvarB}
    \Braket{\Mvec_{\infty}^{\prime}, \;
    \intertwnorm{\lambda_{0}}{\cbseqB}{\lambda_{\infty}} \big(
        \intertwnorm{\lambda_{1}}{\lambda_{2}}{\cbseqB}(\Mvec_{2}, \apvarB)
        \Mvec_{1}, \apvar \big) \, \Mvec_{0}} \\ 
\label{eq: one complicated PBW length reduction}
& -\sum_{k=1}^{\infty}\binom{1-n}{k} (\apvar+\apvarB)^{1-n-k}
    \Braket{\Mvec_{\infty}^{\prime}, \;
    \intertwnorm{\lambda_{0}}{\cbseqB}{\lambda_{\infty}} \big(
        \intertwnorm{\lambda_{1}}{\lambda_{2}}{\cbseqB}(L_{k-1}\Mvec_{2}, \apvarB)
        \Mvec_{1}, \apvar \big) \, \Mvec_{0}}
\bigg)
\end{align}
when the two series are evaluated at $\apvar_1 = x_1$, $\apvar_2 = x_2$, and
$\apvar = x_1$, $\apvarB = x_2-x_1$, respectively.
It remains to check that the expression inside the parentheses 
in~\eqref{eq: one complicated PBW length reduction} coincides with
\begin{align*}
\Braket{\Mvec_{\infty}^{\prime}, \;
    \intertwnorm{\lambda_{0}}{\cbseqB}{\lambda_{\infty}} \big(
        \intertwnorm{\lambda_{1}}{\lambda_{2}}{\cbseqB}(\Mvec_{2}, \apvarB)
        \Mvec_{1}, \apvar \big) \, L_{-n} \Mvec_{0}} .
\end{align*}
For this, we first commute $L_{-n}$ to the left with the
formula of Corollary~\ref{cor:formula_intertw},
\begin{align*}
\phantom{\bigg|} & \Braket{\Mvec_{\infty}^{\prime}, \;
    \intertwnorm{\lambda_{0}}{\cbseqB}{\lambda_{\infty}} \big(
        \intertwnorm{\lambda_{1}}{\lambda_{2}}{\cbseqB}(\Mvec_{2}, \apvarB)
        \Mvec_{1}, \apvar \big) \, L_{-n} \Mvec_{0}} \\
= \;\; & \Braket{ L_n \, \Mvec_{\infty}^{\prime}, \;
    \intertwnorm{\lambda_{0}}{\cbseqB}{\lambda_{\infty}} \big(
        \intertwnorm{\lambda_{1}}{\lambda_{2}}{\cbseqB}(\Mvec_{2}, \apvarB)
        \Mvec_{1}, \apvar \big) \, \Mvec_{0}} \\
& - \sum_{\ell=0}^\infty \binom{1-n}{\ell} \, \apvar^{1-n-\ell} \;
    \Braket{ \Mvec_{\infty}^{\prime}, \;
    \intertwnorm{\lambda_{0}}{\cbseqB}{\lambda_{\infty}} \big(
        L_{\ell-1} \, \intertwnorm{\lambda_{1}}{\lambda_{2}}{\cbseqB}(\Mvec_{2}, \apvarB)
        \Mvec_{1}, \apvar \big) \, \Mvec_{0}}
\end{align*}
The first term here is indeed present
in~\eqref{eq: one complicated PBW length reduction},
so we focus on the second term. Also the $\ell=0$
contribution in the second term is
\begin{align*}
& -\apvar^{1-n} \,
    \Braket{ \Mvec_{\infty}^{\prime}, \;
    \intertwnorm{\lambda_{0}}{\cbseqB}{\lambda_{\infty}} \big(
        L_{-1} \, \intertwnorm{\lambda_{1}}{\lambda_{2}}{\cbseqB}(\Mvec_{2}, \apvarB)
        \Mvec_{1}, \apvar \big) \, \Mvec_{0}} \\
\; = \;\; & -\apvar^{1-n} \, \pder{\apvar}\,
    \Braket{ \Mvec_{\infty}^{\prime}, \;
    \intertwnorm{\lambda_{0}}{\cbseqB}{\lambda_{\infty}} \big(
        \intertwnorm{\lambda_{1}}{\lambda_{2}}{\cbseqB}(\Mvec_{2}, \apvarB)
        \Mvec_{1}, \apvar \big) \, \Mvec_{0}} ,
\end{align*}
which also appears in~\eqref{eq: one complicated PBW length reduction}.
We rewrite the remaining terms by commuting $L_{\ell-1}$ to the right inside 
the intertwiner, leading to a contribution
\begin{align*}
& - \sum_{\ell=1}^\infty \binom{1-n}{\ell} \, \apvar^{1-n-\ell} \;
    \Braket{ \Mvec_{\infty}^{\prime}, \;
    \intertwnorm{\lambda_{0}}{\cbseqB}{\lambda_{\infty}} \big(
        L_{\ell-1} \, \intertwnorm{\lambda_{1}}{\lambda_{2}}{\cbseqB}(\Mvec_{2}, \apvarB)
        \Mvec_{1}, \apvar \big) \, \Mvec_{0}} \\
\; = \;\; & - \sum_{\ell=1}^\infty \binom{1-n}{\ell} \, \apvar^{1-n-\ell}
    \Bigg( \Braket{ \Mvec_{\infty}^{\prime}, \;
        \intertwnorm{\lambda_{0}}{\cbseqB}{\lambda_{\infty}} \big(
        \intertwnorm{\lambda_{1}}{\lambda_{2}}{\cbseqB}(\Mvec_{2}, \apvarB)
        L_{\ell-1} \, \Mvec_{1}, \apvar \big) \, \Mvec_{0}} \\
& \qquad\qquad\qquad + \sum_{k=0}^\infty \binom{\ell}{k} \, \apvarB^{\ell-k} \;
    \Braket{ \Mvec_{\infty}^{\prime}, \;
    \intertwnorm{\lambda_{0}}{\cbseqB}{\lambda_{\infty}} \big(
        \intertwnorm{\lambda_{1}}{\lambda_{2}}{\cbseqB}( L_{k-1} \Mvec_{2}, \apvarB)
        \Mvec_{1}, \apvar \big) \, \Mvec_{0}} \Bigg)
\end{align*}
Again the first term appears
in~\eqref{eq: one complicated PBW length reduction}, so it remains
to treat the last term. The $k=0$ contribution of it is just
\begin{align*}
& - \sum_{\ell=1}^\infty \binom{1-n}{\ell} \, \apvar^{1-n-\ell}
    \, \apvarB^{\ell} \;
    \Braket{ \Mvec_{\infty}^{\prime}, \;
    \intertwnorm{\lambda_{0}}{\cbseqB}{\lambda_{\infty}} \big(
        \intertwnorm{\lambda_{1}}{\lambda_{2}}{\cbseqB}( L_{-1} \Mvec_{2}, \apvarB)
        \Mvec_{1}, \apvar \big) \, \Mvec_{0}} \\
\; = \;\; & - (\apvar + \apvarB)^{1-n} \, \pder{\apvarB} \,
    \Braket{ \Mvec_{\infty}^{\prime}, \;
    \intertwnorm{\lambda_{0}}{\cbseqB}{\lambda_{\infty}} \big(
        \intertwnorm{\lambda_{1}}{\lambda_{2}}{\cbseqB}( \Mvec_{2}, \apvarB)
        \Mvec_{1}, \apvar \big) \, \Mvec_{0}} \\
& + \apvar^{1-n} \, \pder{\apvarB} \,
    \Braket{ \Mvec_{\infty}^{\prime}, \;
    \intertwnorm{\lambda_{0}}{\cbseqB}{\lambda_{\infty}} \big(
        \intertwnorm{\lambda_{1}}{\lambda_{2}}{\cbseqB}( \Mvec_{2}, \apvarB)
        \Mvec_{1}, \apvar \big) \, \Mvec_{0}} .
\end{align*}
Both of these again appear
in~\eqref{eq: one complicated PBW length reduction}.
In the remaining terms, we interchange the order of summations
and change to a new summation variable $m=\ell-k$ to express them as
\begin{align*}
& - \sum_{\ell=1}^\infty \binom{1-n}{\ell} \, \apvar^{1-n-\ell}
    \sum_{k=1}^\infty \binom{\ell}{k} \, \apvarB^{\ell-k} \;
    \Braket{ \Mvec_{\infty}^{\prime}, \;
    \intertwnorm{\lambda_{0}}{\cbseqB}{\lambda_{\infty}} \big(
        \intertwnorm{\lambda_{1}}{\lambda_{2}}{\cbseqB}( L_{k-1} \Mvec_{2}, \apvarB)
        \Mvec_{1}, \apvar \big) \, \Mvec_{0}} \\
= \;\; & - \sum_{k=1}^\infty
    \sum_{\ell=k}^\infty 
    \binom{\ell}{k} \, \binom{1-n}{\ell} \, \apvar^{1-n-\ell} \, \apvarB^{\ell-k} 
    \Braket{ \Mvec_{\infty}^{\prime}, \;
    \intertwnorm{\lambda_{0}}{\cbseqB}{\lambda_{\infty}} \big(
        \intertwnorm{\lambda_{1}}{\lambda_{2}}{\cbseqB}( L_{k-1} \Mvec_{2}, \apvarB)
        \Mvec_{1}, \apvar \big) \, \Mvec_{0}} \\
= \;\; & - \sum_{k=1}^\infty
    \sum_{m=0}^\infty 
    \frac{1}{k! \, m!} \prod_{i=0}^{k+m-1} (1-n-i) \;
    \, \apvar^{1-n-k-m} \, \apvarB^{m} 
    \Braket{ \Mvec_{\infty}^{\prime}, \;
    \intertwnorm{\lambda_{0}}{\cbseqB}{\lambda_{\infty}} \big(
        \intertwnorm{\lambda_{1}}{\lambda_{2}}{\cbseqB}( L_{k-1} \Mvec_{2}, \apvarB)
        \Mvec_{1}, \apvar \big) \, \Mvec_{0}} .
\end{align*}
This equals the remaining term 
in~\eqref{eq: one complicated PBW length reduction}.

There are in principle three more cases to consider:
the application of~$L_{-n}$ on $\Mvec_1 \in \PBWfil{p_{1}}\voaFR{\lambda_{1}}$,
on $\Mvec_2 \in \PBWfil{p_{2}}\voaFR{\lambda_{2}}$, and
on $\Mvec_\infty' \in \PBWfil{p_{\infty}}\voaFR{\lambda_{\infty}}'$.
The calculations are similar to the case above, so we omit the details here.
\end{proof}

\appendix

\section{Proofs of analyticity and Frobenius series}%
\label{app: series expansions}
This appendix is devoted to indicating how the
results of \cite{KP-conformally_covariant_bdry_correlations}
can be strengthened so as to yield the analyticity
and series expansion properties needed in the present 
article, and for some elementary power series estimates
employed in the proofs of the main results in
Sections~\ref{sec: composition} and~\ref{sec: associativity}.

\subsection{Elementary power series estimates}

In Sections~\ref{ssec: quantum group functions as power series}
and~\ref{sub: associativity for the highest weight matrix element},
we needed that three types of operations preserved a suitable
space of parametrized power series, and that these
operations could be performed essentially ``coefficientwise''.
Specifically, we needed
\begin{itemize}
\item differentiation of the power series itself;
\item multiplication of the power series by another power series;
\item differentiation of the power series with respect to parameters.
\end{itemize}

The setup is that we consider power series in a variable~$z$, of the form
\begin{align}\label{eq: parametrized power series}
\sum_{k=0}^\infty c_k(y) \, z^k ,
\end{align}
where the coefficients $c_k$ depend smoothly on parameters~$y \in \Omega$,
where $\Omega \subset \bR^m$ is an open set. The main assumption is that
for some~$R>0$, the parametrized power series is locally uniformly $R$-controlled
in the sense of Definition~\ref{def: controlled power series}.
For concretely working with the power series, it is useful to note 
that the property of being  locally uniformly $R$-controlled
is equivalent to the following: for 
every compact $K \subset \Omega$, every multi-index~$\alpha$,
and every~$0 < R_0 < R$, there exists a $M<\infty$ such that 
\begin{align}\label{eq: elementary estimate for controlled coefficients}
\big| \partial^\alpha c_k(y) \big| \; \leq \; M \, R_0^{-k}
\qquad \text{for all $y \in K$ and $k \in \Znn$.}
\end{align}
This in particular ensures that for every $y \in \Omega$, the radius
of convergence of~\eqref{eq: parametrized power series} is at least~$R$,
and for any $R_0<R$ and any compact subset~$K \subset \Omega$,
the convergence is uniform over $|z| \le R_0$ and $y \in K$.

Crucial for us is that the locally uniformly $R$-controlled
power series are stable under 
differentiation of the power series, multiplication by another
locally uniformly $R$-controlled power series,
and differentiation with respect to the parameters~$y \in \Omega$.
The first two are in essence just familiar operations
on ordinary power series~--- only the last one involves
dependence on the parameters~$y$. The proofs are all elementary.

\begin{lem}
\label{lem: operations on R controlled series}
Let $\Omega \subset \bR^m$ be an open set, and let~$R>0$.
\begin{itemize}
\item[(a)]
If $(c_k)_{k \in \Znn}$ are locally uniformly $R$-controlled,
then also ${\big( (k+1) c_{k+1} \big)_{k \in \Znn}}$ are
locally uniformly $R$-controlled and for any $z \in \bC$
with $|z|<R$ and any $y \in \Omega$ we have
\begin{align*}
\der{z} \sum_{k=0}^\infty c_k(y) \, z^k
= \sum_{k=0}^\infty (k+1) \, c_{k+1}(y) \, z^{k} .
\end{align*}
\item[(b)]
If $(c_k)_{k \in \Znn}$ and $(d_k)_{k \in \Znn}$ are locally
uniformly $R$-controlled,
then also the convoluted coefficients
${\big( \sum_{j=0}^k c_{k-j} d_j \big)_{k \in \Znn}}$ are
locally uniformly $R$-controlled and for any $z \in \bC$
with~$|z|<R$ and any $y \in \Omega$ we have
\begin{align*}
\Big( \sum_{k=0}^\infty d_k(y) \, z^k \Big)
\Big( \sum_{k=0}^\infty c_k(y) \, z^k \Big)
= \sum_{k=0}^\infty
   \bigg( \sum_{j=0}^k c_{k-j}(y) \, d_j(y) \bigg) \, z^{k} .
\end{align*}
\item[(c)]
If $(c_k)_{k \in \Znn}$ are locally uniformly $R$-controlled
and $j \in \set{1,\ldots,m}$, 
then also the coefficients
${\big( \partial_j c_{k} \big)_{k \in \Znn}}$ are
locally uniformly $R$-controlled and for any $z \in \bC$
with~$|z|<R$ and any $y \in \Omega$ we have
\begin{align*}
\pder{y_j} \sum_{k=0}^\infty c_k(y) \, z^k
= \sum_{k=0}^\infty  (\partial_j c_{k}) (y) \, z^{k} .
\end{align*}
\end{itemize}
\end{lem}
\begin{proof}
The statements (a) and (b) at a fixed parameter~$y$
are textbook power series results based on the 
estimates~\eqref{eq: elementary estimate for controlled coefficients},
and the explicit calculations yield locally
uniformly $R$-controlledness by the same characterization.

For property (c), note first that arbitrary
partial derivatives of the coefficients satisfy the same 
estimate~\eqref{eq: elementary estimate for controlled coefficients},
so it is clear that any partial derivatives of the coefficients
remain locally uniformly $R$-controlled.
It suffices to check that the the partial derivative of the
series is the series with partial derivative coefficients.
Fix $y \in \Omega$ and choose a small $r>0$ so that the closed ball
$\overline{B}_r(y) \subset \Omega$.
First fix $R'<R_0<R$ and 
consider $|z| \leq R'$. Since $\overline{B}_r(y) \subset \Omega$ is compact,
we may choose~$M<\infty$
such that for all $k \in \Znn$ and $y \in \overline{B}_r(u)$ we have
both $|c_k(y)| \leq M \, R_0^{-k}$ and
$\big| \partial_j c_k(y) \big| \leq M \, R_0^{-k}$.
The partial derivative of the power series
with respect to the $j$:th parameter $y_j$ at $y=(y_1,\ldots,y_m)$
is 
\begin{align*}
\pder{y_j} \sum_{k=0}^\infty c_k(y) \, z^k
= \; & \lim_{\delta \to 0} \; \sum_{k=0}^\infty
        \frac{c_k(y+\delta \unitvec{j}) - c_k(y)}{\delta} \, z^k \\
= \; & \lim_{\delta \to 0} \; \sum_{k=0}^\infty
        \Big( \int_0^1 (\partial_j c_k)(y+ s \delta \unitvec{j}) \, \ud s \Big) \, z^k .
\end{align*}
For $\delta<r$ the integral above is bounded by the same $M \, R_0^{-k}$
(since the integrand is, and the integration is taken over the unit integral),
so for $|z| \leq R'<R_0$ the power series terms are dominated by
$M \, \big( \frac{R'}{R_0} \big)^k$, which are summable.
Moreover, as $\delta \to 0$, the integrands are tending pointwise w.r.t.~$s$ to
the constant~$(\partial_j c_k)(y)$, and they are bounded by the above.
With these observations, we can apply dominated convergence
to interchange the limit $\delta \to 0$ 
with both the series and the interal, and the assertion~(c) follows.
\end{proof}

\subsection{Power series for the functions from the quantum group method}

We now outline the proofs of
Lemmas~\ref{lem: analyticity of QG functions}
and~\ref{lem: Frobenius series of QG functions}.
For convenience, we also recall their statements here.

\begin{lem*}
[Lemma~\ref{lem: analyticity of QG functions}]
Let $F = \sF[u] \colon \chamber_N \to \bC$
be the function associated to 
any~$u\in\HWspace{\lambdafullseq}$,
and let $(x_1, \ldots, x_N) \in \chamber_N$,
and let $j \in \set{1,\ldots,N}$.
Then we have a power series expansion
\begin{align*}
F(x_1, \ldots, x_{j-1}, z_j, x_{j+1}, \ldots, x_N)
\; = \; \sum_{k=0}^\infty
    c_k (x_1, \ldots, x_{j-1}, x_{j+1}, \ldots, x_N) \; \big( z_j - x_j \big)^k
\end{align*}
in the~$j$:th variable.
For fixed $x_j \in \bR$ and $R>0$,
viewing the other variables $(x_i)_{i \neq j}$ as parameters,
on the subset $\Omega \subset \bR^{N-1}$
defined by the conditions $x_1 < \cdots < x_N$
and $\min_{i \neq j} |x_i - x_j| > R$, the power series
is locally uniformly $R$-controlled.
\end{lem*}
\begin{proof}
The construction 
in~\cite{KP-conformally_covariant_bdry_correlations}
expresses~$F = \sF[u]$ as a finite linear
combination of integrals
\begin{align*}
G(x_1, \ldots, x_N)
= \idotsint_\Gamma 
    f(x_1, \ldots, x_N ; w_1 , \ldots , w_\ell) 
    \; \ud w_1 \cdots \ud w_\ell ,
\end{align*}
where the integrands are branches of the multivalued function
\begin{align*}
& f(x_1, \ldots, x_N ; w_1 , \ldots , w_\ell) \\
=\; & \const \prod_{1 \leq i < j \leq N} 
        (x_j - x_i)^{\frac{2}{\kappa} \lambda_i \lambda_j}
    \prod_{\substack{1 \leq i \leq N \\ 1 \leq r \leq \ell}}
        (x_i - w_r)^{-\frac{4}{\kappa} \lambda_i}
    \prod_{1 \leq r < s \leq \ell}
        (w_s - w_r)^{\frac{8}{\kappa}} ,
\end{align*}
and where the integration surfaces~$\Gamma$ are
collections of non-intersecting loops 
(for details 
see~\cite{KP-conformally_covariant_bdry_correlations}).
Most importantly, the integration surfaces~$\Gamma$ are 
compact, and the integrands $(w_1, \ldots, w_\ell) \mapsto 
f(x_1, \ldots, x_N ; w_1 , \ldots , w_\ell)$ are continuous
on them.

Fix $R>0$ and a compact subset~$K \subset
\Omega_R = \set{ (x_i)_{i \neq j} \in \bR^{N-1} \; \big| \;
\min_{i \neq j} |x_i - x_j| > R}$.
When~$(x_i)_{i \neq j} \in K$,
without changing the homotopy class 
of~$\Gamma$, it is possible to arrange so that
none of the~$\ell$ coordinates of~$\Gamma = \Gamma_K$
intersects the closed ball~$\overline{B}_{R}(x_j)$
of radius~$R$ centered at~$x_j$.
For notational convenience, 
let us keep the coordinates $x_i$, $i \neq j$,
fixed and omit them from the notation, and
consider only the dependence
of the integrand~$f$ 
on the $j$:th variable 
(denoted $z_j$; we will perform power series
expansions around~$z_j=x_j$)
and the integration variables $w_1, \ldots, w_\ell$.
From the explicit formula for~$f$,
it is clear that for any~$(w_1, \ldots, w_\ell) \in \Gamma$,
the function~$z_j \mapsto f(z_j ; w_1 , \ldots , w_\ell)$
is analytic in an open set 
containing~$\overline{B}_{R}(x_j)$,
and has a convergent power series 
expansion
\begin{align*}
f(z_j ; w_1 , \ldots , w_\ell)
= \sum_{k = 0}^\infty 
    \phi_k(w_1, \ldots, w_\ell) \, (z_j - x_j)^k ,
\end{align*}
where the Taylor coefficients
obey the Cauchy integral based estimates
\begin{align*}
\big| \phi_k(w_1, \ldots, w_\ell) \big|
\leq 
    \frac{\max_{z_j \in \overline{B}_{R}(x_j)}
            \big| f(z_j ; w_1 , \ldots , w_\ell) \big|}{R^k} .
\end{align*}
In particular we get the uniform estimate 
\begin{align*}
\big| \phi_k(w_1, \ldots, w_\ell) \big|
\leq \frac{C_K}{R^k} ,
\end{align*}
for all~$(w_1, \ldots, w_\ell) \in \Gamma$,
with the finite constant~$C_K$ taken as the
maximum of $\big| f(z_j ; w_1 , \ldots , w_\ell) \big|$
over the compact 
set~$\overline{B}_{R}(x_j) \times K \times \Gamma_K$
(the middle factor is for the implicit parameters~$x_i$, $i \neq j$).
Then for
\begin{align*}
G(z_j) = \idotsint_{\Gamma_K} f(z_j; w_1 , \ldots, w_\ell)
    \; \ud w_1 \cdots \ud w_\ell ,
\end{align*}
we can use a power series expansion with coefficients
\begin{align*}
c_k := \idotsint_{\Gamma_K} \phi_k(w_1 , \ldots, w_\ell)
    \; \ud w_1 \cdots \ud w_\ell ,
\end{align*}
which
satisfy~$|c_k| \leq C_K \, |\Gamma_K| \, R^{-k}$.
The series converges to~$G(z_j)$
whenever ${|z_j-x_j|<R}$, by virtue of the error term estimate
\begin{align*}
& \Big| G(z_j) - \sum_{k=0}^{k_0} c_k \, (z_j-x_j)^k \Big| \\
\leq \; & \idotsint_{\Gamma_K} 
        \Big| f(z_j; w_1 , \ldots, w_\ell) 
            - \sum_{k=0}^{k_0} \phi_k(w_1 , \ldots, w_\ell) \,
                (z_j - x_j)^k \Big|
    \; \ud w_1 \cdots \ud w_\ell \\
\leq \; & C_K \, |\Gamma_K| \;\sum_{k=k_0+1}^\infty 
    \Big(\frac{|z_j-x_j|}{R} \Big)^k
\; \underset{k_0 \to \infty}{\longrightarrow} \; 0 .
\end{align*}
From the above explicit estimate for coefficients~$(c_k)_{k\in \bN}$, we see
that~$z_j \mapsto G(z_j)$ is given by a locally 
uniformly $R$-controlled power series parametrized by~$\Omega_R$.
The same holds for the
finite linear combination~$z_j \mapsto F(z_j)$
of such terms.
\end{proof}

The Frobenius series statement that we will use
is the following. Variants of this formulation
with obvious modifications to the statement and
proof could be done as well.
\begin{lem*}[Lemma~\ref{lem: Frobenius series of QG functions}]
Let $j \in \set{2,\ldots,N}$. 
Suppose that~$\tau \in \selRule{\lambda_{j-1}}{\lambda_j}$
and that $u\in\HWspace{\lambdafullseq}$
is such that~$u = \QGlongprcan{j-1}{j}{\tau} (u)$.
The function
$F = \sF[u] \colon \chamber_N \to \bC$
associated to~$u$ has a Frobenius series
expansion in the 
variable~$z=x_j-x_{j-1}$
\begin{align*}
F \big( x_1, \ldots 
    , x_{j-1} , (x_{j-1}+z) , x_{j+1} , \ldots, x_N \big)
\; = \; z^{\deltasym} \,
    \sum_{k=0}^\infty 
      c_k (x_1, \ldots, x_{j-1}, x_{j+1}, \ldots, x_N) \; z^{k}
\end{align*}
where
the indicial exponent is
$\deltasym = \hwFR{\tau} - \hwFR{\lambda_{j}} - \hwFR{\lambda_{j-1}}$. 
For fixed $R>0$,
viewing the other variables $(x_i)_{i \neq j}$ as parameters,
on the subset $\Omega \subset \bR^{N-1}$
defined by the conditions $x_1 < \cdots < x_N$
and $\min_{i \neq j, j-1} |x_i - x_{j-1}| > R$, the power series
part of this Frobenius series is locally uniformly $R$-controlled,
and for $0<z<R$ the Frobenius series converges to the function~$F$
on the left hand side.
\end{lem*}
\begin{proof}
The idea of the proof is otherwise similar to
that of Lemma~\ref{lem: analyticity of QG functions},
except that a part of the integration surface needs to
be separated from the rest and that part of the surface
also undergoes scaling proportional to 
variable~$x_j-x_{j-1}$ of the Frobenius series.

Again for notational simplicity, we 
suppress the fixed variables~$x_i$, $i \neq j$,
from the notation.
By translation invariance we can moreover
assume~$x_{j-1}=0$, so that the 
Frobenius series variable is simply~$x_j$
and notation with scalings centered 
at~$x_{j-1}$ simplifies.

With the methods of~\cite[Sections 3.4, 4.2, 
and 4.3]{KP-conformally_covariant_bdry_correlations},
the assumption
$u = \QGlongprcan{j-1}{j}{\tau} (u)$ implies that
the function~$x_j \mapsto F(x_j) = \sF[u](\ldots,x_j,\ldots)$
can be written as a finite
linear combination of integrals
\begin{align*}
H(x_j) = \idotsint_{\Gamma'} \bigg(
    \idotsint_{x_j \Gamma_0} f(x_j;w'_1,\ldots,w'_m,
        w_1, \ldots, w_\ell) \; \ud w'_1 \cdots \ud w'_{m}
    \bigg) \ud w_1 \cdots \ud w_{\ell} ,
\end{align*}
where the surface $\Gamma'$ is compact and can be kept
fixed as~$x_j \downarrow 0$,
and where and $x_j \Gamma_0$ stands for a scaling
by a factor~$x_j>0$ of a fixed surface~$\Gamma_0$
(same for all terms in the linear combination
and for all sufficiently small~$x_j>0$) of
dimension $m = \frac{1}{2} 
    \big(\lambda_{j-1} + \lambda_{j} - \tau \big)$,
and the integrand
$f(x_j; w'_1,\ldots,w'_m, w_1, \ldots, w_\ell)$
is as in 
Lemma~\ref{lem: analyticity of QG functions}.
(now all of
$w'_1,\ldots,w'_m, w_1, \ldots, w_\ell$
are in the same role as $w_1, \ldots, w_\ell$ originally).
With a change of variables to unit scale,
$t_r = w'_r / x_j$ for $r=1,\ldots,m$,
both integration surfaces become constant, and the
dependence on~$x_j$ is entirely in the integrand:
we find
\begin{align*}
H(x_j) = \idotsint_{\Gamma'} \bigg( 
    \idotsint_{\Gamma_0} 
        x_j^{m} \, f(x_j ; x_j t_1,\ldots , x_j t_m,
        w_1, \ldots, w_\ell) \; \ud t_1 \cdots \ud t_{m}
    \bigg) \ud w_1 \cdots \ud w_{\ell} .
\end{align*}
We will compare
$f(x_j ; x_j t_1,\ldots , x_j t_m, w_1, \ldots, w_\ell)$ 
with (we omit all of the fixed arguments
$x_i$, $i \neq j$, for brevity)
\begin{align*}
& \hat{f}( \ ; 
    w_1 , \ldots , w_\ell) \\
=\; & \const \prod_{\substack{1 \leq i < i' \leq N :\\ i, i' \neq j}}
        (x_{i'} - x_i)^{\frac{2}{\kappa} \hat{\lambda}_i \hat{\lambda}_{i'}}
    \prod_{\substack{1 \leq i \leq N : \; i \neq j \\
            1 \leq r \leq \ell}}
        (x_i - w_r)^{-\frac{4}{\kappa} \hat{\lambda}_i}
    \prod_{1 \leq r < s \leq \ell}
        (w_s - w_r)^{\frac{8}{\kappa}} ,
\end{align*}
where 
\begin{align*}
\hat{\lambda}_i = \begin{cases}
                \tau & \text{ if $i = j-1$.} \\
                \lambda_i & \text{ if $i \neq j-1, j$,} 
                \end{cases}
\end{align*}
because 
$\hat{f}( \ ; w_1 , \ldots , w_\ell)$ is the integrand
whose integral on~$\Gamma'$ yields the asserted
leading asymptotic coefficient.

The ratio
\[ \frac{f(x_j ; x_j t_1,\ldots , x_j t_m, w_1, \ldots, w_\ell)}%
    {\hat{f}( \ ; w_1 , \ldots , w_\ell)} \]
contains factors (we set~$x_{j-1}=0$)
\begin{align*}
(x_j - 0)^{\frac{2}{\kappa}\lambda_{j-1}\lambda_{j}}
\; = \;\; & x_j^{\frac{2}{\kappa}\lambda_{j-1}\lambda_{j}} \\
(x_j t_a - 0)^{-\frac{4}{\kappa}\lambda_{j-1}}
\; = \;\; & x_j^{-\frac{4}{\kappa}\lambda_{j-1}} \;
    (t_a - 0)^{-\frac{4}{\kappa}\lambda_{j-1}}
    & & \text{for } 1 \leq a \leq m \\
(x_j t_a - x_j)^{-\frac{4}{\kappa}\lambda_{j}}
\; = \;\; & x_j^{-\frac{4}{\kappa}\lambda_{j}} \;
    (t_a - 1)^{-\frac{4}{\kappa}\lambda_{j}}
    & & \text{for } 1 \leq a \leq m \\
(x_j t_b - x_j t_a)^{\frac{8}{\kappa}}
\; = \;\; & x_j^{\frac{8}{\kappa}} \;
    (t_b - t_a)^{\frac{8}{\kappa}}
    & & \text{for } 1 \leq a < b \leq m ,
\end{align*}
out of which we extract in particular the powers 
of the scaling factor~$x_j$, which combined with the
factor~$x_j^m$ from the change of variables produce
the correct overall scaling by
\begin{align*}
x_j^{m + \frac{2}{\kappa}\lambda_{j-1}\lambda_{j}
    - \frac{4}{\kappa} (\lambda_{j-1} + \lambda_{j}) m 
    + \frac{8}{\kappa} m(m-1)/2}
= x_j^{h(\tau) - h(\lambda_{j-1}) - h(\lambda_{j})} 
= x_j^{\Delta} .
\end{align*}
The~$t_a$-dependent factors are yet
to be integrated over~$\Gamma_0$.
The ratio can naturally also be rearranged
(by appropriately splitting the factor
$(w_r-0)^{-\frac{4}{\kappa} \hat{\lambda}_{j-1}}$
in the denominator) to contain factors
\begin{align*}
\left( \frac{x_i - x_j}{x_i-0} \right)^{\frac{2}{\kappa} \lambda_i 
\lambda_j}
\; = \;\; & \left( 1- \frac{x_j}{x_i} \right)^{\frac{2}{\kappa} \lambda_i 
\lambda_j}
= \; 1 + \sum_{k=1}^\infty c_k(x_i) \; x_j^k
    & & \text{for } i \notin \set{j-1,j} \\
\left( \frac{x_i - x_j t_a}{x_i-0} \right)^{-\frac{4}{\kappa} \lambda_i }
\; = \;\; & \left( 1- \frac{x_j t_a}{x_i} \right)^{-\frac{4}{\kappa} \lambda_i }
= \; 1 + \sum_{k=1}^\infty c_k(x_i) \; t_a^k \; x_j^k
    & & \text{for } i \notin \set{j-1,j} \text{, }
    1 \leq a \leq m \\
\left( \frac{w_r - x_j}{w_r-0} \right)^{-\frac{4}{\kappa} \lambda_j}
\; = \;\; & \left( 1- \frac{x_j}{w_r} \right)^{-\frac{4}{\kappa} \lambda_j}
= \; 1 + \sum_{k=1}^\infty c_k(w_r) \; x_j^k
    & & \text{for } 1 \leq r \leq \ell \\
\left( \frac{w_r - x_j t_a}{w_r-0} \right)^{\frac{8}{\kappa}}
\; = \;\; & \left( 1- \frac{x_j t_a}{w_r} \right)^{\frac{8}{\kappa}}
= \; 1 + \sum_{k=1}^\infty c_k(w_r) \; t_a^k \; x_j^k
    & & \text{for } 1 \leq r \leq \ell \text{, }
    1 \leq a \leq m \\
\end{align*}
which can be expanded
as power series in~$x_j$ with unit constant coefficient,
all having a radius of convergence at least a fixed positive
multiple of 
the distance~$R'$ of $x_{j-1}$ to the contours in~$\Gamma'$,
when $(t_1 , \ldots, t_m)$ lies on 
the compact set~$\Gamma_0$.
All remaining factors in the ratio cancel directly.
    
We again conclude that after extracting the overall
scaling factor~$x_j^{\Delta}$, the integrand
in~$H(x_j)$ admits a power series expansion in~$x_j$,
with the $k$:th coefficient bounded 
by~$p(k) (C R')^{-k}$, where $p(k)$ is a polynomial 
and $C>0$ is a fixed constant. By arguments similar
to the previous case, a term by term integration
of this  yields a power series
for~$x_j^{-\Delta} \, H(x_j)$.
The proof of the existence of a convergent Frobenius
series expansion with some positive radius of convergence
is complete once one notices
that the integral of the constant coefficient
of the power series factorizes
to integrals over~$\Gamma_0$ and $\Gamma'$,
the former yielding the beta-function coefficient~$B$
and the latter yielding the function with one fewer
variable and 
labels~$(\hat{\lambda}_i)_{1 \leq i \leq N : \, i \neq j}$.

The remaining part of the assertion is 
the uniform~$R$-controlledness on compact subsets
of the domain determined by $\min_{i \neq j, j-1} |x_i - x_{j-1}| > R$.
For any such compact subset, from the start we can choose
$\Gamma'$ so that for some~$\eps>0$, on $\Gamma'$
we have $|w_r| \geq (1 + \eps) R$ for all~$r$, i.e., $R' \ge (1+\eps)R$.
Moreover, $\Gamma_0$ can be chosen so that on~$\Gamma_0$
we have $|t_a| \le 1 + \eps/2$ for all~$a$. With such choices,
the constant above is $C \ge \frac{1}{1+\eps/2}$, and the estimates
indeed yield uniform $R$-controlledness.
\end{proof}

\section{Construction of the intertwining operators for Verma modules}%
\label{app: construction for Verma modules}
Our goal is to show the following theorem (Theorem~\ref{thm:intertw_Vermas}):
\begin{thm}
For $\hmid,\hin ,\hout \in\C$,
\begin{equation*}
	\dim \spintertw {\voaFusion{\VermaCH{c}{\hmid}}{\VermaCH{c}{\hin}}{\VermaCH{c}{\hout}^{\prime}}}=1.
\end{equation*}
\end{thm}

An intertwining operator in this case is unique up to multiplicative constants if exists (Proposition~\ref{prop: uniqueness of intertwiners up to constant}).
Hence, we only have to prove the existence.
We apply the construction in~\cite{Li99} to the generic Virasoro VOA to obtain an intertwining operator among Verma modules; we include this so as
to be self-contained, and also because we believe that the concrete
case of the Virasoro VOA is instructive.
The procedure is divided into two parts: in the former part, we will construct a linear map
\begin{equation}
\label{eq:intertw_op_meanttobe}
	\intertw\colon \VermaCH{c}{\hmid}\to \Hom (\VermaCH{c}{\hin},\VermaCH{c}{\hout}^{\prime})\{\apvar\},
\end{equation}
and in the latter part, we will show that this linear map $\intertw$ is indeed an intertwining operator of the desired type.

\subsection{Construction of a linear map}
For convenience, we write $\modMin:=\VermaCH{c}{\hin}$, $\modMmid:=\VermaCH{c}{\hmid}$ and $\modMout:=\VermaCH{c}{\hout}$.
For each $k\in \Znn$, we also write $\modMin(k):=(\modMin)_{(\hin+k)}$, $\modMmid(k):=(\modMmid)_{(\hmid+k)}$ and $\modMout(k):=(\modMout)_{(\hout+k)}$
for the eigenspaces of~$L_{0}$.

Before constructing the linear map~\eqref{eq:intertw_op_meanttobe},
it is instructive to observe anticipated properties.
First of all, defining a linear map~\eqref{eq:intertw_op_meanttobe}
is equivalent to defining its matrix elements
\begin{equation*}
	\modMout \times (\modMmid\otimes \modMin) \to \bC\{\apvar\};\quad (\Mvecout, \Mvecmid\otimes\Mvecin)\mapsto \braket{\Mvecout,\intertw (\Mvecmid,\apvar)\Mvecin}.
\end{equation*}

Second, if $\intertw$ gives an intertwining operators 
of the desired type, it is expanded as
\begin{equation*}
	\intertw (\Mvecmid,\apvar)=\sum_{m\in\Z}(\Mvecmid)_{(m)}\apvar^{\Delta-m-1},\quad \Delta=\hout-\hmid-\hin .
\end{equation*}
Hence, for our purpose of constructing an intertwining operator,
it is natural to specify the linear map~(\ref{eq:intertw_op_meanttobe})
in terms of the family of infinitely many bilinear functionals
\begin{equation*}
	\modMout \times (\modMmid\otimes\modMin) \to \C;\quad (\Mvecout,\Mvecmid\otimes\Mvecin)\mapsto \braket{\Mvecout,(\Mvecmid)_{(m)}\Mvecin}
\end{equation*}
labeled by $m\in \Z$.
We will actually incorporate these bilinear functionals, by introducing the affinization $\modMmidaff:=\C [t^{\pm1 }]\otimes \modMmid$,
into a single bilinear functional
\begin{equation*}
	\modMout\times (\modMmidaff\otimes \modMin)\to \C ;\quad (\Mvecout, (t^{m}\otimes\Mvecmid)\otimes \Mvecin)\mapsto \braket{\Mvecout,(\Mvecmid)_{(m)}\Mvecin}.
\end{equation*}

Finally, as we have seen in Proposition~\ref{prop: uniqueness of intertwiners up to constant},
an intertwining operator among highest weight modules is determined uniquely by the initial term.
Therefore, the desired bilinear functional should be determined recursively by the value at $(\hwv{c,\hout}, (t^{-1}\otimes \hwv{c,\hmid})\otimes\hwv{c,\hin})$.

\subsubsection*{Step 1}
The first step is to fix an initial term.
Taking a constant $B\in \C\backslash\{0\}$,
we define a bilinear functional\footnote{%
At various stages of the construction, the superscripts to~$\cb$
are meant to indicate in which of the modules
$\modMin$, $\modMmid$, $\modMout$
we restrict attention to only the one-dimensional subspaces
$\modMin(0)$, $\modMmid(0)$, $\modMout(0)$
of highest weight vectors.}
\begin{align*}
& \cb^{(\outsym, \midsym, \insym)}\colon \modMout (0)\times (\modMmid(0)\otimes \modMin (0))\to\C 
\end{align*}
by $\cb^{(\outsym, \midsym, \insym)}
    \left(\hwv{c,\hout}, \hwv{c,\hmid}\otimes\hwv{c,\hin}\right) := B$.

\subsubsection*{Step 2}
We extend the bilinear functional $\cb^{(\outsym, \midsym, \insym)}$
so that an arbitrary vector from $\modMmid$ can be inserted.

\begin{prop}
There exists a unique bilinear functional
\begin{align*}
\cb^{(\outsym, \insym)}\colon \modMout (0)\times (\modMmid\otimes \modMin (0))\to\C 
\end{align*}
which coincides with~$\cb^{(\outsym, \midsym, \insym)}$
on the subspace
\begin{align*}
\modMout (0)\times (\modMmid(0)\otimes \modMin (0)) 
\; \subset \; \modMout (0)\times (\modMmid\otimes \modMin (0)) 
\end{align*}
and which has the property that
\begin{equation*}
\cb^{(\outsym, \insym)}(\hwv{c,\hout}, L_{-n}\Mvecmid \otimes \hwv{c,\hin})
= (-1)^{-n+1}(-h-n\hin +\hout) \cb^{(\outsym, \insym)}(\hwv{c,\hout}, \Mvecmid \otimes \hwv{c,\hin}),
\end{equation*}
for any $n>0$ and $\Mvecmid \in (\modMmid)_{(h)}$. 

\end{prop}
\begin{proof}
The proof is by induction with respect the PBW
filtration~$(\PBWfil{p}\modMmid)_{p \in \bN}$ of~$\modMmid$.
We recursively define
\begin{align*}
\cb^{(\outsym, \insym)}_{p}\colon
    \modMout (0)\times (\PBWfil{p}\modMmid \otimes \modMin (0))\to\C
\end{align*}
for~$p \in \bN$, and we show consistency and the desired
property.

The zeroth level of the filtration is just the highest
weight subspace, $\PBWfil{0}\modMmid = \modMmid(0)$, so the
required coincidence with $\cb^{(\outsym, \midsym, \insym)}$
fully determines~$\cb^{(\outsym, \insym)}_{0}$, and provides the
base case for the recursion.

Suppose then 
that $\cb^{(\outsym, \insym)}_{p}\colon \modMout (0)\times (\PBWfil{p}\modMmid \otimes \modMin (0))\to\C$
are well-defined and consistent for $p \le r-1$.
We want to
define $\cb^{(\outsym, \insym)}_{r}\colon \modMout (0)\times (\PBWfil{r}\modMmid \otimes \modMin (0))\to\C$ according to the required property, by setting
\begin{equation*}
	\cb^{(\outsym, \insym)}_{r}(\hwv{c,\hout}, L_{-n}\Mvecmid \otimes \hwv{c,\hin}):=(-1)^{-n+1}(-h-n\hin +\hout)\cb^{(\outsym, \insym)}_{r-1}(\hwv{c,\hout}, \Mvecmid \otimes \hwv{c,\hin}) 
\end{equation*}
for $n>0$ and $\Mvecmid \in \PBWfil{r-1}\modMmid \cap (\modMmid)_{(h)}$.
For well-definedness,
it suffices to show that
\begin{equation*}
	\cb^{(\outsym, \insym)}_{r}\left(\hwv{c,\hout}, (L_{-m}L_{-n}-L_{-n}L_{-m}-[L_{-m},L_{-n}])\Mvecmid \otimes \hwv{c,\hin}\right)=0,
\end{equation*}
for any $m,n>0$ and $\Mvecmid \in \PBWfil{r-2}\modMmid$.
We may assume that $\Mvecmid\in (\modMmid)_{(h)}$.
Note that
\begin{align*}
	&\cb^{(\outsym, \insym)}_{r}\left(\hwv{c,\hout}, L_{-m}L_{-n}\Mvecmid \otimes \hwv{c,\hin}\right)\\
	&=(-1)^{-m-n}\left(-h-n-m\hin+\hout\right)\left(-h-n\hin+\hout\right)\cb^{(\outsym, \insym)}_{r-2}(\hwv{c,\hout},\Mvecmid\otimes \hwv{c,\hin}).
\end{align*}
Hence we have
\begin{align*}
	&\cb^{(\outsym, \insym)}_{r}\left(\hwv{c,\hout}, (L_{-m}L_{-n}-L_{-n}L_{-m})\Mvecmid \otimes \hwv{c,\hin}\right) \\
	&=(-1)^{-m-n}\left((m-n)(-h+\hout)-(m^{2}-n^{2})\hin \right)\cb^{(\outsym, \insym)}_{r-2}(\hwv{c,\hout}, \Mvecmid \otimes \hwv{c,\hin}) \\
	&=(-m+n)(-1)^{-m-n+1}\left(-h-(m+n)\hin+\hout\right)\cb^{(\outsym, \insym)}_{r-2}(\hwv{c,\hout}, \Mvecmid \otimes \hwv{c,\hin})  \\
	&=(-m+n)\cb^{(\outsym, \insym)}_{r-1}(\hwv{c,\hout}, L_{-m-n}\Mvecmid \otimes \hwv{c,\hin}).
\end{align*}
Therefore $\cb^{(\outsym, \insym)}_{r}$ is well-defined.
Consistency is clear from the construction: for $s \le r$
the map~$\cb^{(\outsym, \insym)}_{r}$
coincides with~$\cb^{(\outsym, \insym)}_{s}$ on the
subspace $\modMout (0)\times (\PBWfil{s}\modMmid \otimes \modMin (0))$.
\end{proof}
\subsubsection*{Step 3}
As we anticipated, we now consider
the affinization of $\modMmid$, $\modMmidaff:=\C[t^{\pm 1}]\otimes \modMmid$.
Let us also introduce a $\Z$-grading of it as
\begin{equation*}
	\deg (t^{n}\otimes \Mvecmid):=k-n-1\quad \mbox{when }\Mvecmid\in \modMmid (k),\quad n\in \Z.
\end{equation*}

We extend the bilinear functional $\cb^{(\outsym, \insym)}$ to
\begin{align*}
\cbaff^{(\outsym, \insym)}\colon \modMout (0) \times (\modMmidaff \otimes \modMin (0))\to\C
\end{align*}
by
\begin{align}\label{eq: def cb out in affinized}
\cbaff^{(\outsym, \insym)}\left(\hwv{c,\hout}\otimes (t^{n}\otimes \Mvecmid)\otimes \hwv{c,\hin}\right)
:=\delta_{k-n-1,0}\cb^{(\outsym, \insym)}(\hwv{c,\hout}\otimes \Mvecmid \otimes \hwv{c,\hin})
\end{align}
for $n\in\Z$ and $\Mvecmid \in \modMmid (k)$.
In particular, $\cbaff^{(\outsym, \insym)}$ vanishes unless $t^{n}\otimes \Mvecmid\in \modMmidaff$ is of degree $0$.

\subsubsection*{Step 4}
The following fact will be used to
extend the bilinear functional $\cbaff^{(\outsym, \insym)}$
so that arbitrary vectors from $\modMin$ can be inserted.
\begin{prop}
\label{prop:Virasoro_action_modMmidaff}
For $m,n\in \Z$, and $\Mvecmid \in \modMmid$, we set
\begin{equation*}
	L_{m}(t^{n}\otimes \Mvecmid):=\sum_{k=0}^{\infty}\binom{m+1}{k}t^{m+n+1-k}\otimes L_{k-1}\Mvecmid, \quad C(t^{n}\otimes \Mvecmid):=0.
\end{equation*}
Then, this gives a representation of the Virasoro algebra
of central charge~$0$ on~$\modMmidaff$.
Furthermore, $L_{m}$ 
has degree~$-m$, i.e., when $\Mvecmid\in\modMmid$ is homogeneous, we have
\begin{equation*}
	\deg (L_{m}(t^{n}\otimes \Mvecmid))=\deg (t^{n}\otimes \Mvecmid)-m.
\end{equation*}
\end{prop}
\begin{proof}
Note that, for any $m,n,p\in\Z$ and $\Mvecmid \in \modMmid$,
\begin{align*}
	L_{m}L_{n}(t^{p}\otimes \Mvecmid)=\sum_{k,l=0}^{\infty}\binom{n+1}{k}\binom{m+1}{l}t^{m+n+p+2-k-l}\otimes L_{l-1}L_{k-1}\Mvecmid.
\end{align*}
Therefore,
\begin{align*}
	&(L_{m}L_{n}-L_{n}L_{m})(t^{p}\otimes \Mvecmid) \\
	&=\sum_{k,l=0}^{\infty}\binom{n+1}{k}\binom{m+1}{l}t^{m+n+p+2-k-l}\otimes [L_{l-1},L_{k-1}]\Mvecmid \\
	&=\sum_{k,l=0}^{\infty}\binom{n+1}{k}\binom{m+1}{l}(l-k)t^{m+n+p+2-k-l}\otimes L_{l+k-2}\Mvecmid  \\
	&=\sum_{k=0}^{\infty}\sum_{l=0}^{k}\binom{n+1}{k-l}\binom{m+1}{l}(l-(k-l))t^{m+n+p+2-k}\otimes L_{k-2}\Mvecmid
\end{align*}
Here, we use the following identity:
\begin{align*}
	&\sum_{l=0}^{k}\binom{n+1}{k-l}\binom{m+1}{l}(l-(k-l)) \\
	&=(m+1)\sum_{l=1}^{k}\binom{n+1}{k-l}\binom{m}{l-1}-(n+1)\sum_{l=0}^{k-1}\binom{n}{k-l-1}\binom{m+1}{l} \\
	&=(m-n)\binom{m+n+1}{k-1}.
\end{align*}
Hence, we have
\begin{align*}
	&(L_{m}L_{n}-L_{n}L_{m})(t^{p}\otimes \Mvecmid) \\
	&=(m-n)\sum_{k=0}^{\infty}\binom{m+n+1}{k}t^{m+n+p+1-k}\otimes L_{k-1}\Mvecmid \\
	&=[L_{m},L_{n}](t^{p}\otimes \Mvecmid),
\end{align*}
which implies that the action defines a representation of the Virasoro algebra of central charge $0$.
The latter property regarding the degree is obvious from the definition.
\end{proof}

Now we define the bilinear functional $\cbaff^{(\outsym)}\colon \modMout (0)\times (\modMmidaff \otimes \modMin) \to\C$ by
\begin{align}\label{eq: def cb out affinized}
\cbaff^{(\outsym)}\left(\hwv{c,\hout}, (t^{n}\otimes \Mvecmid)\otimes y\hwv{c,\hin}\right)
:=\cbaff^{(\outsym, \insym)}(\hwv{c,\hout}, \sigma (y)(t^{n}\otimes \Mvecmid )\otimes \hwv{c,\hin})
\end{align}
for $n\in\Z$, $\Mvecmid \in \modMmid$ and $y\in \UEA(\vir_{<0})$, where $\sigma$ is the anti-involution of $\UEA(\vir)$ defined by $\sigma (X):=-X$, $X\in \vir$.

\subsubsection*{Step 5}
Finally we extend the bilinear functional to the whole space $\modMout \times (\modMmidaff\otimes\modMin)$.
For that purpose, we define an action of the Virasoro algebra on $\modMmidaff\otimes \modMin$ by the so-called coproduct action:
\begin{equation*}
	X((t^{m}\otimes\Mvecmid)\otimes\Mvecin):=X(t^{m}\otimes\Mvecmid)\otimes \Mvecin+(t^{m}\otimes\Mvecmid)\otimes X\Mvecin,\quad X\in \vir.
\end{equation*}
Then $\modMmidaff\otimes \modMin$
becomes a representation of the Virasoro algebra of central charge~$c$.

Note that the tensor product $\modMmidaff\otimes \modMin$ is naturally $\Z$-graded so that
\begin{align*}
	\modMmidaff\otimes \modMin&=\bigoplus_{n\in\Z}\left(\modMmidaff\otimes \modMin\right)(n), \\
	\left(\modMmidaff\otimes \modMin\right)(n)&=\bigoplus_{k,r\in\N}\left(\C t^{k+r-n-1}\otimes \modMmid (k)\right)\otimes \modMin (r).
\end{align*}

We extend $\cbaff^{(\outsym)}$ to a bilinear functional $\cbaff\colon \modMout \times  (\modMmidaff \otimes \modMin) \to\C$ by
\begin{equation*}
	\cbaff (y\hwv{c,\hout}, \Mvec ):=\cbaff^{(\outsym)}(\hwv{c,\hout}, \theta (y)\Mvec), \quad y\in \UEA(\vir_{<0}), \quad \Mvec \in \modMmidaff\otimes\modMin ,
\end{equation*}
where $\theta$ is the anti-involution of $\UEA (\vir)$ defined by $\theta (L_{n})=L_{-n}$, $n\in\Z$ and $\theta (C)=C$.

Now we define the linear map (\ref{eq:intertw_op_meanttobe}) that is meant to be an intertwining operator of the desired type.
\begin{defn}
\label{defn:intertw_op_meanttobe}
Define a linear map
\begin{equation*}
	\intertw \colon \modMmid\to \Hom (\modMin,\modMout^{\prime})\{\apvar\}
\end{equation*}
by the matrix elements
\begin{equation*}
	\braket{\Mvecout,\intertw (\Mvecmid,\apvar)\Mvecin}:=\sum_{m\in \Z}\cbaff (\Mvecout, (t^{m}\otimes \Mvecmid)\otimes\Mvecin)\apvar^{\Delta-m-1},
\end{equation*}
where $\Mvecin\in \modMin$, $\Mvecmid\in \modMmid$ and $\Mvecout\in \modMout$ are arbitrary, and $\Delta=\hout-\hmid-\hin$.
\end{defn}

\subsection{Some properties of the linear map}
We are going to show that the linear map $\intertw$ constructed in Definition~\ref{defn:intertw_op_meanttobe} gives an intertwining operator.
We begin with collecting some immediate properties.

\begin{lem}
\label{lem:property_cbaff}
Let~$\Mvec \in \modMmidaff \otimes \modMin$. Then we have:
\begin{enumerate}
\item 	If $\Mvec \in (\modMmidaff \otimes \modMin )(k)$ and $\Mvecout\in \modMout (l)$, we have $\cbaff \left(\Mvecout , \Mvec\right)=0$ unless $k=l$.
\item 	For any $y\in\vir_{<0}$, we have $\cbaff \left(\hwv{c,\hout}, y\Mvec \right)=0$.
\item 	For any $y\in\vir_{0}$, we have $\cbaff \left(\hwv{c,\hout}, y\Mvec \right)=\cbaff \left(y\hwv{c,\hout},\Mvec\right)$.
\end{enumerate}
\end{lem}
\begin{proof}
{(1):}~Writing $\Mvecout=y\hwv{c,\hout}$ with $y\in \UEA (\vir_{<0})$, 
we have $\theta (y)\Mvec \in (\modMmidaff\otimes\modMin)(k-l)$.
Hence it suffices to show that
\begin{equation*}
\cbaff (\hwv{c,\hout},\Mvec)=\cbaff^{(\outsym)} (\hwv{c,\hout},\Mvec)=0,\quad \Mvec\in (\modMmidaff\otimes\modMin)(k)
\end{equation*}
unless $k\neq 0$.
We may further assume that
$\Mvec = (t^{n}\otimes \Mvecmid)\otimes y'\hwv{c,\hin}$ with 
$y'\in \UEA (\vir_{<0})$ to find
\begin{equation*}
\cbaff (\hwv{c,\hout},\Mvec)
= \cbaff ^{(\outsym, \insym)}(\hwv{c,\hout},\sigma (y')(t^{n}\otimes \Mvecmid)\otimes \hwv{c,\hin}).
\end{equation*}
Since by assumption
$\deg (\sigma (y')(t^{n}\otimes \Mvecmid))=k$, this value vanishes unless~$k=0$
from the definition~\eqref{eq: def cb out in affinized}
of $\cbaff ^{(\outsym, \insym)}$.

{(2):}~It is sufficient to consider the case
$y=L_{-n}$, $n>0$, and $\Mvec=(t^{m}\otimes \Mvecmid)\otimes y_{2}\hwv{c,\hin}$, $m\in\Z$, $\Mvecmid\in \modMmid$, $y_{2}\in \UEA(\vir_{<0})$. Then, since $L_{-n}y_{2}\in \UEA(\vir_{<0})$, we observe that
\begin{align*}
&\cbaff \left(\hwv{c,\hout}, L_{-n}\left((t^{m}\otimes \Mvecmid)\otimes y_{2}\hwv{c,\hin}\right)\right) \\
= \; &\cbaff^{(\outsym)}\left(\hwv{c,\hout}, L_{-n}(t^{m}\otimes \Mvecmid)\otimes y_{2}\hwv{c,\hin}\right)+\cbaff^{(\outsym)}\left(\hwv{c,\hout}, (t^{m}\otimes \Mvecmid)\otimes L_{-n}y_{2}\hwv{c,\hin}\right) \\
= \; &\cbaff^{(\outsym)}\left(\hwv{c,\hout}, \sigma (y_{2})L_{-n}(t^{m}\otimes \Mvecmid)\otimes \hwv{c,\hin}\right)+\cbaff^{(\outsym)}\left(\hwv{c,\hout}, \sigma (L_{-n}y_{2}) (t^{m}\otimes \Mvecmid)\otimes \hwv{c,\hin}\right) \\
= \; & \cbaff^{(\outsym)}\left(\hwv{c,\hout}, \sigma (y_{2})L_{-n}(t^{m}\otimes \Mvecmid)\otimes \hwv{c,\hin}\right)-\cbaff^{(\outsym)}\left(\hwv{c,\hout}, \sigma (y_{2})L_{-n}(t^{m}\otimes \Mvecmid )\otimes \hwv{c,\hin}\right)  \\
= \; & 0.
\end{align*}

{(3):}~The assertion is obvious for $y=C$,
so it suffices to consider the case of~$y=L_{0}$.
First, we assume that $\Mvec =(t^{n}\otimes \Mvecmid)\otimes\hwv{c,\hin}\in \modMmidaff\otimes \modMin (0)$, where $\Mvecmid \in \modMmid (k)$ is homogeneous. In this case,
\begin{align*}
&\cbaff^{(\outsym)}\left(\hwv{c,\hout}, L_{0}\left((t^{n}\otimes \Mvecmid )\otimes \hwv{c,\hin}\right)\right) \\
= \; &\delta_{k-n-1,0}\Bigl(\cb^{(\outsym, \insym)}\left(\hwv{c,\hout}, L_{-1}\Mvecmid \otimes \hwv{c,\hin}\right)
	+\cb^{(\outsym, \insym)}\left(\hwv{c,\hout}, L_{0}\Mvecmid \otimes \hwv{c,\hin}\right) \\
&\qquad+\cb^{(\outsym, \insym)}\left(\hwv{c,\hout}, \Mvecmid \otimes L_{0}\hwv{c,\hin}\right) \Bigr)\\
= \; &\delta_{k-n-1,0}\left((-\hmid-k-\hin+\hout)+\hmid+k+\hin\right)\cb^{(\outsym, \insym)}\left(\hwv{c,\hout}, \Mvecmid\otimes \hwv{c,\hin}\right) \\
= \; &\cbaff^{(\outsym, \insym)}\left(L_{0}\hwv{c,\hout}, (t^{n}\otimes \Mvecmid)\otimes \hwv{c,\hin}\right).
\end{align*}
Then consider a general vector~$\Mvecin \in \modMin$, and write it as
$\Mvecin = y' \hwv{c,\hin}$ for~$y' \in \UEA(\vir_{<0})$.
For $n\in\Z$ and $\Mvecmid \in \modMmid$,
we have
\begin{align*}
&\cbaff \left(\hwv{c,\hout}, L_{0}\left((t^{n}\otimes \Mvecmid)\otimes y' \hwv{c,\hin}\right)\right)\\
= \; &\cbaff^{(\outsym)}\left(\hwv{c,\hout}, L_{0}\left(t^{n}\otimes \Mvecmid \right)\otimes y' \hwv{c,\hin}\right)
  + \cbaff^{(\outsym)}\left(\hwv{c,\hout}, (t^{n}\otimes \Mvecmid)\otimes L_{0} y' \hwv{c,\hin}\right).
\end{align*}
Note that $[L_{0},y']\in U(\vir_{<0})$ and $\sigma ([L_{0},y'])=-[\sigma (y'),L_{0}]$. Hence, we have
\begin{align*}
&\cbaff \left(\hwv{c,\hout}, L_{0}\left((t^{n}\otimes \Mvecmid)\otimes y'\hwv{c,\hin}\right)\right)\\
= \; &\cbaff^{(\outsym, \insym)}\left(\hwv{c,\hout}, \sigma (y')L_{0}\left(t^{n}\otimes \Mvecmid\right)\otimes \hwv{c,\hin}\right)
+\cbaff^{(\outsym, \insym)}\left(\hwv{c,\hout}, \sigma (y')\left(t^{n}\otimes \Mvecmid\right)\otimes L_{0}\hwv{c,\hin}\right) \\
&\quad -\cbaff^{(\outsym, \insym)}\left(\hwv{c,\hout}, [\sigma(y'),L_{0}](t^{n}\otimes \Mvecmid)\otimes \hwv{c,\hin}\right) \\
= \; &\cbaff^{(\outsym, \insym)}\left(\hwv{c,\hout}, L_{0}\left(\sigma (y')(t^{n}\otimes \Mvecmid)\otimes \hwv{c,\hin}\right)\right) \\
= \; &\cbaff^{(\outsym, \insym)}\left(L_{0}\hwv{c,\hout}, \sigma (y')(t^{n}\otimes \Mvecmid)\otimes \hwv{c,\hin}\right) \\
= \; &\cbaff^{(\outsym)}\left(L_{0}\hwv{c,\hout}, (t^{n}\otimes \Mvecmid)\otimes y' \hwv{c,\hin}\right).
\end{align*}
This finishes the proof.
\end{proof}

\begin{prop}
\label{prop:theta-invariance_cbaff}
The bilinear functional $\cbaff$ is $\theta$-invariant in the sense that
\begin{equation*}
	\cbaff(y\Mvecout,  \Mvec)=\cbaff (\Mvecout, \theta (y)\Mvec), \quad y\in \UEA(\vir), \quad \Mvecout \in \modMout, \quad \Mvec \in \modMmidaff \otimes \modMin.
\end{equation*}
\end{prop}
\begin{proof}
Let us write $\Mvecout =y_{+}\hwv{c,\hout}$, where $y_{+}\in \UEA(\vir_{<0})$.
Since it follows from the PBW theorem that $\UEA (\vir)=(\UEA (\vir_{<0}) \, \UEA (\vir_{0}))\oplus \UEA (\vir)\vir_{>0}$,
we can uniquely write $yy_{+}=y_{1}+y_{2}$, $y_{1}\in \UEA (\vir_{<0}) \, \UEA (\vir_{0})$, $y_{2}\in \UEA (\vir)\vir_{>0}$.
Since $y_{2}$ annihilates the highest weight vector,
\begin{equation*}
	\cbaff (y\Mvecout , \Mvec )=\cbaff (yy_{+}\hwv{c,\hout} , \Mvec)=\cbaff (y_{1}\hwv{c,\hout},\Mvec).
\end{equation*}
From the definition of $\cbaff$ and (3) of Lemma \ref{lem:property_cbaff} (note that $\theta|_{\UEA (\vir_{0})}=\id_{\UEA (\vir_{0})}$), we have
\begin{equation*}
	\cbaff (y\Mvecout,  \Mvec)=\cbaff (\hwv{c,\hout}, \theta (y_{1})\Mvec).
\end{equation*}
Notice also that $\theta (y_{2})\in \vir_{<0}\UEA (\vir)$. Then, (2) of Lemma \ref{lem:property_cbaff} gives $\cbaff (\hwv{c,\hout}, \theta (y_{2})\Mvec)=0$.
Therefore, we have
\begin{align*}
\cbaff (y\Mvecout , \Mvec)
& \; = \cbaff (\hwv{c,\hout}, \theta (y_{1}+y_{2})\Mvec)=\cbaff (\hwv{c,\hout},\theta (yy_{+})\Mvec)=\cbaff (\hwv{c,\hout},\theta (y_{+})\theta (y)\Mvec) \\
& \; = \cbaff (y_{+}\hwv{c,\hout},\theta (y)\Mvec)=\cbaff (\Mvecout,\theta (y)\Mvec)
\end{align*}
as desired.
\end{proof}

\subsection{Translation property}
We show that the linear map $\intertw$ satisfies the Jacobi identity (\ref{eq:intertw_assoc}) and the translation property (\ref{eq:intertw_translation}).
Let us begin with the easier one; the translation property.

\begin{prop}
\label{prop:L(-1)-derivative_cbaff}
For $\Mvecmid\in\modMmid$, we have
\begin{equation*}
	\intertw (L_{-1}\Mvecmid,\apvar) = \frac{d}{d\apvar}\intertw (\Mvecmid,\apvar).
\end{equation*}
\end{prop}
\begin{proof}
We take arbitrary vectors $\Mvecin\in \modMin (k)$, $\Mvecmid\in \modMmid (l)$ and $\Mvecout\in \modMout (m)$
and consider the matrix element
\begin{equation*}
\braket{\Mvecout,\intertw (L_{-1}\Mvecmid,\apvar)\Mvecin}
= \sum_{n\in \Z}\cbaff (\Mvecout, (t^{n}\otimes L_{-1}\Mvecmid)\otimes\Mvecin)\apvar^{\Delta-n-1}.
\end{equation*}
Due to (1) of Lemma~\ref{lem:property_cbaff}, the coefficient $\cbaff (\Mvecout, (t^{n}\otimes L_{-1}\Mvecmid)\otimes\Mvecin)$
vanishes unless $n=k+l-m$.
Hence, we have
\begin{equation*}
	\braket{\Mvecout,\intertw (L_{-1}\Mvecmid,\apvar)\Mvecin}
=\cbaff (\Mvecout, (t^{k+l-m}\otimes L_{-1}\Mvecmid)\otimes\Mvecin)\apvar^{\Delta-k-l+m-1}.
\end{equation*}
Recall that the action of $L_{0}$ on $(t^{n-1}\otimes \Mvecmid)\otimes \Mvecin$, $n\in\Z$ reads
\begin{align*}
	L_{0}\left((t^{n-1}\otimes \Mvecmid)\otimes \Mvecin\right)
= (t^{n}\otimes L_{-1}\Mvecmid )\otimes \Mvecin +(\hmid+k+\hin+l)(t^{n-1}\otimes \Mvecmid )\otimes \Mvecin.
\end{align*}
Rearranging the terms, we get
\begin{align*}
	&\braket{\Mvecout,\intertw (L_{-1}\Mvecmid,\apvar)\Mvecin}\\
= \; &\apvar^{\Delta-k-l+m-1}\cbaff \left(\Mvecout,L_{0}\left((t^{k+l-m-1}\otimes \Mvecmid)\otimes \Mvecin\right)\right) \\
	&\quad -\apvar^{\Delta-k-l+m-1}(\hmid+k+\hin+l)\cbaff \left(\Mvecout,(t^{k+l-m-1}\otimes \Mvecmid)\otimes \Mvecin\right).
\end{align*}
Employing the $\theta$-invariance (Proposition~\ref{prop:theta-invariance_cbaff}) of the bilinear functional $\cbaff$,
we obtain
\begin{align*}
	&\braket{\Mvecout,\intertw (L_{-1}\Mvecmid,\apvar)\Mvecin}\\
= \; &(\Delta-k-l+m)\apvar^{\Delta-k-l+m-1}\cbaff \left(\Mvecout,(t^{k+l-m-1}\otimes \Mvecmid)\otimes \Mvecin\right) \\
= \; &\frac{d}{d\apvar}\left(x^{\Delta-k-l+m}\cbaff \left(\Mvecout,(t^{k+l-m-1}\otimes \Mvecmid)\otimes \Mvecin\right)\right).
\end{align*}
Here, we again observe the identity
\begin{align*}
	\braket{\Mvecout,\intertw (\Mvecmid,\apvar)\Mvecin}
= x^{\Delta-k-l+m}\cbaff \left(\Mvecout,(t^{k+l-m-1}\otimes \Mvecmid)\otimes \Mvecin\right)
\end{align*}
to conclude the desired result.
\end{proof}

\subsection{Jacobi identity}
We next show the Jacobi identity (\ref{eq:intertw_assoc}).
Since the Virasoro VOA is generated by the conformal vector $\voaconfvec$, it suffices to show the Jacobi identity for $\Vvec=\voaconfvec$.
For convenience, we introduce the generating series $L(\ipvar)=\sum_{n\in\Z}L_{n}\ipvar^{-n-2}$ whose coefficients are understood as the action of the Virasoro algebra on any representation.
The following commutation relations are sometimes useful:
\begin{align}
\label{eq:commutation_Virasoro_field}
	[L_{n}, L(\ipvar)]
		&=\ipvar^{n+1}\frac{d}{d\ipvar}L(\ipvar)+2(n+1)\ipvar^{n}L(\ipvar)+\frac{c}{2}\binom{n+1}{3}\ipvar^{n-2},\quad n\in \Z,
\end{align}
where $c$ is the central charge of representations under consideration.

In our setting, it is convenient to formulate the Jacobi 
identity in terms of matrix elements. We will therefore prove the following.
\begin{prop}
\label{prop:Jacobi_cbaff}
For $\Mvecin\in \modMin$, $\Mvecmid\in\modMmid$ and $\Mvecout\in \modMout$, we have
\begin{align*}
	&\ipvar_{0}^{-1}\delta\left(\frac{\ipvar_{1}-\apvar}{\ipvar_{0}}\right)\braket{\Mvecout , L(\ipvar_{1})\intertw (\Mvecmid,\apvar ) \Mvecin}
	-\ipvar_{0}^{-1}\delta\left(\frac{\apvar-\ipvar_{1}}{-\ipvar_{0}}\right)\braket{\Mvecout, \intertw (\Mvecmid,\apvar) L(\ipvar_{1})\Mvecin }  \\
= \; &\apvar^{-1}\delta\left(\frac{\ipvar_{1}-\ipvar_{0}}{\apvar}\right) \braket{\Mvecout,\intertw \left(L(\ipvar_{0})\Mvecmid,\apvar\right)\Mvecin  }\notag
\end{align*}
\end{prop}

In the following proof, it will be convenient to use the notation
\begin{equation*}
	\intertw_{t}(\Mvecmid,\apvar):=\sum_{n\in\Z}(t^{n}\otimes \Mvecmid)\apvar^{\Delta-n-1} \in \modMmidaff \{\apvar\}
\end{equation*}
for $\Mvecmid\in \modMmid$.
Then applying the bilinear functional~$\cbaff$ coefficientwise to
formal series in~$\apvar$, we may understand
\begin{equation*}
	\Braket{\Mvecout,\intertw (\Mvecmid,\apvar)\Mvecin}=\cbaff (\Mvecout,\intertw_{t}(\Mvecmid,\apvar)\otimes \Mvecin)
\end{equation*}
for $\Mvecin\in \modMin$, $\Mvecmid\in\modMmid$ and $\Mvecout\in \modMout$.
The following formula derived from Proposition~\ref{prop:Virasoro_action_modMmidaff} will be also useful:
\begin{equation}
\label{eq:Virasoro_action_modMmidaff_series}
	L_{n}\intertw_{t}(\Mvecmid,\apvar)=\sum_{k=0}^{\infty}\binom{n+1}{k}\apvar^{n+1-k}\intertw_{t} (L_{k-1}\Mvecmid,\apvar),\quad n\in\Z.
\end{equation}

The strategy of proving Proposition~\ref{prop:Jacobi_cbaff} is to reduce the Jacobi identity to the {\it commutativity} and the {\it associativity}.
Let us first observe the commutativity.
\begin{lem}
\label{lem:commutativity_intertw_aff}
For any $\Mvecin\in \modMin$, $\Mvecmid\in \modMmid$ and $\Mvecout \in \modMout$, we have
\begin{align*}
	&\braket{\Mvecout, L(\ipvar_{1}) \intertw (\Mvecmid,\apvar) \Mvecin}-\braket{\Mvecout, \intertw (\Mvecmid,\apvar) L(\ipvar_{1})\Mvecin} \\
= \; &\Res_{\ipvar_{0}}\apvar^{-1}\delta\left(\frac{\ipvar_{1}-\ipvar_{0}}{\apvar}\right)\braket{\Mvecout, \intertw \left(L(\ipvar_{0})\Mvecmid,\apvar\right)\Mvecin}.
\end{align*}
\end{lem}
\begin{proof}
From the property of the contragredient module (Lemma~\ref{lem: Virasoro action on the contragredient})
and the $\theta$-invariance of the bilinear functional $\cbaff$ (Proposition~\ref{prop:theta-invariance_cbaff}),
we can see that
\begin{equation*}
	\braket{\Mvecout, L(\ipvar_{1}) \intertw (\Mvecmid,\apvar) \Mvecin}
	=\cbaff (\Mvecout, L(\ipvar_{1})(\intertw_{t}(\Mvecmid,\apvar)\otimes \Mvecin)).
\end{equation*}
Hence, we have
\begin{align*}
	\braket{\Mvecout, L(\ipvar_{1}) \intertw (\Mvecmid,\apvar) \Mvecin}-\braket{\Mvecout, \intertw (\Mvecmid,\apvar) L(\ipvar_{1})\Mvecin}
	=\cbaff (\Mvecout, L(\ipvar_{1})\intertw_{t}(\Mvecmid,\apvar)\otimes \Mvecin).
\end{align*}
Therefore, the proof amounts to showing the identity
\begin{align*}
	L(\ipvar_{1})\intertw_{t}(\Mvecmid,\apvar)=\Res_{\ipvar_{0}}\apvar^{-1}\delta\left(\frac{\ipvar_{1}-\ipvar_{0}}{\apvar}\right)\intertw_{t} \left(L(\ipvar_{0})\Mvecmid,\apvar\right)
\end{align*}
in $\modMmidaff\{\apvar\}$, which is straightforward from (\ref{eq:Virasoro_action_modMmidaff_series}).
\end{proof}
We next state the associativity.
\begin{prop}
\label{prop:associativity_cbaff}
For any $\Mvecin\in \modMin$, there exists $k\in \N$ 
(depending on $\Mvecin$) such that
\begin{align*}
	(\ipvar_{0}+\apvar)^{k}\braket{\Mvecout, L(\ipvar_{0}+\apvar)\intertw (\Mvecmid,\apvar) \Mvecin}
=(\ipvar_{0}+\apvar )^{k}\braket{\Mvecout, \intertw (L(\ipvar_{0})\Mvecmid,\apvar) \Mvecin}
\end{align*}
holds for any $\Mvecmid \in \modMmid$ and $\Mvecout\in\modMout$.
\end{prop}

The proof of Proposition~\ref{prop:associativity_cbaff} requires some preliminaries.
\begin{lem}
\label{lem:pre-associativity_cbaff_left_right_top}
Let $\Mvecmid \in \modMmid$. Then, we have
\begin{align*}
	&\Res_{\ipvar_{0}}\ipvar_{0}^{-1}(\ipvar_{0}+\apvar)^{2}\braket{\hwv{c,\hout},L(\ipvar_{0}+\apvar)\intertw (\Mvecmid,\apvar) \hwv{c,\hin}} \\
= \; & \Res_{\ipvar_{0}}\ipvar_{0}^{-1}(\ipvar_{0}+\apvar)^{2}\braket{\hwv{c,\hout},\intertw \left(L(\ipvar_{0})\Mvecmid,\apvar\right)\hwv{c,\hin}}.
\end{align*}
\end{lem}
\begin{proof}
Assuming that $\Mvecmid\in \modMmid (k)$ is homogeneous,
we can see that both sides coincide with
\begin{equation*}
	\hout \apvar^{\Delta-k}\cb^{(\outsym, \insym)}\left(\hwv{c,\hout},\Mvecmid \otimes \hwv{c,\hin}\right).
\end{equation*}
\end{proof}

\begin{lem}
\label{lem:associativity_cbaff_left_right_top}
For any $\Mvecmid \in \modMmid$ and $i\in\N$, we have
\begin{align*}
	&(\ipvar_{0}+\apvar)^{2+i}\braket{\hwv{c,\hout},L(\ipvar_{0}+\apvar)\intertw (\Mvecmid,\apvar)\hwv{c,\hin}} \\
= \; &(\ipvar_{0}+\apvar)^{2+i} \braket{\hwv{c,\hout},\intertw \left(L(\ipvar_{0})\Mvecmid ,\apvar\right)\hwv{c,\hin}}.
\end{align*}
\end{lem}
\begin{proof}
Note that the assertion is equivalent to that
\begin{align}
\label{eq:associativity_cbaff_left_right_top_equivalent}
	&\Res_{\ipvar_{0}}\ipvar_{0}^{m}(\ipvar_{0}+\apvar)^{2+i}\cbaff \left(\hwv{c,\hout},L(\ipvar_{0}+\apvar)\left(\intertw_{t} (\Mvecmid,\apvar)\otimes \hwv{c,\hin}\right)\right) \\
= \; &\Res_{\ipvar_{0}}\ipvar_{0}^{m}(\ipvar_{0}+\apvar)^{2+i}\cbaff\left(\hwv{c,\hout},\intertw_{t} \left(L(\ipvar_{0})\Mvecmid ,\apvar\right)\otimes \hwv{c,\hin}\right) \notag
\end{align}
holds for any $\Mvecmid\in \modMmid$, $i\in \N$ and $m\in\Z$.
We show this in several different ranges of $m$.

When $m\in \N$, we have
\begin{align*}
	&\Res_{\ipvar_{0}}\ipvar_{0}^{m}(\ipvar_{0}+\apvar)^{2+i}L(\ipvar_{0}+\apvar)\left(\intertw_{t}(\Mvecmid,\apvar)\otimes \hwv{c,\hin}\right) \\
= \; &\Res_{\ipvar_{1}}(\ipvar_{1}-\apvar)^{m}\ipvar_{1}^{2+i}L(\ipvar_{1})\left(\intertw_{t}(\Mvecmid,\apvar)\otimes \hwv{c,\hin}\right)
\end{align*}
by changing the variables as $\ipvar_{1}=\ipvar_{0}+\apvar$.
Note that the formal series
\begin{align*}
	(\ipvar_{1}-\apvar)^{m}\ipvar_{1}^{2+i}\intertw_{t}(\Mvecmid,\apvar)\otimes L(\ipvar_{1})\hwv{c,\hin}
\end{align*}
does not involve any negative powers of $\ipvar_{1}$,
and hence its residue with respect to $\ipvar_{1}$ vanishes.
Therefore we can subtract the residue for free to obtain
\begin{align*}
	&\Res_{\ipvar_{0}}\ipvar_{0}^{m}(\ipvar_{0}+\apvar)^{2+i}L(\ipvar_{0}+\apvar)\left(\intertw_{t}(\Mvecmid,\apvar)\otimes \hwv{c,\hin}\right)  \\
= \; &\Res_{\ipvar_{1}}(\ipvar_{1}-\apvar)^{m}\ipvar_{1}^{2+i}L(\ipvar_{1})\left(\intertw_{t}(\Mvecmid,\apvar)\otimes \hwv{c,\hin}\right) 
	-\Res_{\ipvar_{1}}(\ipvar_{1}-\apvar)^{m}\ipvar_{1}^{2+i}\intertw_{t}(\Mvecmid,\apvar)\otimes L(\ipvar_{1})\hwv{c,\hin} \\
= \; &\Res_{\ipvar_{1}}\Res_{\ipvar_{0}}(\ipvar_{1}-\apvar)^{m}\ipvar_{1}^{2+i}\apvar^{-1}\delta\left(\frac{\ipvar_{1}-\ipvar_{0}}{\apvar}\right)\intertw_{t}\left(L(\ipvar_{0})\Mvecmid,\apvar\right)\otimes \hwv{c,\hin},
\end{align*}
where we used the commutativity (Lemma \ref{lem:commutativity_intertw_aff}) in the last equality.
We also use the identity
\begin{equation*}
	\apvar^{-1}\delta\left(\frac{\ipvar_{1}-\ipvar_{0}}{\apvar}\right)
=\ipvar_{1}^{-1}\delta\left(\frac{\apvar+\ipvar_{0}}{\ipvar_{1}}\right)
\end{equation*}
to obtain
\begin{align*}
	&\Res_{\ipvar_{0}}\ipvar_{0}^{m}(\ipvar_{0}+\apvar)^{2+i}L(\ipvar_{0}+\apvar)\left(\intertw_{t}(\Mvecmid,\apvar)\otimes \hwv{c,\hin}\right)\\
= \; &\Res_{\ipvar_{1}}\Res_{\ipvar_{0}}(\ipvar_{1}-\apvar)^{m}\ipvar_{1}^{2+i}\ipvar_{1}^{-1}\delta\left(\frac{\apvar+\ipvar_{0}}{\ipvar_{1}}\right)\intertw_{t}\left(L(\ipvar_{0})\Mvecmid,\apvar\right)\otimes \hwv{c,\hin} \notag \\
= \; &\Res_{\ipvar_{0}}\ipvar_{0}^{m}(\ipvar_{0}+\apvar)^{2+i}\intertw_{t} \left(L(\ipvar_{0})\Mvecmid ,\apvar\right)\otimes \hwv{c,\hin}. \notag
\end{align*}
This in particular implies (\ref{eq:associativity_cbaff_left_right_top_equivalent}) in the case when $m\in \N$.

We show (\ref{eq:associativity_cbaff_left_right_top_equivalent}) when $m<0$ by induction.
When $m=-1$, we see that
\begin{align*}
	&\Res_{\ipvar_{0}}\ipvar_{0}^{-1}(\ipvar_{0}+\apvar)^{2+i}\cbaff \left(\hwv{c,\hout},L(\ipvar_{0}+\apvar)\left(\intertw_{t}(\Mvecmid,\apvar)\otimes \hwv{c,\hin}\right)\right) \\
= \; &\sum_{r=0}^{i}\binom{i}{r}\Res_{\ipvar_{0}}\ipvar_{0}^{r-1}\apvar^{i-r}(\ipvar_{0}+\apvar)^{2}\cbaff \left(\hwv{c,\hout},L(\ipvar_{0}+\apvar)\left(\intertw_{t}(\Mvecmid,\apvar)\otimes \hwv{c,\hin}\right)\right).
\end{align*}
We divide the sum into the part of $r=0$ and those of $r=1,\dots, i$. For the former case, we apply Lemma \ref{lem:pre-associativity_cbaff_left_right_top},
and for the latter cases, we can already apply (\ref{eq:associativity_cbaff_left_right_top_equivalent}) to obtain
\begin{align*}
	&\Res_{\ipvar_{0}}\ipvar_{0}^{-1}(\ipvar_{0}+\apvar)^{2+i}\cbaff \left(\hwv{c,\hout},L(\ipvar_{0}+\apvar)\left(\intertw_{t}(\Mvecmid,\apvar)\otimes \hwv{c,\hin}\right)\right) \\
= \; &\sum_{r=0}^{i}\binom{i}{r}\Res_{\ipvar_{0}}\ipvar_{0}^{r-1}\apvar^{i-r}(\ipvar_{0}+\apvar)^{2}\cbaff \left(\hwv{c,\hout},\intertw_{t}(L(\ipvar_{0})\Mvecmid,\apvar)\otimes \hwv{c,\hin}\right) \\
= \; &\Res_{\ipvar_{0}}\ipvar_{0}^{-1}(\apvar +\ipvar_{0})^{2+j}\cbaff \left(\hwv{c,h_{\infty}},\intertw_{t}(L(\ipvar_{0})u^{1},\apvar)\otimes \hwv{c,h_{0}}\right).
\end{align*}
Therefore (\ref{eq:associativity_cbaff_left_right_top_equivalent}) holds for $m\geq -1$.

Suppose that $k\in\Zpos$ is such that (\ref{eq:associativity_cbaff_left_right_top_equivalent}) holds for $m\geq -k$.
In particular, we have
\begin{align*}
	&\Res_{\ipvar_{0}}\ipvar_{0}^{-k}(\ipvar_{0}+\apvar)^{2+i}\cbaff \left(\hwv{c,\hout},L(\ipvar_{0}+\apvar)\left(\intertw_{t} (\Mvecmid,\apvar)\otimes \hwv{c,\hin}\right)\right) \\
= \; &\Res_{\ipvar_{0}}\ipvar_{0}^{-k}(\ipvar_{0}+\apvar)^{2+i}\cbaff\left(\hwv{c,\hout},\intertw_{t} \left(L(\ipvar_{0})\Mvecmid ,\apvar\right)\otimes \hwv{c,\hin}\right)
\end{align*}
for any $\Mvecmid \in \modMmid$, and $i\in\N$.
We differentiate both sides in terms of $\apvar$.
The derivative of the left hand side (LHS) becomes
\begin{align*}
	\frac{d}{d\apvar}(\mathrm{LHS})
& \; =\Res_{\ipvar_{0}}\ipvar_{0}^{-k}\left(\frac{\dee}{\dee \ipvar_{0}}(\ipvar_{0}+\apvar)^{2+i}\right)\cbaff \left(\hwv{c,\hout},L(\ipvar_{0}+\apvar)\left(\intertw_{t}(\Mvecmid,\apvar)\otimes \hwv{c,\hin}\right)\right) \\
    &\quad +\Res_{\ipvar_{0}}\ipvar_{0}^{-k}(\ipvar_{0}+\apvar)^{2+i}\frac{\dee}{\dee \ipvar_{0}}\cbaff \left(\hwv{c,\hout},L(\ipvar_{0}+\apvar)\left(\intertw_{t}(\Mvecmid,\apvar)\otimes \hwv{c,\hin}\right)\right) \\
    &\quad +\Res_{\ipvar_{0}}\ipvar_{0}^{-k}(\ipvar_{0}+\apvar)^{2+i}\cbaff \left(\hwv{c,\hout},L(\ipvar_{0}+\apvar)\left(\intertw_{t}(L_{-1}\Mvecmid,\apvar)\otimes \hwv{c,\hin}\right)\right) \\
& \; =k\Res_{\ipvar_{0}}\ipvar_{0}^{-k-1}(\ipvar_{0}+\apvar)^{2+i}\cbaff \left(\hwv{c,\hout},L(\ipvar_{0}+\apvar)\left(\intertw_{t}(\Mvecmid,\apvar)\otimes \hwv{c,\hin}\right)\right) \\
	&\quad +\Res_{\ipvar_{0}}\ipvar_{0}^{-k}(\ipvar_{0}+\apvar)^{2+i}\cbaff \left(\hwv{c,\hout},L(\ipvar_{0}+\apvar)\left(\intertw_{t}(L_{-1}\Mvecmid,\apvar)\otimes \hwv{c,\hin}\right)\right).
\end{align*}
On the other hand, the derivative of the right hand side (RHS) is
\begin{align*}
	\frac{d}{d\apvar}(\mathrm{RHS})
& \; =\Res_{\ipvar_{0}}\ipvar_{0}^{-k}\left(\frac{\dee}{\dee \ipvar_{0}}(\ipvar_{0}+\apvar)^{2+i}\right)\cbaff \left(\hwv{c,\hout},\intertw_{t}(L(\ipvar_{0})\Mvecmid,\apvar)\otimes \hwv{c,\hin}\right) \\
	&\quad +\Res_{\ipvar_{0}}\ipvar_{0}^{-k}(\ipvar_{0}+\apvar)^{2+i}\cbaff\left(\hwv{c,\hout},\intertw_{t}(L_{-1}L(\ipvar_{0})\Mvecmid,\apvar)\otimes \hwv{c,\hin}\right) \\
& \; =k\Res_{\ipvar_{0}}\ipvar_{0}^{-k-1}(\ipvar_{0}+\apvar)^{2+i}\cbaff \left(\hwv{c,\hout},\intertw_{t}(L(\ipvar_{0})\Mvecmid,\apvar)\otimes \hwv{c,\hin}\right) \\
	&\quad +\Res_{\ipvar_{0}}\ipvar_{0}^{-k}(\ipvar_{0}+\apvar)^{2+i}\cbaff \left(\hwv{c,\hout},\intertw_{t}(L(\ipvar_{0})L_{-1}\Mvecmid,\apvar)\otimes \hwv{c,\hin}\right).
\end{align*}
Here we used the commutation relation (\ref{eq:commutation_Virasoro_field}) at $n=-1$ in the last equality.
Comparing these, due to the induction hypothesis, we obtain
\begin{align*}
	&\Res_{\ipvar_{0}}\ipvar_{0}^{-k-1}(\ipvar_{0}+\apvar)^{2+i}\cbaff \left(\hwv{c,\hout},L(\ipvar_{0}+\apvar)\left(\intertw_{t}(\Mvecmid,\apvar)\otimes \hwv{c,\hin}\right)\right) \\
= \; &\Res_{\ipvar_{0}}\ipvar_{0}^{-k-1}(\ipvar_{0}+\apvar)^{2+i}\cbaff \left(\hwv{c,\hout},\intertw_{t}(L(\ipvar_{0})\Mvecmid,\apvar)\otimes \hwv{c,\hin}\right).
\end{align*}
Hence, (\ref{eq:associativity_cbaff_left_right_top_equivalent}) holds also at $m=-k-1$.

In conclusion, we have shown that (\ref{eq:associativity_cbaff_left_right_top_equivalent})
holds for arbitrary $\Mvecmid\in \modMmid$, $i\in \N$ and $m\in\Z$,
which is equivalent to the desired result.
\end{proof}

\begin{proof}[Proof of Proposition~\ref{prop:associativity_cbaff}]
We introduce another filtration on $\modMin$ as
\begin{equation*}
	G^{d}\modMin:=\bigoplus_{k=0}^{d}\modMin (k), \quad d\in \N.
\end{equation*}
In particular, each subspace $G^{d}\modMin$ is finite dimensional.

For $d,r\in \N$, we name the following statement:

\noindent $\mathbf{P}[d;r]$: There exists $k\in \N$ depending on $d$ such that
\begin{align*}
&(\ipvar_{0}+\apvar)^{k}\cbaff (\Mvecout,L(\ipvar_{0}+\apvar)(\intertw_{t} (\Mvecmid,\apvar)\otimes \Mvecin)) \\
= \; &(\ipvar_{0}+\apvar)^{k} \cbaff (\Mvecout,\intertw_{t} \left(L(\ipvar_{0})\Mvecmid ,\apvar\right)\otimes \Mvecin )
\end{align*}
holds for any $\Mvecin\in G^{d}\modMin$, $\Mvecmid\in \modMmid$ and $\Mvecout\in \PBWfil{r}\modMout$.

To prove Proposition~\ref{prop:associativity_cbaff}, we employ the induction in $p,r\in \N$.
The statement $\mathbf{P}[0;0]$ is true due to Lemma~\ref{lem:associativity_cbaff_left_right_top}.

Assume that $\mathbf{P}[d;r]$ is true.
For $\Mvecin\in G^{d}\modMin$, $\Mvecmid\in \modMmid$, $\Mvecout\in \PBWfil{r}\modMout$ and $n>0$ such that $L_{-n}\Mvecin\in G^{d+1}\modMin$, we have
\begin{align*}
	&\cbaff \left(\Mvecout,L(\ipvar_{0}+\apvar)\left(\intertw_{t}(\Mvecmid,\apvar)\otimes L_{-n}\Mvecin \right)\right) \\
= \; &\cbaff \left(\Mvecout,L(\ipvar_{0}+\apvar)L_{-n}\left(\intertw_{t}(\Mvecmid,\apvar)\otimes \Mvecin\right)\right) \\
	&\quad -\sum_{k=0}^{\infty}\binom{-n+1}{k}\apvar^{-n+1-k}\cbaff \left(\Mvecout,L(\ipvar_{0}+\apvar)\left(\intertw_{t}(L_{k-1}\Mvecmid,\apvar)\otimes \Mvecin\right)\right).
\end{align*}

Using the commutation relation (\ref{eq:commutation_Virasoro_field}), we get
\begin{align*}
	&\cbaff \left(\Mvecout,L(\ipvar_{0}+\apvar)\left(\intertw_{t}(\Mvecmid,\apvar)\otimes L_{-n}\Mvecin \right)\right) \\
= \; &\cbaff \left(L_{n}\Mvecout,L(\ipvar_{0}+\apvar)\left(\intertw_{t}(\Mvecmid,\apvar)\otimes \Mvecin\right)\right) \\
	&\quad -(\ipvar_{0}+\apvar)^{-n+1}\frac{\dee}{\dee \ipvar_{0}}\cbaff \left(\Mvecout,L(\ipvar_{0}+\apvar)\left(\intertw_{t}(\Mvecmid,\apvar)\otimes \Mvecin\right)\right)  \\
	&\quad -2(-n+1)(\ipvar_{0}+\apvar)^{-n}\cbaff \left(\Mvecout,L(\ipvar_{0}+\apvar)\left(\intertw_{t}(\Mvecmid,\apvar)\otimes \Mvecin\right)\right) \\
	&\quad -\frac{c}{2}\binom{-n+1}{3} (\ipvar_{0}+\apvar)^{-n-2}\cbaff \left(\Mvecout,\intertw_{t}(\Mvecmid,\apvar)\otimes \Mvecin\right) \\
	&\quad -\sum_{k=0}^{\infty}\binom{-n+1}{k}\apvar^{-n+1-k}\cbaff \left(\Mvecout,L(\ipvar_{0}+\apvar)\left(\intertw_{t}(L_{k-1}\Mvecmid,\apvar)\otimes \Mvecin\right)\right).
\end{align*}
On the other hand, we have
\begin{align*}
	&\cbaff \left( \Mvecout,\intertw_{t}\left(L(\ipvar_{0})\Mvecmid,\apvar \right)\otimes L_{-n}\Mvecin \right) \\
= \; & \cbaff \left( L_{n}\Mvecout,\intertw_{t}\left(L(\ipvar_{0})\Mvecmid,\apvar \right)\otimes \Mvecin \right)
	 -\cbaff \left(\Mvecout, \left(L_{-n}\intertw_{t}(L(\ipvar_{0})\Mvecmid,\apvar)\right)\otimes \Mvecin\right) \\
= \; & \cbaff \left( L_{n}\Mvecout,\intertw_{t}\left(L(\ipvar_{0})\Mvecmid,\apvar \right)\otimes \Mvecin \right) \\
	&\quad -\sum_{k=0}^{\infty}\binom{-n+1}{k}\apvar^{-n+1-k}\cbaff \left(\Mvecout,\intertw_{t}\left(L_{k-1}L(\ipvar_{0})\Mvecmid,\apvar\right)\otimes \Mvecin\right) \\
= \; &\cbaff \left( L_{n}\Mvecout,\intertw_{t}\left(L(\ipvar_{0})\Mvecmid,\apvar \right)\otimes \Mvecin \right) \\
	&\quad -\sum_{k=0}^{\infty}\binom{-n+1}{k}\apvar^{-n+1-k}\cbaff \left(\Mvecout,\intertw_{t}\left(L(\ipvar_{0})L_{k-1}\Mvecmid,\apvar\right)\otimes \Mvecin\right) \\
	&\quad -\sum_{k=0}^{\infty}\binom{-n+1}{k}\apvar^{-n+1-k}\Biggl( \ipvar_{0}^{k}\frac{\dee}{\dee \ipvar_{0}}\cbaff \left(\Mvecout,\intertw_{t}(L(\ipvar_{0})\Mvecmid,\apvar)\otimes \Mvecin\right) \\
	&\qquad+2k \ipvar_{0}^{k-1}\cbaff \left(\Mvecout,\intertw_{t}(L(\ipvar_{0})\Mvecmid,\apvar)\otimes \Mvecin \right) \\
	&\qquad +\frac{c}{2}\binom{k}{3}\ipvar_{0}^{k-3}\cbaff \left(\Mvecout, \intertw_{t}(\Mvecmid,\apvar)\otimes \Mvecin\right)\Biggr).
\end{align*}
Using the identities
\begin{align}
\label{eq:identity_binom}
	\sum_{k=0}^{\infty}\binom{-n+1}{k}\binom{k}{j}\apvar^{-n+1-k}\ipvar_{0}^{k-j}
& \; =\binom{-n+1}{j}(\apvar+\ipvar_{0})^{-n+1-j},\quad j\in \N,
\end{align}
we obtain
\begin{align*}
	&\cbaff \left( \Mvecout,\intertw_{t}\left(L(\ipvar_{0})\Mvecmid,\apvar \right)\otimes L_{-n}\Mvecin \right) \\
= \; &\cbaff \left( L_{n}\Mvecout,\intertw_{t}\left(L(\ipvar_{0})\Mvecmid,\apvar \right)\otimes \Mvecin \right) \\
	&\quad -\sum_{k=0}^{\infty}\binom{-n+1}{k}\apvar^{-n+1-k}\cbaff \left(\Mvecout,\intertw_{t}\left(L(\ipvar_{0})L_{k-1}\Mvecmid,\apvar\right)\otimes \Mvecin\right) \\
	&\quad -(\apvar+\ipvar_{0})^{-n+1}\frac{\dee}{\dee \ipvar_{0}}\cbaff \left(\Mvecout,\intertw_{t}(L(\ipvar_{0})\Mvecmid,\apvar)\otimes \Mvecin\right) \\
	&\quad -2(-n+1) (\apvar+\ipvar_{0})^{-n}\cbaff \left(\Mvecout,\intertw_{t}(L(\ipvar_{0})\Mvecmid,\apvar)\otimes \Mvecin \right) \\
	&\quad -\frac{c}{2}\binom{-n+1}{3}(\apvar+\ipvar_{0})^{-n-2}\cbaff \left(\Mvecout, \intertw_{t}(\Mvecmid,\apvar)\otimes \Mvecin\right).
\end{align*}

By induction hypothesis $\mathbf{P}[d;r]$, there exists $m\in\N$ depending on $d$ and $n$ such that
\begin{align*}
	&(\ipvar_{0}+\apvar)^{m}\cbaff\left(\Mvecout ,L(\ipvar_{0}+\apvar)\left(\intertw_{t}(\Mvecmid,\apvar)\otimes L_{-n}\Mvecin\right)\right) \\
= \; &(\ipvar_{0}+\apvar)^{m}\cbaff \left(\Mvecout,\intertw_{t}\left(L(\ipvar_{0})\Mvecmid, \apvar\right)\otimes L_{-n}\Mvecin\right).
\end{align*}
Since $G^{d+1}\modMin$ is finite dimensional, we can maximize such $m$ so that it depends only on $d+1$,
implying that $\mathbf{P}[d+1;r]$ is also true.

Next, we apply $L_{-n}$ with $n>0$ at $\Mvecout\in \PBWfil{r}\modMout$.
On one hand, we have
\begin{align*}
	&\cbaff \left(L_{-n}\Mvecout,L(\ipvar_{0}+\apvar)\left(\intertw_{t}(\Mvecmid,\apvar)\otimes \Mvecin \right)\right) \\
= \; &\sum_{k=0}^{\infty}\binom{n+1}{k}\apvar^{n+1-k}\cbaff \left(\Mvecout, L(\ipvar_{0}+\apvar)\left(\intertw_{t}(L_{k-1}\Mvecmid,\apvar)\otimes \Mvecin \right)\right) \\
	&\quad +\cbaff \left(\Mvecout, L(\ipvar_{0}+\apvar)\left(\intertw_{t}(\Mvecmid,\apvar)\otimes L_{n}\Mvecin\right)\right) \\
	&\quad +(\ipvar_{0}+\apvar)^{n+1}\frac{\dee}{\dee \ipvar_{0}}\cbaff \left(\Mvecout, L(\ipvar_{0}+\apvar)\left(\intertw_{t}(\Mvecmid,\apvar)\otimes \Mvecin\right)\right) \\
	&\quad +2(n+1)(\ipvar_{0}+\apvar)^{n}\cbaff \left(\Mvecout, L(\ipvar_{0}+\apvar)\left(\intertw_{t}(\Mvecmid,\apvar)\otimes \Mvecin\right)\right) \\
	&\quad +\frac{c}{2}\binom{n+1}{3}(\ipvar_{0}+\apvar)^{n-2}\cbaff \left(\Mvecout, \intertw_{t}(\Mvecmid,\apvar)\otimes \Mvecin \right).
\end{align*}
Note that the rational function $\binom{n+1}{3}(\ipvar_{0}+\apvar)^{n-2}$ 
appearing in the last line is a polynomial of $\ipvar_{0}$ and $\apvar$ for all $n>0$.
On the other hand, we have
\begin{align*}
	&\cbaff \left(L_{-n}\Mvecout,\intertw_{t}(L(\ipvar_{0})\Mvecmid,\apvar)\otimes \Mvecin \right) \\
\; =&\sum_{k=0}^{\infty}\binom{n+1}{k}\apvar^{n+1-k}\cbaff \left(\Mvecout, \intertw_{t}(L(\ipvar_{0})L_{k-1}\Mvecmid,\apvar)\otimes \Mvecin \right) \\
	&\quad +(\ipvar_{0}+\apvar)^{n+1}\frac{\dee}{\dee \ipvar_{0}}\cbaff \left(\Mvecout, \intertw_{t}(L(\ipvar_{0})\Mvecmid,\apvar)\otimes \Mvecin\right) \\
	&\quad +2(n+1)(\ipvar_{0}+\apvar)^{n}\cbaff \left(\Mvecout, \intertw_{t}(L(\ipvar_{0})\Mvecmid,\apvar)\otimes \Mvecin\right) \\
	&\quad +\frac{c}{2}\binom{n+1}{3}(\ipvar_{0}+\apvar)^{n-2}\cbaff \left(\Mvecout, \intertw_{t}(\Mvecmid,\apvar)\otimes \Mvecin\right) \\
	&\quad +\cbaff \left(\Mvecout, \intertw_{t}(L(\ipvar_{0})\Mvecmid,\apvar)\otimes L_{n}\Mvecin\right).
\end{align*}
Here we again used the identities from (\ref{eq:identity_binom}).
Since $L_{n}\Mvecin\in G^{d}\modMin$, we can apply the induction hypothesis $\mathbf{P}[d;r]$
to conclude that there exists $m\in \N$ depending on $d$ such that
\begin{align*}
	&(\ipvar_{0}+\apvar)^{m}\cbaff\left(L_{-n}\Mvecout ,L(\ipvar_{0}+\apvar)\left(\intertw_{t}(\Mvecmid,\apvar)\otimes \Mvecin\right)\right) \\
= \; &(\ipvar_{0}+\apvar)^{m}\cbaff \left(L_{-n}\Mvecout,\intertw_{t}\left(L(\ipvar_{0})\Mvecmid, \apvar\right)\otimes \Mvecin\right)
\end{align*}
holds, implying that $\mathbf{P}[d;r+1]$ is also true.
\end{proof}

\begin{proof}[Proof of Proposition~\ref{prop:Jacobi_cbaff}]
The Jacobi identity is equivalent to having 
both of the commutativity (Lemma~\ref{lem:commutativity_intertw_aff}) and the associativity (Proposition~\ref{prop:associativity_cbaff}).
In \cite{Lepowsky_Li-VOA}, this equivalence is shown only for module maps,
but the same arguments work for intertwining operators.
\end{proof}

\newpage

\bibliographystyle{annotate}
\bibliography{generic_Virasoro}

\newcommand{\etalchar}[1]{$^{#1}$}
\begin{thebibliography}{FGST06b}

\bibitem[AK11]{AomotoKita}
Kazuhiko Aomoto and Michitake Kita.
\newblock {\em Theory of hypergeometric functions}.
\newblock Springer, 2011.


\bibitem[BB03a]{BB-zig_zag}
M.~Bauer and D.~Bernard.
\newblock {SLE}, {CFT} and zig-zag probabilities.
\newblock In {\em Proceedings of the conference `Conformal Invariance and
  Random Spatial Processes', Edinburgh}, 2003.

\bibitem[BB03b]{BB-CFTs_of_SLEs}
Michel Bauer and Denis Bernard.
\newblock Conformal field theories of stochastic {L}oewner evolutions.
\newblock {\em Comm. Math. Phys.}, 239(3):493--521, 2003.


\bibitem[BB04]{BauerBernard-conformal_transformations}
Michel Bauer and Denis Bernard.
\newblock {Conformal transformations and the SLE partition function
  martingale}.
\newblock {\em Annales Henri Poincar\'e}, 5(2):289--326, 2004.


\bibitem[BKJ01]{BakalovKirillovJr2001}
B.~Bakalov and A.~Kirillov~Jr.
\newblock {\em Lectures on Tensor Categories and Modular Functors}, volume~21
  of {\em University Lecture Series}.
\newblock American Mathematical Society, 2001.


\bibitem[BPZ84a]{BPZ-infinite_conformal_symmetry_in_2d_QFT}
A.~A. Belavin, A.~M. Polyakov, and A.~B. Zamolodchikov.
\newblock Infinite conformal symmetry in two-dimensional quantum field theory.
\newblock {\em Nucl. Phys. B}, 241:333--380, 1984.


\bibitem[BPZ84b]{BPZ-infinite_conformal_symmetry_of_critical_fluctiations}
A.~A. Belavin, A.~M. Polyakov, and A.~B. Zamolodchikov.
\newblock Infinite conformal symmetry of critical fluctuations in two
  dimensions.
\newblock {\em J. Stat. Phys.}, 34:763--774, 1984.


\bibitem[BSA88]{BSA-degenerate_CFTs_and_explicit_expressions}
L.~Benoit and Y.~Saint-Aubin.
\newblock Degenerate conformal field theories and explicit expressions for some
  null vectors.
\newblock {\em Phys. Lett. B}, 215(3):517--522, 1988.


\bibitem[CJH{\etalchar{+}}21]{CJORY-tensor_categories_arising_from_Virasoro_algebra}
Thomas Creutzig, Cuipo Jiang, Florencia~Orosz Hunziker, David Ridout, and
  Jinwei Yang.
\newblock Tensor categories arising from the {Virasoro} algebra.
\newblock {\em Advances in Mathematics}, 380:107601, 2021.


\bibitem[CLR21]{CreutzigLentnerRupert}
Thomas Creutzig, Simon Lentner, and Matthew Rupert.
\newblock Characterizing braided tensor categories associated to logarithmic
  vertex operator algebras.
\newblock {\em https://arxiv.org/abs/2104.13262}, 2021.


\bibitem[DFMS97]{DMS-CFT}
P.~Di~Francesco, P.~Mathieu, and D.~S\'en\'echal.
\newblock {\em {C}onformal {F}ield {T}heory}.
\newblock Springer Verlag, New York, 1997.


\bibitem[Dri89]{Drinfeld1989}
V.~G. Drinfeld.
\newblock Quasi-{Hopf} algebras.
\newblock {\em Algebra i Analiz}, 1(6):114--148, 1989.


\bibitem[Dub07]{Dubedat-commutation}
J.~Dub{\'e}dat.
\newblock Commutation relations for $\mathrm{SLE}$.
\newblock {\em Comm. Pure Appl. Math.}, 60(12):1792--1847, 2007.


\bibitem[Dub15]{Dubedat-fusion}
J.~Dub{\'e}dat.
\newblock $\mathrm{SLE}$ and {Virasoro} representations: Fusion.
\newblock {\em Comm. Math. Phys.}, 336(2):761--809, 2015.


\bibitem[FBZ04]{FrenkelBen-Zvi2004}
E.~Frenkel and D.~Ben-Zvi.
\newblock {\em Vertex Algebras and Algebraic Curves}, volume~88 of {\em
  Mathematical Surveys and Monographs}.
\newblock American Mathematical Society, 2nd edition, 2004.


\bibitem[Fel89]{Felder-BRST_approach}
G.~Felder.
\newblock {BRST} approach to minimal models.
\newblock {\em Nucl. Phys. B}, 317(1):215--236, 1989.
\newblock Erratum ibid., 324(2):548, 1989.


\bibitem[FF90]{FF-representations}
B.~L. Feigin and D.~B. Fuchs.
\newblock Representations of the {V}irasoro algebra.
\newblock In {\em Representation of Lie groups and related topics}, volume~7 of
  {\em Adv. Stud. Contemp. Math.}, pages 465--554. Gordon and Breach, New York,
  1990.

\bibitem[FGST06a]{FeiginGainutdinovSemikhatovTipunin2006b}
B.~L. Feigin, A.~M. Gainutdinov, A.~M. Semikhatov, and I.~Y. Tipunin.
\newblock The {Kazhdan-Lusztig} correspondence for the representation category
  of the triplet {W}-algebra in logarithmic {CFT}.
\newblock {\em Theor. Math. Phys.}, 148:1210--1235, 2006.


\bibitem[FGST06b]{FeiginGainutdinovSemikhatovTipunin2006a}
B.~L. Feigin, A.~M. Gainutdinov, A.~M. Semikhatov, and I.~Y. Tipunin.
\newblock Modular group representations and fusion in logarithmic conformal
  field theories and in the quantum group center.
\newblock {\em Commun. Math. Phys.}, 265:47--93, 2006.


\bibitem[FHL93]{FrenkelHuangLepowsky1993}
I.~B. Frenkel, Y.-Z. Huang, and J.~Lepowsky.
\newblock {\em On Axiomatic Approaches to Vertex Operator Algebras and
  Modules}, volume 104 of {\em Memoirs of the American Mathematical Society}.
\newblock American Mathematical Society, 1993.


\bibitem[FK15]{FK-solution_space_for_a_system_of_null_state_PDEs_all}
S.~M. Flores and P.~Kleban.
\newblock A solution space for a system of null-state partial differential
  equations, {Parts I-IV}.
\newblock {\em Comm. Math. Phys.}, 2015.
\newblock 333(1):389--434 ,333(1):435--481, 333(2):597--667, 333(2):669--715.


\bibitem[FLM89]{FLM-VOAs_and_the_Monster}
Igor Frenkel, James Lepowsky, and Arne Meurman.
\newblock {\em Vertex operator algebras and the Monster}.
\newblock Academic press, 1989.


\bibitem[FP18a]{FP-generators_projectors_and_the_JW}
Steven~M Flores and Eveliina Peltola.
\newblock {Generators, projectors, and the Jones-Wenzl algebra}.
\newblock https://arxiv.org/abs/1811.12364, 2018.

\bibitem[FP18b]{FP-standard_modules_radicals_and_the_valenced_TL}
Steven~M Flores and Eveliina Peltola.
\newblock {Standard modules, radicals, and the valenced Temperley-Lieb
  algebra}.
\newblock https://arxiv.org/abs/1801.10003, 2018.

\bibitem[FP20]{FP-quantum_Schur_Weyl}
Steven~M Flores and Eveliina Peltola.
\newblock {Higher-spin quantum and classical Schur-Weyl duality for
  $\mathfrak{sl}_2$}.
\newblock https://arxiv.org/abs/2008.06038, 2020.

\bibitem[FP21]{FP-monodromy_invariant_correlations}
S.~M. Flores and E.~Peltola.
\newblock Monodromy invariant {CFT} correlation functions of first column kac
  operators, 2021.
\newblock In preparation.

\bibitem[Fre07]{Frenkel2007}
E.~Frenkel.
\newblock Lectures on the {Langlands} program and conformal field theory.
\newblock In {\em Frontiers in Number Theory, Physics, and Geometry II}, pages
  387--533. Springer, 2007.

\bibitem[FW91]{FW-topological_representation_of_Uqsl2}
G.~Felder and C.~Wieczerkowski.
\newblock Topological representation of the quantum group {$U_{q}(sl_{2})$}.
\newblock {\em Comm. Math. Phys.}, 138(3):583--605, 1991.


\bibitem[FZ92]{FrenkelZhu1992}
I.~B. Frenkel and Y.~Zhu.
\newblock Vertex operator algebras associated to representations of affine and
  {Virasoro} algebras.
\newblock {\em Duke Math. J.}, 66:123--168, 1992.


\bibitem[FZ12]{FrenkelZhu2012}
I.~Frenkel and M.~Zhu.
\newblock Vertex algebras associated to modified regular representations of the
  {Virasoro} algebra.
\newblock {\em Adv. Math.}, 229:3468--3507, 2012.


\bibitem[Gaw99]{Gawedzki-lectures_on_CFT}
Krzysztof Gaw{\c e}dzki.
\newblock Lectures on conformal field theory.
\newblock In {\em Quantum fields and strings: A course for mathematicians (IAS
  Princeton)}, volume 1--2, pages 727--805. Amer. Math. Soc., Providence, RI,
  1999.

\bibitem[GN21]{GannonNegron-quantum_sltwo_and_log_VOAs}
Terry Gannon and Cris Negron.
\newblock {Quantum SL(2) and logarithmic vertex operator algebras at
  (p,1)-central charge}.
\newblock {\em https://arxiv.org/abs/2104.12821}, 2021.


\bibitem[GRAS96]{GRS-quantum_groups_in_2d_physics}
C.~G. G{\'o}mez, M.~Ruiz-Altaba, and G.~Sierra.
\newblock {\em Quantum groups in two-dimensional physics}.
\newblock Cambridge University Press, 1996.


\bibitem[GSW87]{GreenSchwarzWitten_superstring}
M.~B. Green, J.~H. Schwarz, and E.~Witten.
\newblock {\em Superstring theory I, II}.
\newblock Cambridge University Press, 1987.


\bibitem[HKJL15]{Huang--Kirillov--Lepowsky}
Y.-Z. Huang, A.~Kirillov~Jr., and J.~Lepowsky.
\newblock Braided tensor categories and extensions of vertex operator algebras.
\newblock {\em Commun. Math. Phys.}, 337:1143--1159, 2015.


\bibitem[HL92]{HuangLepowsky1992}
Y.-Z. Huang and J.~Lepowsky.
\newblock Toward a theory of tensor products for representations of a vertex
  operator algebra.
\newblock In {\em Proceedings of 20th International Conference on Differential
  Geometric Methods in Theoretical Physics, New York, 1991}, volume~1, pages
  344--354. World Scientific, 1992.

\bibitem[HL94]{HuangLepowsky1994}
Y.-Z. Huang and J.~Lepowsky.
\newblock Tensor products of modules for a vertex operator algebra and vertex
  tensor categories.
\newblock In {\em Lie Theory and Geometry, in Honor of Bertram Kostant}, pages
  349--383. Birkh{\"a}user, Boston, 1994.

\bibitem[HL95a]{Huang--LepowskyI}
Y.-Z. Huang and J.~Lepowsky.
\newblock A theory of tensor products for module categories for a vertex
  operator algebra, {I}.
\newblock {\em Selecta Mathematica, New Series}, 1(4):699--756, 1995.


\bibitem[HL95b]{Huang--LepowskyII}
Y.-Z. Huang and J.~Lepowsky.
\newblock A theory of tensor products for module categories for a vertex
  operator algebra, {II}.
\newblock {\em Selecta Mathematica, New Series}, 1(4):757--786, 1995.


\bibitem[HL95c]{Huang--LepowskyIII}
Y.-Z. Huang and J.~Lepowsky.
\newblock A theory of tensor products for module categories for a vertex
  operator algebra, {III}.
\newblock {\em J. Pure Appl. Algebra}, 100:141--171, 1995.


\bibitem[HLZ10a]{HuangLepowskyZhangII}
Y.-Z. Huang, J.~Lepowsky, and L.~Zhang.
\newblock Logarithmic tensor category theory, {II}: Logarithmic formal calculus
  and properties of logarithmic intertwining operators, 2010.
\newblock arXiv:1012.4196.

\bibitem[HLZ10b]{HuangLepowskyZhangIII}
Y.-Z. Huang, J.~Lepowsky, and L.~Zhang.
\newblock Logarithmic tensor category theory, {III}: Intertwining maps and
  tensor product bifunctors, 2010.
\newblock arXiv:1012.4197.

\bibitem[HLZ10c]{HuangLepowskyZhangIV}
Y.-Z. Huang, J.~Lepowsky, and L.~Zhang.
\newblock Logarithmic tensor category theory, {IV}: Constructions of tensor
  product bifunctors and the compatibility conditions, 2010.
\newblock arXiv:1012.4198.

\bibitem[HLZ10d]{HuangLepowskyZhangV}
Y.-Z. Huang, J.~Lepowsky, and L.~Zhang.
\newblock Logarithmic tensor category theory, {V}: Convergence condition for
  intertwining maps and the corresponding compatibility condition, 2010.
\newblock arXiv:1012.4199.

\bibitem[HLZ10e]{HuangLepowskyZhangVI}
Y.-Z. Huang, J.~Lepowsky, and L.~Zhang.
\newblock Logarithmic tensor category theory, {VI}: Expansion condition,
  associativity of logarithmic intertwining operators, and the associativity
  isomorphisms, 2010.
\newblock arXiv:1012.4202.

\bibitem[HLZ11a]{HuangLepowskyZhangVII}
Y.-Z. Huang, J.~Lepowsky, and L.~Zhang.
\newblock Logarithmic tensor category theory, {VII}: Convergence and extension
  properties and applications to expansion for intertwining maps, 2011.
\newblock arXiv:1110.1929.

\bibitem[HLZ11b]{HuangLepowskyZhangVIII}
Y.-Z. Huang, J.~Lepowsky, and L.~Zhang.
\newblock Logarithmic tensor category theory, {VIII}: Braided tensor category
  structure on categories of generalized modules for a conformal vertex
  algebra, 2011.
\newblock arXiv:1110.1931.

\bibitem[HLZ14]{HuangLepowskyZhangI}
Y.-Z. Huang, J.~Lepowsky, and L.~Zhang.
\newblock Logarithmic tensor category theory for generalized modules for a
  conformal vertex algebra, {I}: Introduction and strongly graded algebras and
  their generalized modules.
\newblock In {\em Conformal Field Theories and Tensor Categories}, pages
  169--248. Springer, 2014.

\bibitem[Hua95]{Huang1995}
Y.-Z. Huang.
\newblock A theory of tensor products for module categories for a vertex
  operator algebra, {IV}.
\newblock {\em J. Pure Appl. Algebra}, 100:173--216, 1995.


\bibitem[Hua05]{Huang2005}
Y.-Z. Huang.
\newblock Differential equations and intertwining operators.
\newblock {\em Commun. Contemp. Math.}, 7:375--400, 2005.


\bibitem[Hua12]{Huang-CFT_and_VOA}
Yi-Zhi Huang.
\newblock {\em Two-dimensional conformal geometry and vertex operator
  algebras}, volume 148.
\newblock Springer Science \& Business Media, 2012.


\bibitem[IK11]{IK-representation_theory_of_the_Virasoro_algebra}
K.~Iohara and Y.~Koga.
\newblock {\em Representation theory of the Virasoro algebra}.
\newblock Springer Monographs in Mathematics. Springer, 2011.


\bibitem[JJK16]{JJK-SLE_boundary_visits}
N.~Jokela, M.~J{\"a}rvinen, and K.~Kyt{\"o}l{\"a}.
\newblock $\mathrm{SLE}$ boundary visits.
\newblock {\em Ann. Henri Poincar\'e}, 17(6):1263--1330, 2016.
\newblock arXiv:1311.2297, 2013.


\bibitem[Kac79]{Kac-contravariant_form}
Victor~G Kac.
\newblock Contravariant form for infinite-dimensional {Lie} algebras and
  superalgebras.
\newblock In {\em Group theoretical methods in physics}, pages 441--445.
  Springer, 1979.

\bibitem[Kac97]{Kac-vertex_algebras}
Victor Kac.
\newblock {\em Vertex algebras for beginners}, volume~10 of {\em University
  Lecture Series}.
\newblock American Mathematical Society, Providence, RI, 1997.


\bibitem[Kas95]{Kassel-quantum_groups}
Christian Kassel.
\newblock {\em Quantum groups}, volume 155 of {\em Graduate Texts in
  Mathematics}.
\newblock Springer Verlag, 1995.


\bibitem[Kaw15]{Kawahigashi2015}
Y.~Kawahigashi.
\newblock Conformal field theory, tensor categories and operator algebras.
\newblock {\em J. Phys. A: Math. Theor.}, 48:303001, 2015.


\bibitem[KKP19]{KKP-conformal_blocks}
A.~Karrila, K.~Kyt{\"o}l{\"a}, and E.~Peltola.
\newblock Conformal blocks, $q$-combinatorics, and quantum group symmetry.
\newblock {\em Ann. Inst. Henri Poincar\'e D}, 6(3):449--487, 2019.
\newblock \url{doi:10.4171/aihpd/88}, Preprint:
  \url{https://arxiv.org/abs/1709.00249}.


\bibitem[KL94]{KL-tensor_structures_affine_Lie}
David Kazhdan and George Lusztig.
\newblock {Tensor structures arising from affine Lie algebras I-IV}.
\newblock {\em Journal of the American Mathematical Society}, 1994.
\newblock 6(4):905--947, 6(4):949--1011, 7(2):335--381, 7(2):383--453.


\bibitem[KP16]{KP-pure_partition_functions_of_multiple_SLEs}
K.~Kyt{\"o}l{\"a} and E.~Peltola.
\newblock Pure partition functions of multiple $\mathrm{SLE}$s.
\newblock {\em Comm. Math. Phys.}, 346(1):237--292, 2016.
\newblock Preprint: \url{http://arxiv.org/abs/1506.02476}.


\bibitem[KP20]{KP-conformally_covariant_bdry_correlations}
K.~Kyt{\"o}l{\"a} and E.~Peltola.
\newblock Conformally covariant boundary correlation functions with a quantum
  group.
\newblock {\em J. Eur. Math. Soc.}, 22:55--118, 2020.
\newblock \url{doi:10.4171/JEMS/917}, Preprint:
  \url{https://arxiv.org/abs/1408.1384}.


\bibitem[KS11]{KondoSaito2011}
H.~Kondo and Y.~Saito.
\newblock Indecomposable decomposition of tensor products of modules over the
  restricted quantum universal enveloping algebra associated to $sl_{2}$.
\newblock {\em J. Algebra}, 330:103--129, 2011.


\bibitem[Law08]{Lawler-conformally_invariant_processes_in_the_plane}
Gregory~F Lawler.
\newblock {\em Conformally invariant processes in the plane}.
\newblock Number 114 in Mathematical Surveys and Monographs. American
  Mathematical Soc., 2008.


\bibitem[Len21]{Lentner2021}
Simon~D. Lentner.
\newblock Quantum groups and nichols algebras acting on conformal field
  theories.
\newblock {\em Advances in Mathematics}, 378:107517, 2021.


\bibitem[Li99]{Li99}
H.~Li.
\newblock Determining fusion rules by {$A(V)$}-modules and bimodules.
\newblock {\em J. Algebra}, 212:515--556, 1999.


\bibitem[LL04]{Lepowsky_Li-VOA}
J.~Lepowsky and H.~Li.
\newblock {\em Introduction to Vertex Operator Algebras and Their
  Representations}.
\newblock Birkh{\"a}user Boston, 2004.


\bibitem[Lus93]{Lusztig1993}
G.~Lusztig.
\newblock {\em Introduction to quantum groups}.
\newblock Birkh{\"a}user, 1993.


\bibitem[McR16]{McRae-nonneg_integer_level_affine_Lie_algebra_tensor_cat}
R.~McRae.
\newblock {Non-Negative Integral Level Affine Lie Algebra Tensor Categories and
  Their Associativity Isomorphisms}.
\newblock {\em Comm. Math. Phys.}, 346:349--395, 2016.


\bibitem[MR89]{MR-comment_on_quantum_group_symmetry_in_CFT}
G.~Moore and N.~Reshetikhin.
\newblock A comment on quantum group symmetry in conformal field theory.
\newblock {\em Nucl. Phys. B}, 328(3):557--574, 1989.


\bibitem[Mus10]{Mussardo-statistical_field_theory}
Giuseppe Mussardo.
\newblock {\em Statistical field theory: an introduction to exactly solved
  models in statistical physics}.
\newblock Oxford University Press, 2010.


\bibitem[Nah00]{Nahm-bridge_over_troubled_waters}
W.~Nahm.
\newblock Conformal field theory: a bridge over troubled waters, in: Quantum
  field theory - a twentieth century profile.
\newblock In {\em Hindustani Book Agency and Indian National Science Academy},
  pages 571--604, 2000.

\bibitem[NT11]{NagatomoTsuchiya2011}
K.~Nagatomo and A.~Tsuchiya.
\newblock The triplet vertex operator algebra {$W(p)$} and the restricted
  quantum group at root of unity.
\newblock {\em Adv. Stud. Pure Math.}, 61:1--49, 2011.


\bibitem[PS90]{PS-common_structures_between_finite_systems_and_CFTs}
V.~Pasquier and H.~Saleur.
\newblock Common structures between finite systems and conformal field theories
  through quantum groups.
\newblock {\em Nucl. Phys. B}, 330(2):523--556, 1990.


\bibitem[RRRA91]{RRR-contour_picture_of_quantum_groups_in_CFT}
C.~Ramirez, H.~Ruegg, and M.~Ruiz-Altaba.
\newblock The contour picture of quantum groups in conformal field theories.
\newblock {\em Nucl. Phys. B}, 364(1):195--233, 1991.


\bibitem[RS05]{RS-basic_properties}
Steffen Rohde and Oded Schramm.
\newblock Basic properties of {SLE}.
\newblock {\em Ann. of Math.}, 161(2):883--924, 2005.


\bibitem[SV91]{SV-quantum_groups_and_homology_of_local_systems}
Vadim~V. Schechtman and Alexander~N. Varchenko.
\newblock Quantum groups and homology of local systems.
\newblock In Akira Fujiki, Kazuya Kato, Yujiro Kawamata, Toshiyuki Katsura, and
  Yoichi Miyaoka, editors, {\em ICM-90 Satellite Conference Proceedings}, pages
  182--197, Tokyo, 1991. Springer Japan.

\bibitem[TV14]{TeschnerVartanov2014}
J.~Teschner and G.~Vartanov.
\newblock $6j$ symbols for the modular double, quantum hyperbolic geometry, and
  supersymmetric gauge theories.
\newblock {\em Lett. Math. Phys.}, 104:527--551, 2014.


\bibitem[TW13]{TsuchiyaWood2013}
A.~Tsuchiya and S.~Wood.
\newblock The tensor structure on the representation category if the
  $\mathcal{W}_{p}$ triplet algebra.
\newblock {\em J. Phys. A: Math. Theor.}, 46:445203, 2013.


\bibitem[TW14]{TsuchiyaWood_IMRN_2014}
Akihiro Tsuchiya and Simon Wood.
\newblock {On the Extended W-Algebra of Type $\mathfrak{sl}_{2}$ at Positive
  Rational Level}.
\newblock {\em International Mathematics Research Notices},
  2015(14):5357--5435, 06 2014.


\bibitem[Var95]{Varchenko-multidimensional_hypergeometric_functions_and_representation_theory_of_Lie_algebras_and_quantum_groups}
A.~Varchenko.
\newblock {\em Multidimensional hypergeometric functions and representation
  theory of Lie algebras and quantum groups}, volume~21 of {\em Advanced Series
  in Mathematical Physics}.
\newblock World Scientific, 1995.


\bibitem[Wan93]{Wang-rationality_of_Virasoro_VOA}
Weiqiang Wang.
\newblock Rationality of virasoro vertex operator algebras.
\newblock {\em International Mathematics Research Notices}, 1993(7):197--211,
  1993.


\bibitem[Wer04]{Werner-random_planar_curves_and_SLE}
Wendelin Werner.
\newblock Random planar curves and {Schramm-Loewner} evolutions.
\newblock In {\em Lectures on probability theory and statistics}, pages
  107--195. Springer, 2004.

\bibitem[Xu98]{Xu1998}
X.~Xu.
\newblock {\em Introduction to Vertex Operator Superalgebras and Their
  Modules}, volume 456 of {\em Mathematics and Its Applications}.
\newblock Springer, 1998.


\bibitem[Zhu96]{Zhu96}
Y.~Zhu.
\newblock Modular invariance of characters of vertex operator algebras.
\newblock {\em J. Amer. Math. Soc.}, 9:237--302, 1996.


\end{thebibliography}

\end{document}